\documentclass[11pt]{amsart}

\usepackage[utf8]{inputenc}
\usepackage[T1]{fontenc}
\usepackage[french, english]{babel}

\usepackage{amsthm}
\usepackage{amsmath}
\usepackage{amssymb}
\usepackage{mathrsfs}
\usepackage{bbm}
\usepackage{graphicx}
\usepackage{listings}
\usepackage{mathrsfs}
\usepackage{enumerate}
\usepackage{pict2e}
\usepackage{stmaryrd}
\usepackage{mathtools}

\newcommand{\im}{\mbox{Im}}
\newcommand{\re}{\mbox{Re}}

\newcommand{\1}{\mathbf{1}}

\renewcommand{\epsilon}{\varepsilon}

\numberwithin{equation}{section}

\usepackage[all,cmtip]{xy}

\usepackage{multirow}

\usepackage{color}
\definecolor{marin}{rgb}   {0.,   0.1,   0.5} 
\definecolor{rouge}{rgb}   {0.8,   0.,   0.} 
\definecolor{sepia}{rgb}   {0.4,   0.25,   0.} 
\definecolor{mag}{rgb}   {0.3,   0,   0.3} 

\usepackage[colorlinks,citecolor=sepia,linkcolor=marin,urlcolor=mag ,bookmarksopen,bookmarksnumbered]{hyperref}

\newtheorem{theorem}{Theorem}[section]
\newtheorem{corollary}[theorem]{Corollary}
\newtheorem{lemma}[theorem]{Lemma}
\newtheorem{proposition}[theorem]{Proposition}
\newtheorem{definition}[theorem]{Definition}
\newtheorem{remark}[theorem]{Remark}
\newtheorem{comment}[theorem]{Comment}

\newcommand{\eps}{\varepsilon}

\begin{document}

\title[Exponential stability of solutions to the Schr\"odinger--Poisson equation
]{Exponential stability of solutions to the Schr\"odinger--Poisson equation}

\author{Joackim Bernier}

\address{\small{Nantes Universit\'e, CNRS, Laboratoire de Math\'ematiques Jean Leray, LMJL,
F-44000 Nantes, France
}}

\email{joackim.bernier@univ-nantes.fr}

\author{Nicolas Camps}

\address{\small{Nantes Universit\'e, CNRS, Laboratoire de Math\'ematiques Jean Leray, LMJL,
F-44000 Nantes, France
}}

\email{nicolas.camps@univ-nantes.fr}

\author{Beno\^it Gr\'ebert}

\address{\small{Nantes Universit\'e, CNRS, Laboratoire de Math\'ematiques Jean Leray, LMJL,
F-44000 Nantes, France
}}

\email{benoit.grebert@univ-nantes.fr}

\author{Zhiqiang Wang}

\address{\small{Chern Institute of Mathematics and LPMC, Nankai University, Tianjin 300071, China
}}

\email{zqwang@nankai.edu.cn}



\newcommand{\ma}{a}
\newcommand{\mb}{b}
\newcommand{\mB}{B}

\newcommand{\mi}{\mathbf{i}}
\newcommand{\md}{\mathrm{d}}
\newcommand{\mr}{\mathfrak r}
\newcommand{\bj}{\boldsymbol{j}}
\newcommand{\bh}{\mathbf{h}}
\newcommand{\bs}{\boldsymbol{S}}
\newcommand{\bk}{\mathbf{k}}
\newcommand{\oj}{\overline{j_0}}
\newcommand{\X}{\boldsymbol{X}}
\newcommand{\id}{\mathrm{I}}
\newcommand{\Id}{\mathrm{Id}}
\newcommand{\ad}{\mathsf{ad}}
\newcommand{\irr}{\mathsf{Irr}}
\newcommand{\sfc}{\mathsf C}
\newcommand{\mc}{\mathsf c}
\newcommand{\sis}{{\sigma}}
\newcommand{\lif}{{\ell^\infty}}
\newcommand{\lof}{{\ell_{\Gamma}^\infty}}
\newcommand{\bg}{\mathbb B}
\newcommand{\bn}{\mathcal B_\nu}
\newcommand{\lpi}[1]{l\in\bk^{-1}(#1)}
\newcommand{\lppi}[1]{l'\in\bk^{-1}(#1)}
\newcommand{\udl}[1]{\underline{~\text{provided }\sfc\ge #1~}}
\newcommand{\bN}{\mathfrak{N}}

\begin{abstract}

 We prove an exponential stability result for the small solutions of the Schr\"odinger--Poisson equation on the circle without exterior parameters in Gevrey class. More precisely we prove that for most of the initial data of Gevrey-norm smaller than $\varepsilon$ small enough, the solution of the Schr\"odinger--Poisson equation remains smaller than $2\varepsilon$ for time of order $ \exp \big(\alpha \frac{|\log\epsilon|^2}{\log|\log\epsilon|} \big)$. We stress out that this is the optimal time expected for PDEs as conjectured by Jean Bourgain in  \cite{Bou04b}.

\end{abstract} 
\maketitle

\setcounter{tocdepth}{1} 
\tableofcontents

\section{Introduction}

Hamiltonian PDEs have been very popular in many mathematical researchers over the past decades. Meanwhile, Birkhoff normal form theory is a key technique to study  the long-time behaviors of small amplitude solutions of nonlinear PDEs on bounded domains. In particular Birkhoff normal forms have been used  in non-resonant cases to prove stability over long periods of small and regular solutions \cite{Bou96,Bam03,BG06,BDGS07,G,GIP,Del12,FGL13,YZ14,BD17,FI19,BFG20b,BMP20,BFM22,BMM22}.  However, in all these works, external parameters were required to ensure  non-resonance  conditions on the linear frequencies.\\
 Without external parameters, Kuksin-P\"oschel \cite{KP1996} constructed a Cantor set of quasi-periodic solutions for resonant nonlinear Schr\"odinger equation (NLS) through KAM normal forms. The key was to use the nonlinearity to modulate the resonant linear frequencies and then to avoid  resonances. In short, the nonlinear term contributes to the stability of the solutions. 
Later  Bernier-Faou-Gr\'ebert \cite{KillBill}  exhibited rational normal forms to study Hamiltonian PDEs without external parameters and obtained  the stability over very long time ($\epsilon^{-r}$ for any large $r$ given) of the solutions to nonlinear Schr\"odinger equations and Schr\"odinger-Poisson equation (NLSP) on the circle (see also \cite{Bou00}). Then similar results were proved for Korteweg-De Vries equations and Benjamin Ono equations in \cite{KtA} and for Kirchhoff equations in \cite{BH22}.


In the present work we want to obtain exponential stability, i.e. stability for time of order $ \exp\big(\alpha \frac{|\log\epsilon|^2}{\log|\log\epsilon|} \big)$ for solutions of initial size $\epsilon$ in Gevrey class and still without using exterior parameters. As commented below this is the optimal time expected for PDEs (as conjectured by J. Bourgain, see  \cite{Bou04b}). Roughly speaking this exponential time is obtained by optimizing the procedure described in \cite{KillBill}. It  turns out that the NLSP case is easier than the NLS case. The reason is that  the cubic term of NLSP is enough to eliminate the higher orders (this simplifies the algebraic construction of the rational normal form) while the quintic term of NLS was used in \cite{KillBill}  to get rid of certain non-integrable resonant terms.
   The same result can be obtained for NLS, but at the cost of an extra technical effort and, in this work, we want to highlight the general method that can be applied in other cases, rather than fighting for the most general framework.\\
   We note that, very recently, J. Liu and D. Xiang have deposited a paper on arxiv \cite{LX23} which demonstrates a similar result for NLS still based on \cite{KillBill} but they do not obtain the optimal time but rather time of order $\exp(\alpha|\log\epsilon|^{1+\beta})$ for any $\beta<1$. \\
So in this work we consider the following  Schr\"odinger-Poisson equation on the torus $\mathbb T=\mathbb R/2\pi \mathbb Z$
\begin{equation}\tag{NLSP}\label{sp}
\begin{cases}
\mi\partial_tz-\Delta z+Wz=0,\\
\displaystyle
-\Delta W=|z|^2-\frac1{2\pi}\int_{\mathbb T}|z|^2\,\md x,
\end{cases}\quad
z=z(t,x),~(t,x)\in\mathbb R\times \mathbb T,
\end{equation}

As usual, we identify each function $z\in L^1(\mathbb{T};\mathbb{C})$ with the sequence of its Fourier coefficients $(z_a)_{a\in\mathbb{Z}} \in \mathbb{C}^\mathbb{Z}$ defined by
$$
z_a=\frac1{2\pi}\int_{\mathbb T}z(x)e^{-\mi ax}\,\md x.
$$
Given $\sigma>0,\theta\in(0,1)$, we  define the following  Gevrey space (the parameter $\theta$ and $\sigma$ are fixed once and for all)
\begin{equation}\label{gevrey}
	\mathcal G:=\left\{z \in \mathbb{C}^\mathbb{Z} ~|~||z||_{\sis}:=2\sum_{a\in\mathbb Z}e^{\sigma|a|^\theta}|z_a|<\infty\right\}.
\end{equation}
Note that, since $\sigma$ and $\theta$ are positive, this is a space of smooth functions, i.e. $\mathcal G \subset C^\infty(\mathbb{T};\mathbb{C})$. Moreover, since it is based on $\ell^1$, it is a Banach algebra.

By the second equation one has $W=V\star|z|^2$ where $V$ is the Green function of the operator $-\Delta$ with zero average on the torus. Therefore, we obtain the associated Hamiltonian 
\begin{equation}\label{Hamu}
	H(z)=\frac{1}{2\pi}\int_{\mathbb T}\left(|\nabla z|^2+\frac12(V\star|z|^2)|z|^2\right)\md  x.
\end{equation}
Note that it rewrites
\begin{equation}\label{Hamz}
H(z)=\sum_{a\in\mathbb Z}a^2|z_a|^2+\sum_{\substack{a_1+a_2=b_1+b_2\\a_1\ne b_1}}\frac{1}{2(a_1-b_1)^2}z_{a_1}z_{a_2}\overline{z_{b_1}}\overline{z_{b_2}}=:L_2(z)+P_4(z).
\end{equation}

\subsection{Main result}

We are going to consider solution of \eqref{sp} generated by initial data in a family of open subsets $\Theta_\varepsilon$ of $\mathcal{G}$ surrounding the origin (i.e. $0\in \overline{\cup_{\varepsilon >0 }\Theta_\varepsilon}$).
They are explicitly defined by \eqref{eq:big_set}, but their precise definition requires lots of notations and is not really necessary to state and to understand our main result. 

First, we are going to state a theorem, describing, for very long times, the dynamics of the solutions of \eqref{sp} whose initial data are in $\Theta_\varepsilon \cap \overline{\mathbb{B}(0, \varepsilon)}$ where $\mathbb{B}(0, \varepsilon)$ denotes the ball of radius $ \varepsilon$ in $\mathcal{G}$ and centered at the origin. Then we will state three propositions to explain why it allows to control the dynamics of most of the solutions in a sense that we specify carefully.

\begin{theorem} \label{thm:main} There exists $\epsilon_{\ast}(\sigma,\theta)>0$ such that for all $\epsilon\in(0,\epsilon_{\ast})$, for all $z^{(0)}\in \Theta_\varepsilon$ of size $\| z^{(0)} \|_\sigma \leq  \varepsilon$, the  local solution with data $z(0)=z^{(0)}$ to \eqref{sp} exists in $C([-T_\varepsilon,T_\varepsilon] ; \mathcal{G})$ with
\begin{equation}\label{Teps}
T_{\varepsilon} := \varepsilon^{-r_\eps}\quad \text{with}\quad r_\eps=\frac{\min(\sigma,1)\theta(1-\theta)}{1500}\frac{\log \varepsilon^{-1}}{\log\log \varepsilon^{-1}},
\end{equation} 
and this solution satisfies for $|t| \leq T_{\varepsilon}$
\[
\|z(t)\|_{\sigma}\leq2 \| z(0) \|_{\sis}\,,\quad \sum_{a\in\mathbb{Z}}e^{\sigma|a|^{\theta}}| |z_a(t)|^2-|z_a(0)|^2|^{\frac{1}{2}}\leq \| z(0) \|_{\sis}^{\frac{3}{2}}\,.
\] 
\end{theorem}

Now, it remains to prove that the sets $\Theta_\varepsilon$ are large and enjoy good properties. First, we focus on the geometric properties of $\Theta_\varepsilon$. To claim this first proposition, we need to introduce the projection $\Pi_M : \mathcal{G} \to \mathcal{G}$ defined by
\begin{equation}
\label{PiM}
\Pi_{M}z = \sum_{|a|\leq M}z_{a}e^{\mi ax}\,
\end{equation}
then we naturally define $\mathcal{G}_{M} := \Pi_{M}\mathcal{G}$ and $\mathbb{B}_M(0,\varepsilon) =\Pi_{M} \mathbb{B}(0,\varepsilon). $

\begin{proposition}
\label{prop:geom}
For all $\varepsilon \in (0,\epsilon_{\ast})$, $\Theta_\varepsilon$ is an open subset of $\mathcal{G}$, which is a right cylinder of direction  $(\mathrm{Id}_\mathcal{G} - \Pi_{M_\varepsilon}) \mathcal{G}$ where $M_\varepsilon = (\log \varepsilon^{-1})^{1+\frac{4}{\theta}}$, i.e.
$$
z\in \Theta_\varepsilon \quad \iff \quad \Pi_{M_\varepsilon} z \in \Theta_\varepsilon
$$
and that is
invariant by translation of the angles, i.e.
$$
\sum_{a\in\mathbb Z}z_ae^{\mi ax}\in \Theta_\varepsilon \Longleftrightarrow\sum_{a\in\mathbb Z}|z_a|e^{\mi ax}\in \Theta_\varepsilon.
$$ 
\end{proposition}

\begin{proposition} \label{prop:meas} For all $\varepsilon \in (0,\epsilon_{\ast})$, setting $M_\varepsilon = (\log \varepsilon^{-1})^{1+\frac{4}{\theta}}$, we have that
\begin{equation}
\label{eq:meas}
\mathrm{meas}(\Theta_\varepsilon \cap \mathbb{B}_{M_\varepsilon}(0,\varepsilon))\geq(1-\varepsilon^{\frac{1}{6}})\mathrm{meas}(\mathbb{B}_{M_\varepsilon}(0,\varepsilon))
\end{equation}
where $\mathrm{meas}$ denotes any Lebesgue measure\footnote{Note that, thanks to the change of variable formula, the estimate \eqref{eq:meas} does not depend on the Lebesgue measure we choose on $\mathcal{G}_{M_\varepsilon}$.} on $\mathcal{G}_{M_\varepsilon}$.
\end{proposition}
\begin{comment} We stress out that the first property of Proposition \ref{prop:geom} says that the belonging of $z$ to the set $\Theta_\varepsilon$ does not depend on the high Fourier modes of $z$ but only on those with indices smaller than $M_\varepsilon$. Note that, from a numerical point of view, $M_\varepsilon$ is actually very small in terms of $\varepsilon^{-1}$. Furthermore Proposition \ref{prop:meas} asserts that, reduced to this finite number of modes, the set $\Theta_\varepsilon$ is asymptotically of full Lebesgue measure. \end{comment} 
Now we state a probability result on $\Theta_\varepsilon$.
\begin{proposition} \label{prop:proba} Let $Y$ be a random function in $\mathcal{G}$, whose Fourier coefficients $Y_a$ are real, independent and uniformly distributed in $(0, \langle a \rangle^{-2} e^{-\sigma |a|^\theta})$, and let $Z^{(0)} = Y / \| Y\|_{\sigma}$ be the projection of $Y$ on the unit sphere. Then, provided that $\varepsilon_0$ is small enough, we have
$$
\mathbb{P}(\forall 0<\varepsilon \leq \varepsilon_0, \ \varepsilon Z^{(0)} \in \Theta_\varepsilon) \geq 1 - \varepsilon_0^{1/12}.
$$
\end{proposition}
As a consequence we deduce that $\bigcup_{\varepsilon >0 }\Theta_\varepsilon \cap \overline{ \mathbb{B}(0, \varepsilon)}$ (i.e. the set of the "good" initial data) is "almost surely star shaped":
\begin{corollary}
Almost surely, there exists $\varepsilon_0>0$ such that for all $\varepsilon \in (0,\varepsilon_0)$, $\varepsilon Z^{(0)} \in \Theta_\varepsilon$.
\end{corollary}

\subsection{Related results}

The stability of small solutions over exponential times, in a stronger or weaker sense depending on the article, has already been demonstrated using external parameters. For the NLS equation on $\mathbb T^d$ in analytic regularity, Faou-Gr\'ebert in \cite{FG13} obtain a time $T_\epsilon=\exp(\alpha|\log\epsilon|^{1+\beta})$ for any $\beta<1$. For the NLS equation on $\mathbb T$ with Gevrey regularity, Biasco-Massetti-Procesi in \cite{BMP20} obtain, still using external parameters, a time of order $\exp (\alpha| \log \epsilon|^{1+\theta/4})$ (see also \cite{CMWZ22} for derivative NLS and \cite{CCMW22} for a result in class of regularity between $C^\infty$ and Gevrey). As already noticed above,  Liu and Xiang (see \cite{LX23}), for NLS without parameters in Gevrey regularity, obtained very recently a time of order $T_\epsilon=\exp(\alpha|\log\epsilon|^{1+\beta})$ for any $\beta<1$.  \\
In finite dimension $n$, the standard Nekhoroshev result \cite{Nek77} controls the dynamics over times of order $\exp \big(\alpha \epsilon^{-1/(\tau +1)} \big)$ for some $\alpha>0$ and $\tau>n+1$   (see for instance \cite{ben85, GG85, P93}) which is, of course, much better than our $T_\epsilon$. Nevertheless, clearly this standard result does not extend to the infinite dimensional context: when $n\to +\infty$ we get $\tau\to+\infty$. Actually this kind of exponential times $ \exp \big(\alpha \frac{|\log\epsilon|^2}{\log|\log\epsilon|} \big)$ were obtained by Benettin-Fr\"ohlich-Giorgilli in \cite{BFG88} for a Hamiltonian system with infinitely many degrees of freedom but with finite-range couplings. We also notice that this time was suggested by Bourgain as the optimal time that we could obtain in an analytical context (see eq. (2.14) in   \cite{Bou04b}). It's somewhat surprising that we get the same time for Gevrey regularities. We are also convinced that this optimal time is not model-dependent, and that we could adapt our method to NLS and probably also to KdV. As usual the method is based on an optimization with respect to $\eps$ of the order $r_\eps$ of the normal form  (in this case a rational normal form) that we perform. An important point in our approach (see section \ref{sec:outline} ) is that we perform the optimization  on a reduced system involving only low frequency modes. The initial model mainly plays a role in controlling the high modes  (and their influence on low modes) and not directly on the optimal time. Note in passing that this method could also be applied to simpler cases, in particular to NLS with convolutive potentials acting as  external parameters. By the way, this is also the time that two of us obtained  in \cite{BG22} for special choices of convolution  potentials in NLS on $\mathbb  T^d$ in $H^s$ for any $s>d/2$.\\
Finally, we'd obviously like to obtain a result that is valid for any small initial condition and not just for almost any small initial data. In finite dimension, this is achieved by the geometric Nekhoroshev method (in contrats to the analytic Nekhoroshev method used in this work) based on convexity arguments. This approach is very tempting in the case of resonant PDEs, but unfortunately the \cite{BouKal} article is not very optimistic about its success in high dimension, let alone infinite dimension (see however \cite{Bam99}).

\subsection{Proof strategy and  paper outline}\label{sec:outline} In Section~\ref{sec:ham} we introduce Hamiltonian setting and we define a class of polynomial Hamiltonians for which we describe the associated Hamiltonian flows  (Lemma~\ref{flow}).  Then, in Section~\ref{sec:bnf}, we are in position to state and prove a resonant normal form result, Theorem~\ref{Bnormalform}. Notice that in formula \eqref{decomp1} NLSP appears as a perturbation of $L_2+L_4$ which is nonlinear but integrable. We will use this fact to modulate the frequencies.\\
 In Section~\ref{sec:high}, just using the resonant normal form, we control the dynamics of the high frequencies over exponential times, under some a priori control on the norm of the solution. Then it remains to control the dynamics of the low frequencies, says $u_{\leq M}:=(u_k)_{|k|\leq M}$. Naturally $u_{\leq M}$ satisfies a {\it finite} dimensional system governed by the truncated Hamiltonian. Thus in Section~\ref{sec:rational1} and Section~\ref{sec:rational2} we perform a  rational normal form for this truncated Hamiltonian, Theorem~\ref{Rnormalform}, in a {\it finite} dimensional setting. This way of proceeding (i.e. reducing first the problem in finite dimension) has the enormous advantage (comparing with \cite{KillBill}) of leaving aside many of  tricky problems of convergence and in particular the problem of  the existence of derivatives with respect to time, which usually require a somewhat convoluted passage through density. We emphasize that in Section~\ref{sec:rational1} and Section~\ref{sec:rational2} we follow carefully the dependency of all the parameters with respect to the degree $r$ of the normal form. This was not necessary in  \cite{KillBill} since this order $r$ was fixed from the very beginning, but now we want to optimize it. \\
Finally in Section~\ref{sec:dyn} we put all together and prove the main Theorem by optimizing the parameters (in particular the degree $r$) introduced in the previous sections in term of $\varepsilon$, the size of the solution.  The optimal degree is $r_\varepsilon= c(\sigma,\theta)\frac{\log \varepsilon^{-1}}{\log\log \varepsilon^{-1}}$ and thus we naturally obtained the optimal time \eqref{Teps} in Theorem~\ref{thm:main}.   The optimized truncation parameter is $M_\varepsilon = (\log \varepsilon^{-1})^{1+\frac{4}{\theta}}$ and since the construction of the set $\Theta_\varepsilon$ only depends on the rational normal form construction of Section~\ref{sec:rational2}, we obtain easily Proposition~\ref{prop:geom}.

\subsection{Notations}
In all this paper we set $\mathbb{U}_2 := \{-1,1\}$. $L^2(\mathbb{T};\mathbb{C})$ is equipped with its real scalar product
\begin{equation}
\label{def:scalar_prod}
(u,v)_{L^2} = \re \sum_{a\in \mathbb{Z}} u_a \overline{v_a} = \frac1{2\pi} \re \int_{\mathbb{T}} u(x) \overline{v(x)}\mathrm{d}x
\end{equation}
As usual, we always identify sequences $z\in \mathbb{C}^{\mathbb{Z}}$ with sequences in $\mathbb{C}^{\mathbb{U}_2 \times \mathbb{Z}}$ by using the convention 
$$
z_{-1,a} := \overline{z_a} \quad \mathrm{and} \quad z_{1,a} := z_a, \quad \forall a \in \mathbb{Z}.
$$
Given $j=(\delta,a) \in \mathbb{U}_2 \times \mathbb{Z}$, we define its conjugate by
\begin{equation}\label{jbar}
\overline{j} = \overline{(\delta,a)} = (-\delta,a)
\end{equation}
and we set
$$
\delta(j) = \delta \quad \mathrm{and} \quad a(j) =a.
$$
Furthermore, we use the following unusual conventions
$$
|j|:=|(\delta,a)|:=|a| \quad \mathrm{and} \quad \langle j \rangle:=\langle (\delta,a) \rangle:=\max(1,|a|). 
$$
Given a set $E$ and finite sequence $v \in E^n$ of elements of $E$  of size $n\geq 0$ for some $n\in \mathbb{N}$ then $\# v$ denotes the number of elements of $v$ that is $\# v := n$.

\subsection{Acknowledgments}
During the preparation of this work the authors benefited from the support of the Centre Henri Lebesgue ANR-11-LABX-0020-0 and J.B. and N.C. were also supported by the region "Pays de la Loire" through the project "MasCan"
and Z.W. was also supported by CSC grant (202206100095) and Nankai Zhide Foundation.
The authors were partially supported by the ANR project KEN ANR-22-CE40-0016.

\section{Functional setting and polynomials}
\label{sec:ham}
\subsection{Differential calculus and symplectic structure} In this paper, we always consider $\mathcal{G}$ as a real normed vector space (and not as a complex one). Note that, in particular, it implies that $z \mapsto \overline{z}$ is smooth. Being given a real Banach space $E$ and a map $f \in C^1(\mathcal{G};E)$ we define its partial derivatives with respect to $z_a$ and $\overline{z_a}$ by
$$
2\partial_{z_a} f :=  \partial_{\re z_a} f - \partial_{\im z_a} f \quad \mathrm{and} \quad 2\partial_{\overline{z_a}} f :=  \partial_{\re z_a} f + \partial_{\im z_a} f.
$$
Moreover, we set
$$
\partial_{I_a} f := \partial_{z_a}\partial_{\overline{z_a}} f. 
$$
Given an open set $\mathcal{O} \subset \mathcal{G}$ and a Hamiltonian $P \in  C^1 (\mathcal{O};\mathbb{C})$ we define its \emph{Hamiltonian vector field} $\X_P : \mathcal{G} \to \mathbb{C}^\mathbb{Z}$  by
$$
(\X_P(z))_a := \mi\partial_{\overline{z_a}} P(z).
$$
Note that if $P$ is real valued it implies that
$$
\X_P =\frac{\mi}2 \nabla P\,,
$$
where the gradient is associated the $L^2$ scalar product \eqref{def:scalar_prod}. Moreover, if $P$ is real valued, we have the useful relation
$$
(\X_P(z))_{j} = \delta(j) \mi \frac{\partial P}{\partial z_{\bar j}}(z), \quad \forall j\in \mathbb{U}_2 \times \mathbb{Z}.
$$
Given $P,Q \in  C^1 (\mathcal{O};\mathbb{C})$ two Hamiltonians such that $\X_P$ or $\X_Q$ takes values in $\mathcal{G}$, we define their \emph{Poisson bracket} by\footnote{Note that, it makes sense because by duality we always have $\X_P(z),\X_Q(z) \in \mathcal{G}'$.}
$$
\{P,Q\}:=\mi\sum_{j\in\mathbb U_2\times\mathbb Z}\delta(j)\frac{\partial P}{\partial z_j}\frac{\partial Q}{\partial z_{\bar j}}.
$$
Note that if $P$ and $Q$ are real valued then, with these conventions, we have
\[
\{P,Q\} =-\frac12 (\mi \nabla P, \nabla Q)_{L^2}\,.
\]
Given an open set $\mathcal{O} \subset \mathcal{G}$, a map $\Psi : \mathcal{G} \to  \mathcal{G} $ is said \emph{symplectic} if for all $z\in \mathcal{O}$,
\[
\forall v,w \in \mathcal{G}, \quad (\mi \mathrm{d} \Psi(z) (v)  ,\mathrm{d} \Psi(z) (w))_{L^2} = (\mi v,w)_{L^2}.
\]

\subsection{Multi-indices} We denote by $\mathcal{J}$ the following set of the multi-indices
\[
\mathcal{J} = \bigcup_{m \geq 0} \mathcal{J}_m \quad \mathrm{where} \quad \mathcal{J}_m = \left(\mathbb U_2\times\mathbb Z\right)^{2m}.
\]
We extended the conjugation to multi-indices by setting
\[
 \bar\bj:=(\overline{j_1},\cdots,\overline{j_{2m}}).
\]
 We denote by $\mu_i(\bj)$ the $i^{\rm th}$ largest element among the collection $\{\langle j_\beta\rangle|\,\beta=1,\cdots,2m\}$, i.e. 
\[
\mu_1(\bj)\ge\mu_2(\bj)\ge\cdots\ge\mu_{2m}(\bj).
\]
 Given a sequence $z\in \mathbb{C}^\mathbb{Z}$, we define $z_{\bj} \in \mathbb{C}$ by
 $$
 z_{\bj} :=  z_{j_1} \cdots   z_{j_{2m}}.
 $$
 By abuse of notation, $z_{\bj}$ will also refer to the associated monomial.\\
 For $m\in\mathbb N$ we define three nested index sets satisfying zero momentum conditions of increasing order as follows:
\begin{gather*}
	\mathcal Z_m=\left\{\bj=(\delta,a)\in \mathcal{J}_m\,|\,\sum_{\alpha=1}^{2m}\delta_\alpha=0\right\}, \\
	\mathcal M_m=\left\{\bj=(\delta,a)\in\mathcal Z_m\,|\,\sum_{\alpha=1}^{2m}\delta_\alpha a_\alpha=0\right\},\\
	\mathcal R_m=\left\{\bj=(\delta,a)\in\mathcal M_m\,|\,\sum_{\alpha=1}^{2m}\delta_\alpha a_\alpha^2=0\right\}.
\end{gather*}
We define
$$
\mathcal Z=\bigcup_{m\ge1}\mathcal Z_m, \quad \mathcal M=\bigcup_{m\ge1}\mathcal M_m \quad \mathrm{and} \quad \mathcal R=\bigcup_{m\ge1}\mathcal R_m.
$$
For future use, we note the following result
\begin{lemma}\label{lem:mu1mu3}
Let $\bj\in\mathcal R_m$ with $m\geq2$ and let $\ell\in\mathbb Z$.
 Then either $\{ I_\ell, z_{\bj}\}=0$
  or  $\mu_3(\bj)\geq \big(\frac{\langle\ell\rangle}{m}\big)^{\frac12}$.
\end{lemma}
\proof
Let $\bj\in\mathcal R_m$. Without lost of generality we can assume that
$$|j_1|\geq |j_2|\geq \cdots \geq |j_{2m}|.$$
 Using that $\bj\in\mathcal R_m$ we have
$$|j_1|^2\pm |j_2|^2\pm\cdots \pm|j_{2m}|^2=0$$
from which we deduce that
\begin{equation}\label{mu3}(2m-2)|j_3|^2\geq |j_1|^2\pm |j_2|^2.\end{equation}
Now if  $\{ I_\ell, z_{\bj}\}\neq0$,  then $(\pm1,\ell)\in\{j_1,\cdots,j_{2m}\}$. If $\langle\ell\rangle\leq \mu_3(\bj)$  then $\mu_3(\bj)\geq \big(\frac{\langle\ell\rangle}{m}\big)^{\frac12}$ and the lemma is proved. So we can assume $\langle\ell\rangle> \mu_3(\bj)$ which in turn leads to $\langle\ell\rangle=\langle j_1\rangle$ or $\langle\ell\rangle=\langle j_2\rangle$.  But then
we cannot have\footnote{The notation $\bar j$ is defined in \eqref{jbar}.} $j_2=\bar j_1$ because in this case we would have $\{ I_\ell, z_{\bj}\}=0$ (recall that  $\langle\ell\rangle> \mu_3(\bj)$). 
Now if $j_2\neq\bar j_1$, we have $|j_1|^2\pm |j_2|^2\geq 2|j_1|-1$ and thus \eqref{mu3} leads to the result.
\endproof

\subsection{Polynomials}

\begin{definition} Given $m\geq 1$, we denote by $\mathcal H_m$ the space of the homogeneous polynomial of degree $2m$ of the form
\[
P(z)=\sum_{\bj\in\mathcal M_m}P_{\bj}z_{\bj}
\]
where the coefficients $P_{\bj} \in \mathbb{C}$ satisfy 
\begin{itemize}
\item the reality condition: $\forall \bj \in \mathcal M_m, \quad \overline{P_{\bj}}   = P_{\overline{\bj}} $
\item the symmetry condition: $\forall \bj \in \mathcal M_m, \forall \varphi \in \mathfrak{S}_{2m}, \quad P_{\bj} = P_{\varphi \bj}$
\item the bound:
\[
||P||_{\lif}:=\sup_{\bj\in\mathcal M_m}|P_{\bj}|< \infty.
\]
\end{itemize}
\end{definition}

\begin{remark} \begin{itemize} \item 
The bound ensures that $P$ is locally bounded on $\ell^1$ (by the Young inequality). A fortiori it is locally bounded on $\mathcal{G}$. Note that since it is a polynomial it implies that it is smooth on $\mathcal{G}$. 
\item The reality condition ensures that $P$ is real valued. The symmetry condition ensures that the coefficients $P_{\bj}$ are unique and, as a consequence, that the norm $||P||_{\lif}$ is well defined.
\end{itemize}
\end{remark}

\begin{definition} We denote by  $\mathcal{H}_m^\mathcal R$ the subspace of $\mathcal H_m$ made of resonant polynomials 
 \[
\mathcal H_m^\mathcal R=\Big\{P\in\mathcal H_m\,|\,P_{\bj}\ne0\Rightarrow\bj\in\mathcal R \Big\}.
\]
\end{definition}

\begin{lemma}[Vector field]\label{vectorfield}
Let $m\ge2$ and $P\in\mathcal H_m$, then the associated Hamiltonian vector field is smooth and locally Lipschitz. More precisely, we have the estimates
\begin{gather*}
||\X_P(z)||_{\sis}\le { 2} m||P||_{\lif}||z||_{\sis}^{2m-1},\\
||\md\X_P(z)(w)||_{\sis}
\le { 4}  m^2||P||_{\lif}||z||_{\sis}^{2(m-1)}||w||_{\sis}.
\end{gather*}
\end{lemma}
\begin{proof}
By symmetry of the coefficients of $P$, we have
$$
\big(\X_P(z)\big)_{j_0} = \mi\delta(j_0) 2m\sum_{\substack{\bj\in\mathcal M_m\\ j_1=\oj}} P_{\bj} \prod_{\alpha=2}^{2m} z_{j_\alpha},
$$
and so, using successively the zero momentum condition, the Young convolution inequality and the Minkowski inequality (because $\theta \in (0,1)$), we get
\begin{equation*}
\begin{split}
||\X_P(z)||_\sis &= \sum_{j_0\in\mathbb U_2\times\mathbb Z}e^{\sigma|j_0|^\theta} |\big(\X_P(z)\big)_{j_0}| 
\leq 2m \|P\|_{\ell^\infty} \sum_{j_0\in\mathbb U_2\times\mathbb Z}e^{\sigma|j_0|^\theta} \sum_{\substack{\bj\in\mathcal M_m\\ j_1=\oj}}  \prod_{\alpha=2}^{2m} |z_{j_\alpha}| \\
&= 2m \|P\|_{\ell^\infty} \sum_{\bj\in\mathcal M_m}e^{\sigma|\delta_1 j_1 + \cdots + \delta_{2m}j_{2m}|^\theta} \prod_{\alpha=2}^{2m} |z_{j_\alpha}| 
\leq  2m \|P\|_{\ell^\infty}  \sum_{\bj\in\mathcal M_m}  \prod_{\alpha=2}^{2m} e^{\sigma|j_\alpha|^\theta}  |z_{j_\alpha}| \\
&\leq 2m \|P\|_{\ell^\infty} ||z||_{\sis}^{2(m-1)}.
\end{split}
\end{equation*}

The estimate for the derivative of the vector field, is just a direct refinement of the previous estimate on  $\X_P$ using its multi-linearity.

\end{proof}

\begin{lemma}[Poisson bracket]\label{poissonbracket}
Let $P\in\mathcal H_{m}$ and $P'\in\mathcal H_{m'}$, then there exists $P''\in\mathcal H_{m''}$ with $m''=m+m'-1$ such that  
\[
P''=\{P,P'\},
\]
satisfying
\[
||P''||_{\lif}\le { 4} mm'||P||_{\lif}||P'||_{\lif}.
\]
\end{lemma}
\begin{proof}
The proof of this Lemma is very classical. We refer for example to Proposition 2.5 of \cite{BG22}. Actually, this lemma is also a direct corollary of Lemma \ref{Rpoissonbracket} below (it corresponds to the case $\mathfrak{n}_Q =\mathfrak{n}_{Q'}=0 $).
\end{proof}
The main result of this section concerns the Hamiltonian flows generated by polynomials:
\begin{lemma}[Local flow]\label{flow}
Let $m\ge2$ and $\bs\in\mathcal H_{m}$, we define
\[
\varepsilon_1:=\frac14\left( { 4} m||\bs||_{\lif}\right)^{-\frac1{2(m-1)}}.
\]
The flow $\Phi_{\bs}^t(\cdot)$ defines  a smooth map from $[-1,1]\times\bg(0,\varepsilon_1)$ into $\mathcal G$ such that
\begin{equation}\label{PhiS1}
	\begin{cases}
\partial_t\Phi_{\bs}^t(z)=\X_{\bs}\left(\Phi_{\bs}^t(z)\right),\\
\Phi_{\bs}^0(z)=z.
\end{cases}
\end{equation}
Moreover, given  $t\in[-1,1]$, one has
\\
(i) $\Phi_{\bs}^t$ is symplectic,\\
(ii) $\Phi_{\bs}^t$ is locally invertible:
\[
\Phi_{\bs}^{-t}\circ\Phi_{\bs}^t(z)=z,\quad\text{if}\quad
z\in\bg(0,\varepsilon_1) \text{ and }\ \Phi_{\bs}^t(z)\in\bg(0,\varepsilon_1),
\]
(iii) $\Phi_{\bs}^t$ is close to identity:
\[
||\Phi_{\bs}^t(z)-z||_{\sis}\le2^{3m}||\bs||_\lif||z||^{2m-1}_{\sis} \quad \text{for } z\in \bg(0,\varepsilon_1),
\]
(iv) $\Phi_{\bs}^t$ is locally Lipschitz:
\[
||\md\Phi_{\bs}^t(z)(w)||_{\sis}\le 2||w||_{\sis}\quad \text{for } z\in \bg(0,\varepsilon_1)\text{ and } w\in\mathcal G.
\]
\end{lemma}
\begin{proof}
Thanks to Lemma \ref{vectorfield}, we know that the associated Hamiltonian vector fields $\X_{\bs}$ is locally-Lipschitz, which ensures that the smoothness of the flow $\Phi_{\bs}^t$  by the Cauchy-Lipschitz Theorem. The only thing we need to check is that the solution exists for $|t|\le1$. Without loss of generality we only consider positive times. Let $T>0$ and $y=y(t)\in C^1([0,T);\mathcal G)$ be the maximal solution of the Cauchy problem 
\[
\begin{cases}
\dot y=\X_{\bs}(y),\\
y(0)=z\in\bg(0,\varepsilon_1).
\end{cases}
\]
It is enough to prove that if $0\le t<T$ and $t\le1$ then $||y(t)||_\sis\le 2||z||_\sis<2\varepsilon_1$. We set $E_0=[0,T)\cap[0,1]$ 
\[
E_1=\Big\{t\in E_0\,|\,\forall\,\tau\in[0,t],||y(\tau)||_\sis\le 2||z||_\sis\Big\}.
\]
Obviously, the set $E_1$ is non-empty, connected and closed in $E_0$ by the continuity of the flow. Moreover, by Lemma \ref{vectorfield} one has for $t\in E_0$
\begin{align*}
||y(t)-z||_\sis&\le\int_0^t||\X_{\bs}(y(\tau))||_\sis\md\tau\le\int_0^t { 2}m||\bs||_\lif||y(\tau)||_\sis^{2m-1}\md\tau\le { 2}m2^{2m-1}||\bs||_\lif||z||_\sis^{2m-1}\label{close2m}\\
&\le \left(\frac{2||z||_\sis}{4\varepsilon_1}\right)^{2(m-1)}||z||_\sis\le 2^{-2m+2}||z||_\sis<||z||_\sis.\notag
\end{align*}
The last inequality implies that $E_1$ is open, which is therefore the universal set, i.e. $E_1=E_0$. Now we have checked the existence of $\Phi_{\bs}$, which ensures the property (ii) since the vector field is local-Lipschitz. Also, the property (i) is canonical as $\Phi_{\bs}^t$ is  a Hamiltonian flow. Moreover, the previous estimate directly implies the property (iii). 

Finally, we focus on the property (iv). Observe that for $z\in\bg(0,\varepsilon_1)$ and $w\in\mathcal G$
\begin{equation}\label{dPhizw1}
	\partial_t\md\Phi_{\bs}^t(z)(w)=\md_y\X_{\bs}(y)\big(\md\Phi_{\bs}^t(z)(w)\big),
\end{equation}
where $y=\Phi_{\bs}^t(z)$. By Lemma \ref{vectorfield} one has
\[
||\md_y\X_{\bs}(y)\big(\md\Phi_{\bs}^t(z)(w)\big)||_\sis\le {4}m^2||\bs||_\lif||y||_\sis^{2(m-1)}||\md\Phi_{\bs}^t(z)(w)||_\sis.
\]
Therefore using \eqref{dPhizw1} and $\md\Phi_{\bs}^0(z)(w)=\md \, \mathrm{Id} (z) (w)=w$, we get
\begin{align*}
||\md\Phi_{\bs}^t(z)(w)||_\sis&\le||w||_\sis+\int_0^t m\left(\frac{2||z||_\sis}{4\varepsilon_1}\right)^{2(m-1)}||\md\Phi_{\bs}^\tau(z)(w)||_\sis\md\tau\\
&\le||w||_\sis+\int_0^t\frac12||\md\Phi_{\bs}^\tau(z)(w)||_\sis\md\tau.
\end{align*}
By Gr\"onwall's inequality the last estimate leads to
\[
||\md\Phi_{\bs}^t(z)(w)||_\sis\le||w||_\sis\exp\left(\int_0^t\frac12\md\tau\right)\le\sqrt{e}||w||_\sis\le2||w||_\sis.
\]
\end{proof}

\section{Resonant normal form}\label{sec:bnf}
Recall that   $L_2$ given by \eqref{Hamz}, is the quadratic Hamiltonian corresponding to the linear part of \eqref{sp}.
We define for $\bj=(\delta_\beta,a_\beta)_{\beta=1}^{2m}\in(\mathbb U_2\times\mathbb Z)^{2m}$
\begin{equation}\label{smalldivisor-l2}
	\Delta_{\bj}=\sum_{\beta=1}^{2m}\delta_\beta a_\beta^2.
\end{equation}
With this notation, we have for $\bj\in\mathcal Z$
\begin{gather}
\{L_2,z_{\bj}\}=\mi\Delta_{\bj}z_{\bj}.\label{L2}
\end{gather}
We note that  $|\Delta_{\bj}|\ge1$ except  $\bj\in\mathcal R$ for which $\Delta_{\bj}=0$. 
In this section we eliminate all the monomials $z_{\bj}$ of \eqref{Hamz} with $\bj\notin\mathcal R$. This is relatively easy since, in view of \eqref{L2}, the corresponding small denominators are larger than 1. Nevertheless we have to carefully follow the dependence with respect to all the parameters in view of a future optimization.

\subsection{The result}
\begin{proposition}\label{Bnormalform}
Let $H$ be the Hamiltonian given by \eqref{Hamz}. There exists a constant $\sfc \ge2^8$ such that for all $r\ge2$, setting 
\begin{equation}
\label{eq:rho}
\rho=\frac1{8\sfc r}\,,
\end{equation}
there are two $C^\infty$ symplectic maps $\phi_{(0)}$ and $\phi_{(1)}$ making the following diagram to commute 
\begin{equation}
\label{diag1}
\xymatrixcolsep{5pc} \xymatrix{ \mathbb{B}(0, \rho )  \ar[r]^{ \phi_{(0)} }
 \ar@/_1pc/[rr]_{ \mathrm{id}_{\mathcal{G}} } & \mathbb{B}(0, 2\rho )   \ar[r]^{ \hspace{1cm}  \phi_{(1)}}  & 
\mathcal{G}}
\end{equation}
such that $H\circ\phi_{(1)}$ admits the decomposition on $\bg(0,2\rho)$
\begin{equation}\label{decomp1}
	H\circ\phi_{(1)}(z):=L_2+ L_{4}+\sum_{m=3}^{r}K_{2m}+R\,,
\end{equation}
where $L_4$ is the integrable polynomial explicitly given by
\begin{equation}
\label{eq:def_L4}
 L_4 = \sum_{a_1\neq a_2}  \frac{1}{2(a_1-a_2)^2}|z_{a_1}|^2|z_{a_2}|^2,
\end{equation}
  $ K_{2m}\in\mathcal H_{m}^{\mathcal R}$ is a resonant homogeneous polynomial of degree $2m$ satisfying
\begin{equation}
\label{eq:b-bnf}
||K_{2m}||_{\lif}\le \sfc^{2m-3}m^{2(m-2)},\quad\forall\,2\le m\le r\,,
\end{equation}
and $R:\bg(0,2\rho)\mapsto\mathbb R$ is a $C^\infty$ function which is a remainder term of order $2r+2$:
\begin{equation}\label{XR1}
||\X_R(z)||_{\sis}\le 
\sfc^{4r-1}r^{2r}||z||_{\sis}^{2r+1}.
\end{equation}
Furthermore, the map $\phi_{(\iota)}:\bg(0,\rho+\iota\rho)\mapsto\mathcal G$ is close to the identity:
\begin{equation}
\label{eq:Bnf-close}
||\phi_{(\iota)}(z)-z||_{\sis}\le \sfc^2||z||^3_{\sis},\quad\iota=0,1.
\end{equation}
\end{proposition}
\begin{proof} we split the proof in six steps.\\
\underline{$\vartriangleright$\textsf{(1) Induction.}} We are going to prove by induction that for $\mr\in\llbracket 1,r \rrbracket$, setting 
\begin{equation}\label{rho}
	\rho=\frac1{8\sfc\mr},
\end{equation}
there exist two $C^\infty$ symplectic maps $\phi_{(0)}$ and $\phi_{(1)}$ making the diagram \eqref{diag1} to commute and such that $H\circ\phi_{(1)}$ admits the decomposition \eqref{decomp1} where  
\[
K_{2m}\in\mathcal H_{m}^{\mathcal R},\quad\forall\,m\le\mr
\]
satisfying
\[
||K_{2m}||_{\lif}\le  \sfc^{2m-3}\min\{m,\mr\}^{2(m-2)},\quad\forall\,2\le m\le r,
\]
and the remainder $R:\bg(0,2\rho)\mapsto\mathbb R$ is a $C^\infty$ function satisfying
\[
||\X_R(z)||_{\sis}\le \sfc^{3r-1}r^{2r}4^{\mathfrak r-1}\prod_{i=1}^{\mr-1}(1+2^{-i})^{2r+1}||z||_{\sis}^{2r+1}.
\]
Furthermore, the maps $\phi_{(\iota)}:\bg(0,\rho+\iota\rho)\mapsto\mathcal G$ are close to the identity
\[
||\phi_{(\iota)}(z)-z||_{\sis}\le\sum_{i=1}^{\mr-1}\sfc2^{8-i}||z||^3_{\sis},\quad\iota=0,1.
\]
We  proceed by induction on $\mr$. First we note that the case $\mr=1$ is trivial, provided that $\sfc$ is chosen large enough to ensure that $||P_4||_{\lif}\le \sfc$.

From now, we assume that the property holds at step $\mr<r$ and we are going to prove it at the step $\mr+1$. To simplify notations, we denote with a superscript $\sharp$ the maps corresponding to the subsequent step $\mr+1$, such as $R^\sharp$ or $\phi_{(1)}^\sharp$.

\noindent
\underline{$\vartriangleright$\textsf{(2) Resolution of homological equation.}} Now we are going to eliminate the non-resonant terms of $K_{2(\mr+1)}$ by solving the following equation:
\begin{equation}\label{homoequ1}
	\{L_2,\bs\}+K_{2(\mr+1)}=:K_{2(\mr+1)}^{\sharp} \in \mathcal H_{\mr +1}^{\mathcal R}.
\end{equation}
By definition of $L_2$ (see \eqref{Hamz}), we solve the equation above by setting 
\begin{equation}
\label{eq:cohom_res}
K_{2(\mr+1)}^\sharp:=\sum_{\bj\in\mathcal M_m\cap\mathcal R}K_{2(\mr+1),\bj}z_{\bj}\in\mathcal H_{\mr+1}^\mathcal R\quad\text{and}\quad\bs:=\sum_{\bj\in\mathcal M_m\backslash\mathcal R}\frac{\mi K_{2(\mr+1),\bj}}{\Delta_{\bj}}z_{\bj}\in\mathcal H_{\mr+1}.
\end{equation}
Obviously, by induction, we have the estimates
\begin{equation}
\begin{gathered}\label{homoestimate1}
||K_{2(\mr+1)}^\sharp||_{\lif}\le||K_{2(\mr+1)}||_{\lif}\le \sfc^{2\mr-1}\mr^{2(\mr-1)},\\
||\bs||_{\lif}\le||K_{2(\mr+1)}||_{\lif}\le \sfc^{2\mr-1}\mr^{2(\mr-1)}.
\end{gathered}
\end{equation}

\noindent
\underline{$\vartriangleright$\textsf{(3) The new variables.} }
We recall that the Hamiltonian flow of $\bs$ given by Lemma \ref{flow} is well-defined for $|t|\le1$ on $\bg(0,\varepsilon_1)$ where
\begin{equation}\label{radius1}
	\varepsilon_1=\frac14({4}(\mr+1)||\bs||_{\lif})^{-\frac1{2\mr}}\ge\frac14(\sfc^{2\mr}\mr^{2\mr})^{-\frac1{2\mr}}\ge\frac1{2^2\sfc\mr}=2\rho\quad\udl{8}.
\end{equation}
Now we aim at defining the new maps
\[
\phi_{(1)}^\sharp:=\phi_{(1)}\circ\Phi_{\bs}^1,\quad\phi_{(0)}^\sharp:=\Phi_{\bs}^{-1}\circ\phi_{(0)}.
\]
Thanks to \eqref{radius1} we know $\Phi_{\bs}^{\pm1}$ are well-defined on $\bg(0,2\rho)$. By induction hypothesis we have $\phi_{(0)}$ maps $\bg(0,\rho)$ on $\bg(0,2\rho)$. It follows that $\phi_{(0)}^\sharp$ is well-defined on $\bg(0,\rho^\sharp)\subset\bg(0,\rho)$. As for $\phi_{(1)}^\sharp$, we just have to check that $\Phi_{\bs}^1$ maps $\bg(0,2\rho^\sharp)$ on $\bg(0,2\rho)$. Indeed, we consider a larger domain. By the property (iii) of Lemma \ref{flow} we have for $z\in\bg(0,2\rho)$
\begin{align}
	||\Phi_{\bs}^{\pm1}(z)-z||_{\sis}\le&2^{3(\mr+1)}||\bs||_{\lif}||z||_{\sis}^{2\mr+1}\le2^{3(\mr+1)}\sfc^{2\mr-1}\mr^{2(\mr-1)}(2\rho)^{2(\mr-1)}||z||_{\sis}^3\notag\\
	\le&\sfc2^{7-\mr}||z||_{\sis}^3\le\frac{2^3}{\sfc2^\mr\mr^2}||z||_{\sis}.\label{Phi1}
\end{align}
The last inequality implies that for $z\in\bg(0,2\rho^\sharp)$
\[
||\Phi_{\bs}^1(z)||_{\sis}\le(1+\mr^{-1})||z||_{\sis}<\frac{1+\mr}{\mr}2\rho^\sharp=2\rho\quad\udl{4},
\]
i.e. we checked that $\Phi_{\bs}^1$ maps $\bg(0,2\rho^\sharp)$ on $\bg(0,2\rho)$. 
\begin{remark}\label{PhiSt1}
Actually, we have for $t\in[0,1]$ that $\Phi_{\bs}^t$ maps $\bg(0,2\rho^\sharp)$ on $\bg(0,2\rho)$, satisfying
\[
||\Phi_{\bs}^t(z)-z||_\sis\le \sfc2^{7-\mr}||z||_{\sis}^3\le\frac{2^3}{\sfc2^\mr\mr^2}||z||_{\sis}.
\]
\end{remark}
Notice that by construction it is clear that $\phi_{(1)}^\sharp$ and $\phi_{(0)}^\sharp$ are symplectic. Then we want to prove that $\phi_{(1)}^\sharp$ is close to the identity in $\mathcal G$.  Indeed, if $z\in\bg(0,2\rho^\sharp)$
\[
\phi_{(1)}^\sharp(z)-z=\phi_{(1)}\circ\Phi_{\bs}^1(z)-z=\phi_{(1)}\circ\Phi_{\bs}^1(z)-\Phi_{\bs}^1(z)+\Phi_{\bs}^1(z)-z.
\]
Then connecting the induction and  \eqref{Phi1}, one has
\begin{align*}
&||\phi_{(1)}^\sharp(z)-z||_{\sis}\le||\phi_{(1)}\circ\Phi_{\bs}^1(z)-\Phi_{\bs}^1(z)||_{\sis}+||\Phi_{\bs}^1(z)-z||_{\sis}\\
\le&\sum_{i=1}^{\mr-1}\sfc2^{8-i}||\Phi_{\bs}^{1}(z)||^3_{\sis}+\sfc2^{7-\mr}||z||_{\sis}^3\\
\le&\sum_{i=1}^{\mr-1}\sfc2^{8-i}||z||^3_{\sis}\left(1+\frac{2^3}{\sfc2^\mr}\right)^3+\sfc2^{7-\mr}||z||_{\sis}^3\\
\le&\sum_{i=1}^{\mr-1}\sfc2^{8-i}||z||^3_{\sis}\left(1+\frac{2^3\times7}{\sfc2^\mr}\right)+\sfc2^{7-\mr}||z||_{\sis}^3\\
\le&\sum_{i=1}^{\mr-1}\sfc2^{8-i}||z||^3_{\sis}+\sfc2^{7-\mr}||z||_{\sis}^3+2^{8}\sfc\times\frac{2^3\times 7}{\sfc2^\mr}||z||_{\sis}^3\\
\le&\sum_{i=1}^{\mr}\sfc2^{8-i}||z||^3_{\sis}\quad\udl{2^4\times7}.
\end{align*} 
Next we turn to $\phi_{(0)}^\sharp$. Similarly, by the induction one has for $z\in\bg(0,\rho^\sharp)$
\[
||\phi_{(0)}(z)-z||_\sis\le 2^{8}\sfc||z||^3_\sis\le\frac4{\sfc\mr^2}||z||_\sis.
\]
Together with \eqref{Phi1}, we obtain that
\begin{align*}
&||\phi_{(0)}^\sharp(z)-z||_{\sis}\le||\Phi_{\bs}^{-1}\circ\phi_{(0)}(z)-\phi_{(0)}(z)||_{\sis}+||\phi_{(0)}(z)-z||_{\sis}\\
\le&\sfc2^{7-\mr}||\phi_{(0)}(z)||_{\sis}^3+\sum_{i=1}^{\mr-1}\sfc2^{8-i}||z||^3_\sis\\
\le&\sfc2^{7-\mr}||z||_{\sis}^3\left(1+\frac4\sfc\right)^3+\sum_{i=1}^{\mr-1}\sfc2^{8-i}||z||^3_\sis\\
\le&\sfc2^{7-\mr}\left(1+\frac{28}{\sfc}\right)||z||_\sis^3+\sum_{i=1}^{\mr-1}\sfc2^{8-i}||z||^3_\sis\\
\le&\sum_{i=1}^{\mr}\sfc2^{8-i}||z||^3_\sis\quad\udl{28}\\
\le&\sfc2^8||z||_\sis^3\le\frac4\sfc||z||_\sis\le||z||_\sis.
\end{align*}
The last estimate implies that $\phi_{(0)}^\sharp$ maps $\bg(0,\rho^\sharp)$ on $\bg(0,2\rho^\sharp)$.

\noindent
\underline{$\vartriangleright$\textsf{(4) The new expansion.} }
Recall that $\Phi_{\bs}^1$ is the Hamiltonian flow of $\bs$ (see \eqref{PhiS1}). Thus, for $z\in\bg(0,2\rho^\sharp)$ and $K$ a Hamiltonian function, we have
\[
\frac{\md}{\md t}K\circ\Phi_{\bs}^t(z)=\{K,\bs\}\circ\Phi_{\bs}^t(z)=\ad_{\bs}(K)\circ\Phi_{\bs}^t(z).
\]
Hence, by Taylor expansion of $K\circ\Phi^t_{\bs}$ in $t=0$ at the degree $k^*$, one has
\begin{equation}\label{expansion}
	K\circ\Phi^1_{\bs}(z)=\sum_{k=0}^{k^*}\frac1{k!}\ad_{\bs}^k(K)(z)
+\int_0^1\frac{(1-t)^{k^*}}{k^*!}\ad_{\bs}^{k^*+1}(K)\circ\Phi_{\bs}^t(z)\,\md t
\end{equation}
Consequently, we obtain the Taylor expansion of $K_{2n}\circ\Phi_{\bs}^1$
\[
K_{2n}\circ\Phi_{\bs}^1(z)=\sum_{k=0}^{k_n^*}\frac1{k!}\ad_{\bs}^k(K_{2n})(z)+\int_0^1\frac{(1-t)^{k_n^*}}{k_n^*!}\ad_{\bs}^{k_n^*+1}(K_{2n})\circ\Phi_{\bs}^t(z)\,\md t,\quad1\le n\le r.
\]
where $k_n^*$ denotes the largest integer $k$ such that $k\mr+n\le r$, namely
\begin{equation}\label{kn}
	k_n^*\mr+n\le r\quad\text{and}\quad (k_n^*+1)\mr+n>r.
\end{equation}
Recall that $K_2=L_2$. By induction hypothesis, one has
\[
H\circ\phi_{(1)}=\sum_{n=1}^{r}K_{2n}+R,
\]
and thus, since $\phi_{(1)}^\sharp=\phi_{(1)}\circ\Phi_{\bs}^1$, we deduce that
\begin{align*}
	H\circ\phi_{(1)}^\sharp=&\sum_{n=1}^{r}K_{2n}\circ\Phi_{\bs}^1+R\circ\Phi_{\bs}^1\\
	=&\sum_{\substack{0\le k\le k_n^*\\1\le n\le r}}\frac1{k!}\ad_{\bs}^k(K_{2n})+
	\sum_{1\le n\le r}\int_0^1\frac{(1-t)^{k_n^*}}{k_n^*!}\ad_{\bs}^{k_n^*+1}(K_{2n})\circ\Phi_{\bs}^t\,\md t+
	 R\circ\Phi_{\bs}^1\\
	:=&\sum_{m=1}^{r}K_{2m}^\sharp+R^\sharp,
\end{align*}
where 
\begin{gather}
K_{2m}^\sharp=\sum_{\substack{k\mr+n=m \\k\ge0,n\ge1}}\frac1{k!}\ad_{\bs}^k(K_{2n}),\quad1\le m\le r,\label{Kexpansion}
\\
R^\sharp=\sum_{1\le n\le r}\int_0^1\frac{(1-t)^{k_n^*}}{k_n^*!}\ad_{\bs}^{k_n^*+1}(K_{2n})\circ\Phi_{\bs}^t\,\md t+
R\circ\Phi_{\bs}^1.\label{Rexpansion}
\end{gather}
Now we focus on the estimate of \eqref{Kexpansion}.  We have 
\begin{equation}\label{K2m}
	K_{2m}^\sharp=\sum_{\substack{k\mr+1=m \\ k\ge0}}\frac1{k!}\ad_{\bs}^k(L_2)+\sum_{\substack{k\mr+n=m \\ k\ge0,n\ge2}}\frac1{k!}\ad_{\bs}^k(K_{2n}).
\end{equation}
Notice that by above equation \eqref{K2m}, one has
\begin{equation}\label{smallm}
\begin{gathered}
K_{2m}^\sharp=K_{2m},\quad\forall\,1\le m\le\mr,\\
K_{2(\mr+1)}^\sharp=\ad_{\bs}(L_2)+K_{2(\mr+1)}.
\end{gathered}
\end{equation}
Observe that by \eqref{homoequ1}
\[
\frac{\ad_{\bs}(L_2)}{(k+1)!}+\frac{K_{2(\mr+1)}}{k!}=\frac{K_{2(\mr+1)}^\sharp-K_{2(\mr+1)}}{(k+1)!}+\frac{K_{2(\mr+1)}}{k!}=\frac{K_{2(\mr+1)}^\sharp}{(k+1)!}+\left(\frac1{k!}-\frac1{(k+1)!}\right)K_{2(\mr+1)},
\]
which leads, using \eqref{homoestimate1}, to
\begin{equation}\label{adL2}
	\left\|\frac{\ad_{\bs}(L_2)}{(k+1)!}+\frac{K_{2(\mr+1)}}{k!}\right\|_{\lif}\le\frac{||K_{2(\mr+1)}||_{\lif}}{k!}.
\end{equation}
Therefore, we do not  have to take into account the contribution of terms associated with $L_2$ in the estimate of \eqref{K2m}, i.e.
\begin{equation}\label{simpleK2m}
\|K_{2m}^\sharp\|_\lif\le\|K_{2m}\|_\lif+\sum_{\substack{k\mr+n=m \\ k\ge1,n\ge2}}\frac1{k!}\|\ad_{\bs}^k(K_{2n})\|_\lif.
\end{equation}
Next we aim to prove 
\[
\|K_{2m}^\sharp\|_\lif\le \sfc^{2m-3}\min\{q,\mr+1\}^{2(m-2)},\quad\forall\,2\le m\le r.
\]
Since by induction hypothesis and equations \eqref{homoestimate1}, \eqref{smallm}, the above estimate holds for $2\le m\le\mr+1$, it remains to prove that
\begin{equation}\label{K2mEstimate}
	\|K_{2m}^\sharp\|_\lif\le \sfc^{2m-3}(\mr+1)^{2(m-2)},\quad\forall\,\mr+2\le m\le r.
\end{equation}
First, by induction hypothesis, we have
\begin{equation}\label{K2mEstimate1}
	\frac{\|K_{2m}\|_\lif}{\sfc^{2m-3}(\mr+1)^{2(m-2)}}\le\left(\frac{\mr}{\mr+1}\right)^{2(m-2)}\le\left(\frac{\mr}{\mr+1}\right)^{2\mr}\le\left(\frac\mr{\mr+1}\right)^{\mr+1}=\frac1{\left(1+\frac1\mr\right)^{\mr+1}}\le e^{-1}.
\end{equation}
Then, by Lemma \ref{poissonbracket}, we have for the sum in \eqref{simpleK2m}
\begin{align*}
&\sum_{\substack{k\mr+n=m \\ k\ge1,n\ge2}}\frac1{k!}\|\ad_{\bs}^k(K_{2n})\|_\lif\le \sum_{\substack{k\mr+n=m \\ k\ge1,n\ge2}}\frac{{ 4^k}}{k!}\|K_{2n}\|_\lif\big((\mr+1)\|{\bs}\|_\lif\big)^k\prod_{i=0}^{k-1}(n+i\mr)\\
\le&\sum_{\substack{k\mr+n=m \\ k\ge1,n\ge2}} { 4^k} \frac{m^k}{k!}\sfc^{2n-3}\mr^{2(n-2)}\left((\mr+1)\sfc^{2\mr-1}\mr^{2(\mr-1)}\right)^k\\
\le&\sum_{\substack{k\mr+n=m \\ k\ge1,n\ge2}}\frac{m^k}{k!}{ 8^k} \sfc^{2(n+k\mr)-3-k}\mr^{2(n+k\mr-2)-k}\\
=&\sfc^{2m-3}(\mr+1)^{2(m-2)}\sum_{\substack{k\mr+n=m \\ k\ge1,n\ge2}}\frac{m^k{ 8^k}}{k!}\frac{\sfc^{2m-3-k}\mr^{2(m-2)-k}}{\sfc^{2m-3}(\mr+1)^{2(m-2)}}\\
\le&\sfc^{2m-3}(\mr+1)^{2(m-2)}\sum_{k\ge1}\frac{{ 8^k}}{k!}\sfc^{-k}\mr^{-k}m^k\left(\frac{\mr+1}{\mr}\right)^{-2m}\\
\le&\sfc^{2m-3}(\mr+1)^{2(m-2)}\sum_{k\ge1}\frac{1}{k!}\left(\frac k e\right)^k{ 8^k}\sfc^{-k}\mr^{-k}\left(\log\left(\frac{\mr+1}{\mr}\right)^2\right)^{-k}\quad\underline{~\text{by Lemma \ref{mnelog}}~}\\
\le&\sfc^{2m-3}(\mr+1)^{2(m-2)}\sum_{k\ge1}{ 4^k}\sfc^{-k}\mr^{-k}\left(\frac1{1+\mr}\right)^{-k}\\
\le&\frac12\sfc^{2m-3}(\mr+1)^{2(m-2)}\quad\udl{2^7}.
\end{align*}
Connecting the last estimate with \eqref{K2mEstimate1}, we finish the proof of \eqref{K2mEstimate}.

\noindent
\underline{$\vartriangleright$\textsf{(5) The new remainder.}} Now we estimate the new remainder term $R^\sharp$ given by \eqref{Rexpansion}.  Using property (iv) of Lemma \ref{flow} about the estimate of $\mathrm{d}\Phi_{\bs}^1$ and the induction hypothesis we have 
\begin{align}
\|\X_{R\circ\Phi_{\bs}^1}(z)\|_\sis\le2\|\X_R(\Phi_{\bs}(z))\|_\sis\le2\sfc^{3r-1}r^{2r}4^{\mr-1}\prod_{i=1}^{\mr-1}(1+2^{-i})^{2r+1}\|\Phi_{\bs}^1(z)\|_\sis^{2r+1}
\end{align}
Together with \eqref{Phi1} we obtain for $z\in\bg(0,2\rho^\sharp)$
\begin{equation}\label{Restimate1}
\|\X_{R\circ\Phi_{\bs}^1}(z)\|_\sis\le\frac12\sfc^{3r-1}r^{2r}4^{\mr}\prod_{i=1}^{\mr}(1+2^{-i})^{2r+1}\|z\|_\sis^{2r+1}\quad\udl{8}.
\end{equation}
It remains to estimate the sum in \eqref{Rexpansion}. Reasoning as in \eqref{adL2}, with an abuse of notations, we ignore the contribution of the terms associated with $L_2$. Thanks to Remark \ref{PhiSt1} and the triangle inequality, by Lemma \ref{vectorfield} we get for $z\in\bg(0,2\rho^\sharp)$
\stepcounter{equation}
\begin{align*}
&\left\|\frac12\nabla\sum_{2\le n\le r}\int_0^1\frac{(1-t)^{k_n^*}}{k_n^*!}\ad_{\bs}^{k_n^*+1}(K_{2n})\circ\Phi_{\bs}^t(z)\,\md t\right\|_\sis\le\sum_{2\le n\le r}\int_0^1\frac{1}{k_n^*!}\left\|\X_{\ad_{\bs}^{k_n^*+1}(K_{2n})\circ\Phi_{\bs}^t}(z)\right\|_\sis\md t\\
\le&\sum_{2\le n\le r}\int_0^1\frac{2}{k_n^*!}\left\|\X_{\ad_{\bs}^{k_n^*+1}(K_{2n})}\big(\Phi_{\bs}^t(z)\big)\right\|_\sis\md t\\ \label{Restimate1.1}
\le&{4}\sum_{2\le n\le r}\int_0^1\frac{n+(k_n^*+1)\mr}{k_n^*!}\left\|\ad_{\bs}^{k_n^*+1}(K_{2n})\right\|_{\lif}\left\|\Phi_{\bs}^t(z)\right\|_\sis^{2\big(n+(k_n^*+1)\mr\big)-1}\md t\tag{\theequation}
\end{align*}
Recalling the definition \eqref{kn}, we get $r<n+(k_n^*+1)\mr\le 2r$ and $k_n^*\le r-2$. Estimating the last sum \eqref{Restimate1.1} as before, we obtain
\begin{align*}
\eqref{Restimate1.1}&\le{4}\sum_{2\le n\le r}\frac{n+(k_n^*+1)\mr}{k_n^*!}\left\|\ad_{\bs}^{k_n^*+1}(K_{2n})\right\|_{\lif}\left(2||z||_\sis\right)^{2\big(n+(k_n^*+1)\mr\big)-1}\\
&\le{8}r2^{4r-1}\sum_{2\le n\le r}\frac{(k_n^*+1)}{(k_n^*+1)!}\left\|\ad_{\bs}^{k_n^*+1}(K_{2n})\right\|_{\lif}||z||_\sis^{2\big(n+(k_n^*+1)\mr-r-1\big)}\|z\|_\sis^{2r+1}\\
&\le 2 r^2 2^{4r} \sum_{2\le n\le r} \frac{\sfc^{2\big(n+(k_n^*+1)\mr\big)-3}(\mr+1)^{2\big(n+(k_n^*+1)\mr\big)-2}}{\big(4\sfc(\mr+1)\big)^{2\big(n+(k_n^*+1)\mr-r-1\big)}}\|z\|_\sis^{2r+1} \\
&\le  2 r^32^{4r} \sfc^{2r-1}(\mr+1)^{2r}||z||_\sis^{2r+1} \\
&\le  2^{6r}\sfc^{2r-1}r^{2r}||z||_\sis^{2r+1}\le \sfc^{3r-1}r^{2r}||z||_\sis^{2r+1}\quad\underline{~\text{provided }\sfc\ge2^6~} \\
&\le  \frac12\sfc^{3r-1}r^{2r}4^{\mr}\prod_{i=1}^{\mr}(1+2^{-i})^{2r+1}\|z\|_\sis^{2r+1} .
\end{align*}

Connecting the last estimate with \eqref{Restimate1}, we get for $z\in\bg(0,2\rho^\sharp)$
\[
||R^\sharp(z)||_\sis\le \sfc^{3r-1}r^{2r}4^{\mr}\prod_{i=1}^{\mr}(1+2^{-i})^{2r+1}\|z\|_\sis^{2r+1},
\]
which concludes the induction.

\noindent
\underline{$\vartriangleright$\textsf{(6) Calculation of $K_4$ and $L_4$.}}
From the equations \eqref{eq:cohom_res} and \eqref{smallm}, it follows that
$$
K_4(z) = \sum_{\bj\in\mathcal M_m\cap\mathcal R}P_{4,\bj}z_{\bj} = \sum_{\substack{a_1+a_2=b_1+b_2\\a_1^2+a_2^2=b_1^2+b_2^2 \\a_1\ne b_1}}\frac{1}{2(a_1-b_1)^2}z_{a_1}z_{a_2}\overline{z_{b_1}}\overline{z_{b_2}}.
$$
Recalling that (as in \cite{KP1996})
$$
\left. \begin{array}{ccc} a_1+a_2&=&b_1+b_2\\a_1^2+a_2^2&=&b_1^2+b_2^2  \end{array} \right\} \quad \iff \quad \{a_1,a_2\}=\{b_1,b_2\}
$$
we get, as expected (see \eqref{eq:def_L4} for the definition of $L_4$)
$$
K_4(z) = \sum_{a_1\neq a_2}  \frac{1}{2(a_1-a_2)^2}|z_{a_1}|^2|z_{a_2}|^2 =: L_4(z).
$$
This concludes the proof of Proposition~\ref{Bnormalform}. 
\end{proof}

\section{Dynamics of the high modes}

\label{sec:high}
In this Section we use  two (large) truncation parameters $M$ and $N$ linked by
\begin{equation}
\label{eq:m}
 M = 6rN^{2}.
\end{equation}
The parameter  $M$ corresponds to the truncation to the largest index while $N$ will correspond to the truncation at the third largest index.
We will consider as \emph{high} (resp. \emph{low}) the Fourier modes with index modulus greater than $M$ (resp. smaller than  or equal to $M$). We are going to prove some estimates on high modes\footnote{In the sequel we will speak  indistinctly of high and low modes or high and low frequencies, the particularity of high modes obviously being to be associated with high frequencies.} when the Hamiltonian is resonant.
In order to obtain some exponential decay we exploit the Gevrey regularity in a crucial way, building upon the following inequality:
\begin{lemma}\label{lem:expdecay} Let $m,N\ge1$ and $\theta\in(0,1)$. For all $\bj\in\mathcal{M}_{m}$, if $\mu_{3}(\bj)>N$ then 
\[
\sum_{\beta=2}^{2m}|j_\beta|^\theta-\Big(\sum_{\beta=2}^{2m}|j_\beta|\Big)^\theta
\ge(1-\theta)N^\theta.
\]
\end{lemma}
We postpone the proof of this lemma to Appendix~\ref{sec:appendix}.

\begin{proposition}[High frequency vector field estimates for resonant polynomials]\label{prop:highmodes} Let $K_{2m}\in\mathcal{M}_{m}^{\mathcal{R}}$ be a resonant polynomials of order $2m$, then for all $\ell\in\mathbb Z$ we have
\begin{equation}\label{IK}e^{2\sigma|\ell|^{\theta}}|\{ I_\ell,K_{2m}\}|\leq   2e^{-\sigma (1-\theta)\big(\frac{|\ell|}{m}\big)^{\frac{\theta}{ 2}}}\|K_{2m}\|_{\ell^\infty}e^{\sigma|\ell|^\theta}\sqrt{I_\ell} \|z\|^{2m-1}_\sigma,\end{equation}
and for $N\geq1$,
\begin{equation}
\label{IKN}
\sum_{|\ell|\geq mN^{2}}e^{\sigma|\ell|^{\theta}}|\{I_{\ell},K_{2m}\}|^{\frac{1}{2}}\leq\sqrt{2}\|K_{2m}\|_{\ell^{\infty}}^{\frac{1}{2}}e^{-\frac{1}{2}\sigma(1-\theta)N^{\theta}}\|z\|_{\sigma}^{m}\,.
\end{equation}
\end{proposition}

Note that, while $e^{\sigma|\ell|^\theta}\sqrt{I_\ell}$ is the natural behavior linked with $I_\ell$, the extra decay in~\eqref{IK}  is a consequence of the resonant nature of $K_{2m}$.\\
This proposition together with estimate \eqref{XR1} implies that high modes ($|j|\geq M= 6rN^2$) will not move a lot as soon as  $\|z\|_\sigma$ remains under control (and small). This will be concretized in the final step (see Section \ref{sec:dyn}) by a \emph{double bootstrap argument}. 

\begin{proof}[ Proof of Proposition \ref{prop:highmodes}.]
Let $K_{2m}=\sum_{\bj\in\mathcal R_m}c_{\bj}z_{\bj}\in \mathcal M^{\mathcal R}_m$, one has using Lemma \ref{lem:mu1mu3}
\begin{align*}
|\{ I_\ell,K_{2m}\}|&\leq \sum_{\bj\in\mathcal R_m}|c_{\bj}|\ |\{ I_\ell,z_{\bj}\}|\leq \| K_{2m}\|\sum_{\bj\in\mathcal R_m}\ |\{ I_\ell,z_{\bj}\}| \\
&\leq \| K_{2m}\|\sum_{\substack{\bj\in\mathcal R_m\\ \mu_1(\bj)\geq \langle\ell\rangle}}\ |\{ I_\ell,z_{\bj}\}|
\leq 2\| K_{2m}\|\sum_{\substack{((1,\ell),\bj)\in\mathcal M_m\\ \mu_2(\bj)\geq\big(\frac{\langle\ell\rangle}{m}\big)^{\frac12}}}|z_{\bj}||z_\ell|.
\end{align*}
Therefore, writing $\bj=(j_1,\cdots,j_{2m-1})\in(\mathbb U_2\times\mathbb Z)^{2m-1}$ and $j_i=(\delta_i,a_i)$,
\begin{align*}
e^{2\sigma|\ell|^\theta}&|\{ I_\ell,K_{2m}\}|\leq 2 \|K_{2m}\|_{\ell^\infty}e^{\sigma|\ell|^{\theta}}\sqrt{I_\ell}\sum_{\substack{\sum_{i=1}^{2m-1} \delta_j a_j=-\ell\\ \mu_2(\bj)\geq \big(\frac{\langle\ell\rangle}{m}\big)^{\frac12}}}e^{\sigma|\sum_{i=1}^{2m-1} \delta_j a_j|^\theta}|z_{\bj}|\\
&\leq 2 \|K_{2m}\|_{\ell^\infty}\sqrt{I_\ell}   \left(\sup_{\mu_2(\bj)\geq \big(\frac{\langle\ell\rangle}{m}\big)^{\frac12}}e^{\sigma(|\sum_{i=1}^{2m-1} \delta_j a_j|^\theta-\sum_{i=1}^{2m-1}  |a_j|^\theta)} \right) \sum_{\sum_{i=1}^{2m-1} \delta_j a_j=-\ell} \prod_{i=1}^{2m-1}e^{\sigma|j_i|^\theta}| z_{j_i}|
\\
&\leq 2 \|K_{2m}\|_{\ell^\infty}e^{\sigma|\ell|^{\theta}}\sqrt{I_\ell}e^{-\sigma (1-\theta)\big(\frac{|\ell|}{m}\big)^{\frac{\theta}{ 2}}}\Big(\sum_{\sum_{i=1}^{2m-1} \delta_j a_j=-\ell} \prod_{i=1}^{2m-1}e^{\sigma|j_i|^\theta}| z_{j_i}|\Big)
\end{align*}
where we have used, for the last line, Lemma \ref{lem:expdecay}. Estimate~\eqref{IK} easily follows. We conclude the proof of Proposition~\ref{prop:highmodes} as follows: first apply Cauchy-Schwarz's inequality,
\begin{equation*}
\begin{split}
&\sum_{|\ell|>M}e^{\sigma|\ell|^{\theta}}|\{I_{\ell},K_{2m}\}|^{\frac{1}{2}}\\ \leq& \sqrt{2}\|K_{2m}\|_{\ell^{\infty}}^{\frac{1}{2}}e^{-\frac{1}{2}\sigma(1-\theta)N^{\theta}}\sum_{|\ell|>M}e^{\frac{1}{2}\sigma|\ell|^{\theta}}I_{\ell}^{\frac{1}{4}}\Big(\sum_{\sum_{i=1}^{2m-1} \delta_j a_j=-\ell} \prod_{i=1}^{2m-1}e^{\sigma|j_i|^\theta}| z_{j_i}|\Big)^{\frac{1}{2}} \\
\leq& \sqrt{2}\|K_{2m}\|_{\ell^{\infty}}^{\frac{1}{2}}e^{-\frac{1}{2}\sigma(1-\theta)N^{\theta}}\Big(\sum_{|\ell|>M}e^{\sigma}|\ell|^{\theta}\sqrt{I_{\ell}}\Big)^{\frac{1}{2}}\Big(\sum_{|\ell|\geq M}\sum_{\sum_{i=1}^{2m-1} \delta_j a_j=-\ell} \prod_{i=1}^{2m-1}e^{\sigma|j_i|^\theta}| z_{j_i}|\Big)^{\frac{1}{2}}\,,
\end{split}
\end{equation*}
and deduce from Young's convolution inequality that 
\[
\sum_{|\ell|>M}e^{\sigma|\ell|^{\theta}}|\{I_{\ell},K_{2m}\}|^{\frac{1}{2}} \leq\sqrt{2}\|K_{2m}\|_{\ell^{\infty}}^{\frac{1}{2}}e^{-\frac{1}{2}\sigma(1-\theta)N^{\theta}}\|z\|_{\sigma}^{m}\,.
\]
This concludes the proof of Proposition~\ref{prop:highmodes}.

\end{proof}
Next we introduce the set of Hamiltonian polynomials whose monomials contains at least one high modes:
\[
\mathcal{H}_{m}^{(> M)} :=\Big\{P\in\mathcal{H}_{m}\ |\  P_{\bj}\neq0\implies \mu_{1}(\bj)> M\Big\}\,,\quad \mathcal{H}_{m}^{(\leq M)} = \mathcal{H}_{m}\setminus\mathcal{H}_{m}^{(> M)}\,.
\]
When we truncate the Hamiltonian system using $\Pi_M$ (see \eqref{PiM}) in order to reduce our problem to a finite dimensional phase space, we have to control the dynamics of the high modes. This is essentially done by Proposition~\ref{prop:highmodes}, but we also have to control the effect of the high mode part of the Hamiltonian (i.e. the part in $\cup_m\mathcal{H}_{m}^{(> M)}$ ) on the dynamics of the low modes (see \eqref{eq:v}).
In the next Lemma we control this effect in the case of a  resonant Hamiltonian (which is essentially what we get back after Proposition \ref{Bnormalform}).   Actually we exploit the mismatch between high and low modes, together with the resonant structure of the Hamiltonian in resonant normal form, to gain an exponential decay factor. 
\begin{lemma}[Mismatch vector field estimate]\label{lem:mis} For all $m\geq1$, if $K_{2m}\in\mathcal{H}_{m}^{(> M)}$ then for all $z\in\mathcal{G}$ we have (recall \eqref{eq:m})
\[
\|\Pi_M\X_{K_{2m}}(z)\|_{\sigma}\leq 2m\|K_{2m}\|_{\ell^\infty}e^{-\sigma N^{\theta}}\|z\|_{\sigma}^{2m-1}\,.
\]
\end{lemma}

\begin{proof} Let $j_{0}\in\mathbb{U}_{2}\times\mathbb Z$ with $|j_0|\leq M$, $K_{2m}\in\mathcal{H}_{m}^{(> M)}$ and $z\in\mathcal{G}$, we have
\[
\Big(\X_{K_{2m}}(z)\Big)_{j_{0}} =2m\mi\delta(j_{0})\sum_{\substack{\bj\in\mathcal R_{m}\\ \overline{j_{0}}=j_{1}}}K_{2m}(\bj)\prod_{\alpha=2}^{2m}z_{\bj_{\alpha}}\,.
\]
Given $\bj\in\mathcal{M}_{m}$ we write 
\[
\mu_{1}(\bj)=|j_{1}^{\ast}|\geq |j_{2}^{\ast}|\geq\cdots\geq|j_{2m}^{\ast}|
\]
the non-increasing ordering of $|j_1|,\cdots,|j_{2m}|$. Since $K_{2m}\in\mathcal{H}_{m}^{(> M)}$, $\mu_1(\bj)>M$. Thus  if $K_{2m}(\bj)\neq0$ with $j_{1}=\overline{j_{0}}$ then 
\[
j_{1}^{\ast}\in\{j_{2},\cdots,j_{2m}\}.
\]
Now if $|j_{1}^{\ast}|= |j_{2}^{\ast}|$ then two indices in $\{j_{2},\cdots,j_{2m}\}$ have a modulus greater than $M$. But  even if $|j_{2}^{\ast}|<|j_{1}^{\ast}|$ (a case that does not exclude that $j_0=j_{2}^{\ast}$) 
we then have $\{z_{\bj},|z_{j_{1}^{\ast}}|^2\}\neq0$ and we deduce  by Lemma~\ref{lem:mu1mu3} that $|j_{3}^{\ast}|\geq (\frac{M}{m})^\frac{1}{2}$. So in both cases we have
$$\sum_{\alpha=2}^{2m} |j_\alpha|^\theta\geq M^\theta+ (\frac{M}{m})^\frac{\theta}{2}\geq M^\theta+N^\theta .$$
%
%
Therefore we get
\begin{align*}
\|\Pi_{M}\X_{K_{2m}}(z)\|_{\sigma} &= 2m\sum_{\substack{j_{0}\in\mathbb U_{2}\times\mathbb Z\\ |j_{0}|\leq M}}e^{\sigma|j_{0}|^{\theta}}\Big|\sum_{\substack{\bj\in\mathcal R_{m}\\ \overline{j_{0}}=j_{1}}}K_{2m}(\bj)\prod_{\alpha=2}^{2m}z_{\bj_{\alpha}}\Big|\\
&\leq 2m\|K_{2m}\|_{\ell^\infty}\sum_{\substack{j_{0}\in\mathbb U_{2}\times\mathbb Z\\ |j_{0}|\leq M}}\sum_{\substack{\bj\in\mathcal R_{m}\\ \overline{j_{0}}=j_{1}}}e^{\sigma(|j_{0}|^{\theta}-\sum_{\alpha=2}^{2m} |j_\alpha|^\theta)}
\prod_{\alpha=2}^{2m}e^{\sigma|j_\alpha|^\theta}|z_{\bj_{\alpha}}|\\
&\leq 2m\|K_{2m}\|_{\ell^\infty}e^{-\sigma N^\theta}\sum_{j_{0}\in\mathbb U_{2}\times\mathbb Z}\sum_{\substack{\bj\in\mathcal M_{m}\\ \overline{j_{0}}=j_{1}}}\prod_{\alpha=2}^{2m}e^{\sigma|\bj_{\alpha}|^\theta}|z_{\bj_{\alpha}}|
\end{align*}
and we conclude by Young's convolution inequality.

\end{proof}

%

\section{Rational fractions}

\label{sec:rational1}

\subsection{Setting} In this section, we consider a rational normal form theorem on the finite dimensional space
$$
\mathcal{G}_M:=\mathbb{C}^{\llbracket -M,M \rrbracket}\,,
$$
where $M\geq 1$ is a real number and  $\llbracket -M,M \rrbracket := [-M,M] \cap \mathbb{Z}$. We always identify this space with a subspace of $\mathcal G$ by setting 
$$
\forall z\in \mathcal{G}_M, \quad |a|>M \quad \Rightarrow \quad z_a:=0.
$$
Naturally, $\mathbb{B}_M$ will denote the balls of $\mathcal G_M$.
\subsubsection{Integrable multi-indices}

\begin{definition}[Sets $ \mathrm{Int}$, $\mathcal{N}$ and $\bN$] A multi-index $\bj \in \mathcal{J}$ is said \emph{integrable}, and we denote it by $\bj \in \mathrm{Int}$, if $z_{\bj}$ depends only on the actions or more precisely, if there exists a permutation, $\varphi \in \mathfrak{S}_{\# \bj}$, such that for all $\alpha \in \llbracket 1,\# j \rrbracket$, $\bj_{\varphi_\alpha} = \overline{\bj_\alpha}$. We denote by $\mathcal{N}$ the set of the non integrable multi-index which are resonant, that is
$$
\mathcal{N} := \mathcal{R} \setminus \mathrm{Int}
$$
and we denote by $\bN$ the set of the multi-non-integrable multi indices, that is
$$
\bN := \bigcup_{n\geq 0} \mathcal{N}^n.
$$
\end{definition}

\begin{remark}
Note that if $\bj \in \mathcal{N}$ then $\# \bj \geq 6$.
\end{remark}

\subsubsection{Small divisors and non resonant set}
Given $M\geq 1$ and $\bj \equiv (\delta_\beta,a_\beta)_{1\leq \beta \leq \# \bj} \in \mathcal{J}$, we set
\[
\omega_{\bj}^M(z) := \sum_{\beta=1}^{\#\bj}\delta_{\beta}\partial_{I_{a_{\beta}}}L_{4}=\sum_{\beta=1}^{\# \bj}\delta_\beta\sum_{\substack{a\ne a_\beta \\ |a|\leq M }}\frac{|z_a|^2}{(a-a_\beta)^2}\,,
\]
Define
\begin{equation}\label{eq:NrM}
\mathcal{N}^{r,M}:=\{\bj\in\mathcal{N}\ |\  \#\bj\leq r\,,\quad \mu_{1}(\bj)\leq M\}\,,
\end{equation}
and
\begin{align}
\label{eq:U}
\mathfrak{U}_\gamma^{r,M}&:=\{ z \in \mathcal{G}_M \ |\ \underset{\bj\in \mathcal{N}^{r,M}}{\min}  \ |\omega_{\bj}^M(z)|> \gamma||z||_\sis^2\}\,,\\
\mathfrak{U}_{0^+}^{r,M}&:=\{ z \in \mathcal{G}_M \ |\ \underset{\bj\in \mathcal{N}^{r,M}}{\min}   \ |\omega_{\bj}^M(z)|>0\}.
\end{align}

\begin{lemma}\label{lipschitzomega}
	For all $M\geq 1$, $\bj\in\mathcal J$, and $z,z'\in\mathcal G_M$, we have
\begin{equation}
\label{eq:lip}
|\omega_{\bj}^{M}(z)-\omega_{\bj}^{M}(z')|\leq \#\bj\sum_{|a|\leq M}|I_{a}(z)-I_{a}(z')|\,.
\end{equation}
\end{lemma}

\begin{remark}\label{rem:lip} As a consequence we also have 
\[
|\omega_{\bj}^{M}(z)-\omega_{\bj}^{M}(z')|\leq\#\bj\max(\|z\|_{\sigma},\|z'\|_{\sigma})\|z-z'\|_{\sigma}\,,
\]
which we will use when comparing $\omega_{\bj}(z)$ and $\omega_{\bj}(z')$, with for instance $z'=\varphi_{\iota}(z)$. 
\end{remark}

\begin{proof}
Let us first prove that for $M\geq 1$, $j_0\in\mathbb U_2\times\mathbb Z$ and $\bj \in \mathcal{J}$, $\partial_{I_{j_0}} \omega_{\bj}^M$ is a constant, independent of $z$, and satisfies
\begin{equation}
    \label{partialomega}  \left|\partial_{I_{j_0}}\omega_{\bj}^M\right|\le\#\bj.
\end{equation}
	By definition $\partial_{I_{j_0}} = \partial_{z_{j_0}} \partial_{\overline{z_{j_0}}} $, setting $\bj \equiv (\delta_\beta,a_\beta)_{1\leq \beta \leq \# \bj}$ and $j_0 = (\delta_0,a_0)$, we have
	\[
	\partial_{I_{j_0}}\omega_{\bj}^M= \mathbbm{1}_{|a_0|\leq M} \sum_{\substack{1\le \beta \le \# \bj \\ a_\beta \ne a_0}}\delta_\beta \frac1{(a_0-a_\beta)^2}.
	\]
	We conclude by applying the triangular inequality. The Lipschitz estimate~\eqref{eq:lip} then follows from the mean value Theorem:
\[
|\omega_{\bj}^M(z)-\omega_{\bj}^M(z')|\le  \# \bj \sum_{|a|\leq M} | |z_a|^2 - |z_a'|^2| \,.
\]	
This concludes the proof of the Lemma.
\end{proof} 

\subsection{Definition and main properties}

In this paper, we consider rational fractions of the form
\begin{equation}
\label{eq:def_Q_and_f}
Q (z) = \sum_{\bj \in \mathcal{R}} f_{\bj}(z) z_{\bj} \quad \mathrm{with} \quad  f_{\bj}(z) := \sum_{\bh \in \bN} Q_{\bj,\bh}   \prod_{\alpha = 1}^{\# \bh} \frac{\mi}{\omega_{\bh_{\alpha}}^M(z)}.
\end{equation}
where only finitely many coefficients $Q_{\bj,\bh} \in \mathbb{C}$ are non zero. More precisely, we consider the following subspace:

\begin{definition}[Formal rational fractions] 
\label{def:rat:fract}
For $m\geq 3$, $M\geq 1$ and $Q \in \mathbb{C}^{\mathcal{R} \times \bN} $. We says that $Q$ belongs to $\mathscr{H}_{q}^{M}$ if it
  satisfies the following conditions 
\begin{enumerate}[i)]
\item \underline{\emph{Order $2q$.}} For all $\bj \in \mathcal{R}$, all $\bh \in \bN$, if $Q_{\bj,\bh} \neq 0$ then 
$$
2q= \#\bj - 2 \# \bh.
$$

\item \underline{\emph{Finite number of modes.}} For all $\bj \in \mathcal{R}$, all $\bh \in \bN$, if $Q_{\bj,\bh} \neq 0$ then
$$
\mu_1( \bj ) \leq M \quad \mathrm{and} \quad \max_{1\leq \alpha \leq \# \bh} \mu_1(\bh_\alpha) \leq M.
$$

\item \underline{\emph{Reality condition.}} For all $\bj \in \mathcal{R}$, all $\bh \in \bN$, $\overline{Q_{\bj,\bh}} =  Q_{\overline{\bj},\overline{\bh}}$.

\item \underline{\emph{Uniform bound on the degree of the numerators.}} We have
$$
\mathfrak m_Q:=\frac12 \sup_{\substack{\bj \in \mathcal{R} \\ \exists  \bh \in \bN, \ Q_{\bj,\bh} \neq 0 } } \# \bj < \infty.
$$

\item \underline{\emph{Finite complexity of the denominators.}}  We have
$$
\mathfrak h_Q:=\sup_{\substack{\bh \in \bN \\ \exists \bj \in \mathcal{R}, \ Q_{\bj,\bh} \neq 0 } } \ \sup_{1\leq \alpha \leq \# \bh} \# \bh_\alpha < \infty.
$$
\end{enumerate}
\end{definition}

\begin{remark}
\label{rem:ark}
Note that assumptions i), ii), iv) and v) imply that only finitely many coefficients are non-zero. In particular, the order condition i) implies that
$$
\mathfrak n_Q := \sup_{\substack{\bh \in \bN \\ \exists \bj \in \mathcal{R}, \ Q_{\bj,\bh} \neq 0 } } \# \bh\leq  \mathfrak{m}_Q+q\leq 2 \mathfrak{m}_Q < \infty.
$$
We also note that thanks to the reality condition iii), the maps $z\mapsto Q(z)$ are real valued.
\end{remark}
\begin{remark}
\label{rem:unique}
The coefficients $Q_{\bj,\bh}$ are not unique in the sense that two distinct sequences of coefficients may generate the same rational function. We should impose some heavy symmetry conditions to remedy this point. Nevertheless, it does not matter for us. The important point is that, the map, defined implicitly by \eqref{eq:def_Q_and_f}, which associates a sequence of coefficients $(Q_{\bj,\bh})_{\mathcal{R} \times \bN} \in \mathbb{C}^{\mathcal{R} \times \bN}$ with a function on $\mathfrak{U}_{0^+}^{\mathfrak{h}_{Q},M}$, is $\mathbb{R}$-linear. \end{remark}

\begin{definition}[Norm of the rational fractions]\label{control} Given $Q  \in \mathscr{H}_{q}^{M}$, we set
	\begin{equation}\label{Qlof}
	||Q||_\lof:=\sum_{0\le m\le\mathfrak m_Q}\sup_{\bj\in\mathcal R_m} \sum_{\bh \in \bN}|Q_{\bj,\bh}|.
\end{equation}
\end{definition}

Now we present the properties of rational Hamiltonians.

\begin{lemma}[Rational vector field]\label{Rvectorfield}
Let $q\ge2$ and $Q\in\mathscr H_q^M$, 
then the associated Hamiltonian vector field is smooth and local Lipschitz. More precisely, for all $\gamma \in (0,1)$ and all $z\in \mathfrak{U}_{\gamma }^{\mathfrak{h}_Q,M} $ we have the estimates
\begin{gather*}
||\X_Q(z)||_{\sis}\le 2\frac{\mathfrak m_Q(1+\mathfrak h_Q)}{\gamma^{\mathfrak n_Q+1}}||Q||_{\lof}||z||_{\sis}^{2q-1},\\
||\md\X_Q(z)(w)||_{\sis}\le 4\frac{\mathfrak m_Q^2(1+2\mathfrak h_Q)^2}{\gamma^{\mathfrak n_Q+2}}||Q||_{\lof}||z||_{\sis}^{2(q-1)}||w||_{\sis}.
\end{gather*}
\end{lemma}

\begin{proof} \underline{$\vartriangleright$\textsf{Vector field.}} First, we note that, given $j\in \mathbb{U}_2\times\mathbb Z$, we have
$$
\partial_{z_{j}} Q(z) =  \partial_{u_{j}}  \underline{Q} (z,z) + \partial_{v_{j}} \underline{Q} (z,z) =: F_j + G_j
$$
where we have set
\begin{equation}\label{Q_under}
\underline{Q}(u,v):=\sum_{\bj\in\mathcal R}f_{\bj}^Q(u) v_{\bj}.
\end{equation}

When we estimate $G$ the $f_{\bj}^Q$ coefficients can be seen as constants. As a consequence, the estimate is the same as the polynomial one we proved in Lemma \ref{vectorfield}, i.e. one has
$$
G_j(z) =\sum_{0\leq m \leq \mathfrak{m}_Q}  \sum_{0 \leq \alpha \leq 2 m}\sum_{\substack{\bj\in\mathcal R_m \\ \bj_\alpha =j }}f_{\bj}^Q(z) \prod_{\beta \neq \alpha} z_{\bj_\beta}
$$
and so using the Young convolution inequality and the order condition,
\begin{equation*}
\begin{split}
||G||_{\sis} &= \sum_{j \in\mathbb U_2\times\mathbb Z}e^{\sigma|j|^\theta} \sum_{0\leq m \leq \mathfrak{m}_Q} \sum_{0 \leq \alpha \leq 2 m}  \sum_{\substack{\bj\in\mathcal R_m \\ \bj_{\alpha} =j }} |f_{\bj}^Q(z)| \prod_{\beta \neq \alpha} |z_{\bj_\beta}| \\
&\leq \sum_{0\leq m \leq \mathfrak{m}_Q} \sum_{0 \leq \alpha \leq 2 m}  \sum_{j \in\mathbb U_2\times\mathbb Z} \sum_{\substack{\bj\in\mathcal R_m \\ \bj_{\alpha}=j }} \left( \sum_{\bh \in \bN} \frac{|Q_{\bj,\bh}| }{(\gamma \|z\|_{\sis}^2)^{\# \bh}}   \right) \prod_{\beta\neq \alpha} e^{\sigma|\bj_\beta|^\theta} |z_{\bj_\beta}| \\
&= \sum_{0\leq m \leq \mathfrak{m}_Q} \sum_{0 \leq \alpha \leq 2 m} \sum_{j \in\mathbb U_2\times\mathbb Z} \sum_{\substack{\bj\in\mathcal R_m \\ \bj_{\alpha}=j }} \left( \sum_{\bh \in \bN} |Q_{\bj,\bh}| \gamma^{- \# \bh} \| z\|_{\sis}^{2q-2m}   \right) \prod_{\beta\neq \alpha} e^{\sigma|\bj_\beta|^\theta} |z_{\bj_\beta}| \\
&\leq (2m)\sum_{0\leq m \leq \mathfrak{m}_Q}   \big( \sup_{\substack{\bj\in\mathcal R_m  }} \sum_{ \bh \in \bN }|Q_{\bj,\bh}| \big) \gamma^{- \mathfrak{n}_Q} ||z||_{\sis}^{2q-1}   \leq 2  \mathfrak m_Q \gamma^{-\mathfrak n_Q} ||Q||_{\lof} ||z||_{\sis}^{2q-1}.
\end{split}
\end{equation*}
Now, we focus on estimating $F$. First, we note that we have
$$
F_j(z) = \sum_{0\leq m \leq \mathfrak{m}_Q}  \sum_{\bj\in\mathcal R_m } \left(  \sum_{\bh \in \bN} Q_{\bj,\bh}\sum_{ \beta =1 }^{\# \bh} \frac{(-1)z_{\overline{j}} \partial_{I_j} \omega_{\bh_{\beta}}^M(z)}{  \omega_{\bh_{\beta}}^M(z) }  \prod_{\alpha=1}^{\# \bh} \frac{\mi}{\omega_{\bh_{\alpha}}^M(z)} \right) z_{\bj}
$$
Using that $ |\partial_{I_j} \omega_{\bh_{\beta}}^M(z)| \leq \# \bh_{\beta} \leq \mathfrak{h}_Q$ (see Lemma \eqref{partialomega}) and $\mathfrak{n}_Q \leq 2\mathfrak{m}_Q$, we get that
\begin{equation*}
\begin{split}
|F_j(z)| &\leq  \sum_{0\leq m \leq \mathfrak{m}_Q}  \sum_{\bj\in\mathcal R_m}   \gamma^{-\mathfrak n_Q -1}  |z_j| \mathfrak{n}_Q \mathfrak{h}_Q \left(    \sum_{\bh \in \bN} |Q_{\bj,\bh}| \right) \| z\|_{\sis}^{2q-2m-2} |z_{\bj}|\\
 &\leq    2 \mathfrak m_Q ||Q||_{\lof}   \gamma^{-\mathfrak n_Q -1} \mathfrak{h}_Q   \| z\|_{\sis}^{2q-2} |z_j|.
\end{split}
\end{equation*}
and so we deduce that
$$
||F||_{\sis} \leq 2  \mathfrak m_Q \mathfrak{h}_Q ||Q||_{\lof}   \gamma^{-\mathfrak n_Q -1}    \| z\|_{\sis}^{2q-1}.
$$
Since $\gamma<1$, putting together the estimates we proved on $F$ and $G$ we get the estimate on $||\X_Q(z)||_{\sis}$.

\medskip

\noindent \underline{$\vartriangleright$\textsf{Differential of the Vector field.}} We have to distinguish $5$ sub-cases depending on which kind of term we have to derive (numerators or denominators). Nevertheless, the estimates are almost the same as the ones for the vector fields. So we omit the proof.

\end{proof}

\begin{lemma}[Rational Poisson bracket]\label{Rpoissonbracket}
Let $Q\in\mathscr H_{q}^M$ and $Q'\in\mathscr H_{q'}^M$ with $q,q'\geq 2$, then there exists $Q''\in\mathscr H_{q''}^M$ with $q''=q+q'-1$ such that 
\[
Q''=\{Q,Q'\} \quad \mathrm{on} \quad \mathfrak{U}_{0^+}^{\mathfrak{h}_{Q''},M}
\]
satisfying 
\[
	||Q''||_{\lof}\le 4 \, \mathfrak m_Q\mathfrak m_{Q'}(1+\mathfrak h_Q+\mathfrak h_{Q'})||Q||_{\lof}||Q'||_{\lof}.
\]
Moreover, one has
$$
\mathfrak m_{Q''}\le\mathfrak m_Q+\mathfrak m_{Q'},\quad
\mathfrak n_{Q''}\le\mathfrak n_Q+\mathfrak n_{Q'}+1,\ \quad
\mathfrak h_{Q''}\le\max\{\mathfrak h_Q,\mathfrak h_{Q'}\}. $$
\end{lemma}
\begin{proof} We define $\underline{Q},\underline{Q}'$ as in \eqref{Q_under}. First, we note that since for all $z\in \mathcal{G}_M$, $\underline{Q}(\cdot,z),\underline{Q}'(\cdot,z)$ depends only on the actions, we have 
$$
\{ \underline{Q}(\cdot,z),\underline{Q}'(\cdot,z) \} = 0
$$
and so for all $z\in \mathfrak{U}_{0^+}^{\mathfrak{h}_{Q},M} \cap \mathfrak{U}_{0^+}^{\mathfrak{h}_{Q'},M} $
\begin{equation*}
\begin{array}{cccccccc}
\{ Q,Q'\}(z) &=& \{ \underline{Q}(z,\cdot),\underline{Q}'(z,\cdot) \}(z) &+&  \{ \underline{Q}(z,\cdot),\underline{Q}'(\cdot,z) \}(z)  &+&  \{ \underline{Q}(\cdot,z),\underline{Q}'(z,\cdot) \}(z) \vspace{0.3cm} \\
&=:& Q^{(1)}(z) &+& Q^{(2)}(z) &+& Q^{(3)}(z).
\end{array}
\end{equation*}

\medskip

\noindent \underline{$\vartriangleright$\textsf{Estimation of $Q^{(1)}$.}}  It is defined as the Poisson bracket of $2$ polynomials, so the proof is very similar to the classical one. Indeed, we have
\begin{equation*}
\begin{split}
Q^{(1)}(z) &= \sum_{\bj\in\mathcal R} \sum_{\bj'\in\mathcal R} f_{\bj}^Q(z) f_{\bj'}^{Q'}(z) \{ z_{\bj} ,  z_{\bj'} \} \\
&= \sum_{\substack{0\leq m \leq \mathfrak{m}_Q\\ 0\leq m' \leq \mathfrak{m}_{Q'}}} \sum_{\substack{\bj\in\mathcal R_m \\ \bj'\in\mathcal R_{m'} }} \sum_{\substack{1\leq \alpha \leq 2m \\ 1\leq \alpha' \leq 2m'}}  \sum_{j \in \mathbb{U}_2 \times \mathbb{Z}} f_{\bj}^Q(z) f_{\bj'}^{Q'}(z) (\mi \delta(j)) \mathbbm{1}_{\bj_\alpha = j} \mathbbm{1}_{\bj'_{\alpha'} = \overline{j}}\  \frac{z_{\bj} z_{\bj'}}{|z_j|^2}
\end{split}
\end{equation*}
Now, let us consider the map
$$
\Psi_{\alpha,\alpha'}^{m,m'} : \left\{ \begin{array}{ccc} \{(\bj,\bj') \in \mathcal{R}_m \times \mathcal{R}_{m'} \ | \ \bj_\alpha = \overline{\bj'_{\alpha'}}   \} & \to &  \mathcal{J}_{m+m'-1} \\
(\bj,\bj') & \mapsto & ((\bj_\beta)_{\beta\neq \alpha} , (\bj'_{\beta'})_{\beta'\neq \alpha'})  \end{array} \right.
$$
and denote by $\mathcal{I}_{m,m'}$ its image (note that it clearly does not depend on $(\alpha,\alpha')$). It is straightforward to check that $\mathcal{I}_{m,m'}\subset \mathcal{R}_{m+m'-1}$. Moreover, since $\mathcal{R} \subset \mathcal{M}$, it is injective.
 As a consequence, we have
$$
Q^{(1)}(z) = \sum_{\bj'' \in \mathcal{R} } \sum_{\substack{m+m'-1=\#\bj''/2\\ \bj'' \in \mathcal{I}_{m,m'}}}  \sum_{\substack{1\leq \alpha \leq 2m \\ 1\leq \alpha' \leq 2m'\\ (\bj,\bj'):=(\Psi_{\alpha,\alpha'}^{m,m'})^{-1}(\bj'')}}  (\mi \delta_{\bj_\alpha}) f_{\bj}^Q(z) f_{\bj'}^{Q'}(z) z_{\bj''} .
$$
Expanding $f_{\bj}^Q, f_{\bj'}^{Q'}$, we get
$$
Q^{(1)}(z) =   \sum_{\bj''\in\mathcal R} g_{\bj''}(z) z_{\bj''}
$$
where, setting $m''= \# \bj''/2$,
\begin{equation*}
\begin{split}
g_{\bj''}(z) &:=  \sum_{\substack{m+m'-1=m'' \\ \bj'' \in \mathcal{I}_{m,m'}}}  \sum_{\substack{1\leq \alpha \leq 2m \\ 1\leq \alpha' \leq 2m'\\ (\bj,\bj')=(\Psi_{\alpha,\alpha'}^{m,m'})^{-1}(\bj'')}}  (\mi \delta_{\bj_\alpha}) f_{\bj}^Q(z) f_{\bj'}^{Q'}(z)   \\
 &=    \sum_{\substack{m+m'-1=m''\\ \bj'' \in \mathcal{I}_{m,m'}}}  \sum_{\substack{1\leq \alpha \leq 2m \\ 1\leq \alpha' \leq 2m'\\ (\bj,\bj')=(\Psi_{\alpha,\alpha'}^{m,m'})^{-1}(\bj'')}}  (\mi \delta_{\bj_\alpha})  \sum_{\bh'' \in  \bN} \sum_{\substack{ \bh,\bh' \in \bN\\  (\bh,\bh'):= \bh'' }} Q_{\bj,\bh}   Q_{\bj',\bh'} \prod_{\beta = 1}^{\# \bh''} \frac{\mi}{\omega_{\bh_{\beta}''}^M(z)} \\
 &=   \sum_{\bh'' \in  \bN} Q^{(1)}_{\bj'',\bh''} \prod_{\beta = 1}^{\# \bh''} \frac{\mi}{\omega_{\bh_{\beta}''}^M(z)}
 \end{split}
\end{equation*}
with
$$
Q^{(1)}_{\bj'',\bh''} :=  \sum_{\substack{m+m'-1=m''\\ \bj'' \in \mathcal{I}_{m,m'}}}  \sum_{\substack{1\leq \alpha \leq 2m \\ 1\leq \alpha' \leq 2m'\\ (\bj,\bj')=(\Psi_{\alpha,\alpha'}^{m,m'})^{-1}(\bj'')}}  (\mi \delta_{\bj_\alpha})    \sum_{\substack{ \bh,\bh' \in \bN\\  (\bh,\bh'):= \bh'' }}  Q_{\bj,\bh}   Q_{\bj',\bh'}.
$$
Thanks to this explicit expression, it can be easily checked that the coefficients $Q^{(1)}_{\bj'',\bh''}$ satisfy all the conditions to belong to $\mathscr{H}_{q''}^{M}$ with $q''=q+q'-1$.  Moreover, it is clear that
$$
\mathfrak m_{Q^{(1)}}\le\mathfrak m_Q+\mathfrak m_{Q'}-1,\quad
\mathfrak n_{Q^{(1)}}\le\mathfrak n_Q+\mathfrak n_{Q'},\ \quad
\mathfrak h_{Q^{(1)}}\le\max\{\mathfrak h_Q,\mathfrak h_{Q'}\}. 
$$
The crucial point is to estimate $\| Q^{(1)} \|_{\lof}$. Indeed, we have
\begin{equation*}
\begin{split}
\| Q^{(1)} \|_{\lof} &= \sum_{m'' \geq 0}\sup_{\bj''\in\mathcal R_{m''}} \sum_{\bh'' \in \bN}|Q^{(1)}_{\bj'',\bh''}| \\
&\leq \sum_{m'' \geq 0}\sup_{\bj''\in\mathcal R_{m''}} \sum_{\substack{m+m'-1=m''\\ \bj'' \in \mathcal{I}_{m,m'}}}  \sum_{\substack{1\leq \alpha \leq 2m \\ 1\leq \alpha' \leq 2m'\\ (\bj,\bj')=(\Psi_{\alpha,\alpha'}^{m,m'})^{-1}(\bj'')}}
\sum_{\bh,\bh' \in \bN}  |Q_{\bj,\bh} | | Q_{\bj',\bh'}| \\
&\leq\sum_{m'' \geq 0} \sum_{m+m'-1 = m''} (2m)(2m')\left( \sup_{\bj \in \mathcal{R}_m}  \sum_{\bh \in \bN}  |Q_{\bj,\bh} | \right) \left( \sup_{\bj' \in \mathcal{R}_{m'}}  \sum_{\bh' \in \bN}  |Q_{\bj',\bh'} | \right) \\
&\leq  4\mathfrak{m}_Q \mathfrak{m}_{Q'} \|Q\|_{\lof} \|Q'\|_{\lof}.
\end{split}
\end{equation*}

\noindent \underline{$\vartriangleright$\textsf{Estimation of $Q^{(2)}$ and $Q^{(3)}$.}}  By symmetry, without loss of generality, we focus only on $Q^{(2)}$ : 
\begin{equation*}
\begin{split}
Q^{(2)}(z) &= \sum_{\bj\in\mathcal R} \sum_{\bj'\in\mathcal R} z_{\bj} f_{\bj'}^{Q'}(z) \{ f_{\bj}^Q(z) ,  z_{\bj'} \} \\
&= \sum_{\bj\in\mathcal R} \sum_{\bj'\in\mathcal R} \sum_{\alpha' =1}^{\# \bj'} (-\mi \delta(\bj'_{\alpha'}))  f_{\bj'}^{Q'}(z) \left( \partial_{I_{\bj'_{\# \bj'}}}  f_{\bj}^Q(z) \right) z_{\bj} z_{\bj'} .
\end{split}
\end{equation*}
Moreover, we have
$$
\partial_{I_{\bj'_{\alpha'}}}  f_{\bj}^Q(z) =   -  \sum_{\bh \in \bN} Q_{\bj,\bh} \left( \prod_{\alpha = 1}^{\# \bh} \frac{\mi}{\omega_{\bh_{\alpha}}^M(z)} \right) \left(\sum_{\beta=1}^{\# \bh} \frac{\partial_{I_{\bj'_{\alpha'}}}\omega_{\bh_{\beta}}^M(z) }{\omega_{\bh_{\beta}}^M(z)} \right).
$$
Note that since $\omega_{\bh_{\alpha}}^M$ is a linear form on the actions, $\partial_{I_{\bj'_{\alpha'}}} \omega_{\bh_{\alpha}}^M(z) $ is a constant (i.e. it does not depend on $z$). Therefore, setting
$$
Q^{(2)}_{\bj'',\bh''} =- \sum_{\substack{\bj,\bj' \in \mathcal{M} \\ \bj'' = (\bj,\bj')}} \sum_{\alpha' =1}^{\# \bj'}  \delta(\bj'_{\alpha'})  \sum_{\bh \in \bN} \sum_{\beta = 1}^{\# \bh} \sum_{\substack{\bh' \in \bN\\ \bh'' = (\bh,\bh',\bh_\beta)}}Q'_{\bj',\bh'} Q_{\bj',\bh'} \partial_{I_{\bj'_{\alpha'}}}\omega_{\bh_{\beta}}^M(z)  
$$
we have
$$
Q^{(2)}(z) = \sum_{\bj'' \in \mathcal{R}} \left( \sum_{\bh'' \in \bN} Q^{(2)}_{\bj'',\bh''}   \prod_{\alpha = 1}^{\# \bh''} \frac{\mi}{\omega_{\bh_{\alpha}''}^M(z)} \right) z_{\bj''}.
$$
As previously, thanks to this explicit expression, it can be easily checked that the coefficients of $Q^{(2)}$ satisfy all the conditions to belong to $\mathscr{H}_{q''}^{M}$ with $q+q'-1$.  Moreover, it is clear that
$$
\mathfrak m_{Q^{(2)}}\le\mathfrak m_Q+\mathfrak m_{Q'},\quad
\mathfrak n_{Q^{(2)}}\le\mathfrak n_Q+\mathfrak n_{Q'}+1,\ \quad
\mathfrak h_{Q^{(2)}}\le\max\{\mathfrak h_Q,\mathfrak h_{Q'}\}. 
$$
It just remains to estimate $\| Q^{(2)} \|_{\lof}$. Indeed, we have
\begin{equation*}
\begin{split}
\| Q^{(2)} \|_{\lof} &= \sum_{m'' \geq 0}\sup_{\bj''\in\mathcal R_{m''}} \sum_{\bh'' \in \bN}|Q^{(2)}_{\bj'',\bh''}| \\
&\leq \sum_{m'' \geq 0}\sup_{\bj''\in\mathcal R_{m''}} \sum_{\substack{m+m' = m'' \\ \bj'' = (\bj,\bj') \\ \# \bj =2m}}  \sum_{\bh,\bh' \in \bN}  \sum_{\alpha' =1}^{\# \bj'}  |Q_{\bj,\bh} | | Q_{\bj',\bh'}| \sum_{\beta=1}^{\# \bh} |\partial_{I_{\bj'_{\alpha'}}}\omega_{\bh_{\beta}}^M| .
\end{split}
\end{equation*}
Thanks to Lemma \ref{partialomega}, we use the estimate $|\partial_{I_{\bj'_{\alpha'}}}\omega_{\bh_{\beta}}^M| \leq \# \bh_{\beta} \leq \mathfrak{h}_Q $ and according to Remark \ref{rem:ark}, $\# \bh \leq \mathfrak{n}_Q \leq 2\mathfrak{m}_Q$, to get (and to conclude as previously) that
\begin{equation*}
\begin{split}
\| Q^{(2)} \|_{\lof} &\leq (2\mathfrak{m}_Q)  (2\mathfrak{m}_{Q'}) \mathfrak{h}_Q \sum_{m'' \geq 0}\sup_{\bj''\in\mathcal R_{m''}} \sum_{\substack{m+m' = m'' \\ \bj'' = (\bj,\bj') \\ \# \bj =2m}}  \sum_{\bh,\bh' \in \bN}     |Q_{\bj,\bh} | | Q_{\bj',\bh'}| \\
&\leq  4\mathfrak{m}_Q \mathfrak{m}_{Q'}  \mathfrak{h}_Q \|Q\|_{\lof} \|Q'\|_{\lof}.
\end{split}
\end{equation*}
\end{proof}

\begin{lemma}[Rational local flow] \label{Rflow}
Let $q\ge2$, $\gamma \in (0,1)$, $M\geq 1$, $\mathcal S \in\mathscr H_{q}^M$ and define
\[
\varepsilon_2:=\frac14\left(\frac{\mathfrak m_{\mathcal S}^2(1+2\mathfrak h_{\mathcal S})^2}{\gamma^{\mathfrak n_{\mathcal S}+2}}||\mathcal S||_{\lof}\right)^{-\frac1{2(q-1)}}.
\]
There is a smooth map $\Phi_{\mathcal S}:[-1,1]\times\left(\bg(0,\varepsilon_2)\cap\mathfrak U_{\gamma}^{\mathfrak{h}_{\mathcal S},M}\right)\mapsto \bg(0,2\varepsilon_2)\cap\mathfrak U_{\gamma/2}^{\mathfrak{h}_{\mathcal S},M}$ such that
\begin{equation}\label{PhiS2}
	\begin{cases}
\partial_t\Phi_{\mathcal S}^t(z)=\X_{\mathcal S}(\Phi_{\mathcal S}^t(z)),\\
\Phi_{\mathcal S}^0(z)=z.
\end{cases}
\end{equation}
Moreover, given $z\in\bg(0,\varepsilon_2)\cap\mathfrak U_\gamma^{\mathfrak{h}_{\mathcal S},M}$ and $t\in[-1,1]$, one has
\\
(i) $\Phi_{\mathcal S}^t$ is symplectic.\\
(ii) $\Phi_{\mathcal S}^t$ is locally invertible:
\[
\Phi_{\mathcal S}^{-t}\circ\Phi_{\mathcal S}^t(z)=z,\quad\text{if}\quad\Phi_{\mathcal S}^t(z)\in\bg(0,\varepsilon_2)\cap\mathfrak U_\gamma^{\mathfrak{h}_{\mathcal S},M}.
\]
(iii) $\Phi_{\mathcal S}^t$ is close to identity:
\[
||\Phi_{\mathcal S}^t(z)-z||_{\sis}\le \frac{\mathfrak m_{\mathcal S}(1+\mathfrak h_{\mathcal S})2^{2q}}{\gamma^{\mathfrak n_{\mathcal S}+1}}||\mathcal S||_{\lof}||z||^{2q-1}_{\sis}  \le \frac{2^{-2(q-1)}}{1+2\mathfrak h_{\mathcal S}}\gamma \| z\|_{\sis}.
\]
(iv) $\Phi_{\mathcal S}^t$ is locally Lipschitz:
\[
||\md\Phi_{\mathcal S}^t(z)(w)||_{\sis}\le2||w||_{\sis}.
\]
\end{lemma}
\begin{proof}
Mimicking the proof of Lemma \ref{flow} we consider the maximal solution $y\in C^1([0,T);\mathcal{G}_M)$ to the Cauchy problem 
\[
\begin{cases}
	\dot y=\X_{\mathcal S}(y),\\
	y(0)=z\in\bg(0,\varepsilon_2)\cap\mathfrak U_{\gamma}^{\mathfrak{h}_{\mathcal S},M},
\end{cases}
\]
we set $E_0=[0,T)\cap[0,1]$ and 
\[
E_2=\Big\{t\in E_0\,|\,\forall~\tau\in[0,t],\,||y(\tau)||_\sis\le 2||z||_\sis \quad \mathrm{and} \quad y(\tau) \in \overline{\mathfrak U_{\gamma/2}^{\mathfrak{h}_{\mathcal S},M} }\Big\}.
\]
Clearly, it is non-empty, connected and closed in $E_0$ by continuity. 
First, invoking Lemma \ref{Rvectorfield}, one has
\stepcounter{equation}
\begin{align*}
	||y(t)-z||_\sis&\le\int_0^t||\X_{\mathcal S}(y(\tau))||_\sis\,\md\tau\le\int_0^1 2\mathfrak m_{\mathcal S}(1+\mathfrak h_{\mathcal S})\gamma^{-(\mathfrak n_{\mathcal S}+1)}||\mathcal S||_\lof||y(\tau)||_\sis^{2q-1}\,\md\tau\\
	&\le\mathfrak m_{\mathcal S}(1+\mathfrak h_{\mathcal S})2^{2q}\gamma^{-(\mathfrak n_{\mathcal S}+1)}||\mathcal S||_\lof||z||_\sis^{2q-1}\tag{\theequation}\label{close2q}\\
	\stepcounter{equation}
	&\le\frac{4\gamma}{\mathfrak m_{\mathcal S}(1+2\mathfrak h_{\mathcal S})}\left(\frac{2||z||_\sis}{4\varepsilon_2}\right)^{2(q-1)}||z||_\sis
	\le\frac{2^{-2(q-1)}}{1+2\mathfrak h_{\mathcal S}}\gamma||z||_\sis<||z||_\sis.\tag{\theequation}\label{closeId}
\end{align*}
Thanks to Lemma \ref{lipschitzomega} we have, for each $\bj  \in \mathcal{N}$ satisfying  $\# \bj \leq \mathfrak{h}_{\mathcal{S}}$ and $\mu_1( \bj)\leq M$,
\begin{align*}
	|\omega_{\bj}^M(y)-\omega_{\bj}^M(z)|&\le \# \bj \max\big\{||y||_\sis,||z||_\sis\big\}||y-z||_\sis\le 2  \mathfrak{h}_{\mathcal{S}} ||z||_\sis||y-z||_\sis\\
	&\le\frac14\gamma||z||_\sis^2\le\frac14|\omega_{\bj}^M(z)|.
\end{align*}
It follows that $|\omega_{\bj}^M(y) | \geq (3/4) |\omega_{\bj}^M(z)|$.
Together with \eqref{closeId}, it implies that $E_2$ is open, and so, since it is connected that $E_2 = E_0$ (and a fortiori that $T>1$). 

Since $T>1$, we have checked the existence of the flow $\Phi_{\mathcal S}$. It is symplectic because the associated vector field is Hamiltonian and (ii) is an application property of the group property of the flow. Furthermore, estimate \eqref{close2q} implies the property (iii). The proof the estimate (iv) is the same as in the polynomial case (see proof of Lemma \ref{flow}).
\end{proof}

\section{Rational normal form}

\label{sec:rational2}

Given $r\geq 1$, we denote by $H^{(2r)}$ a resonant normal form of order $2r$ of the Schr\"odinger--Poisson Hamiltonian. More precisely, we set
$$
H^{(2r)} := L_2+ L_4+\sum_{m=3}^{r}K_{2m}
$$
where $K_{6},\cdots, K_{2r}$ are given by Proposition \ref{Bnormalform}.

Then, given $M\geq 1$, we define $H^{(2r,M)} : \mathcal{G}_M \to \mathbb{R}$ as the restriction of $H^{(2r)}$ to $\mathcal{G}_M$, that is
\begin{equation}
\label{eq:def_H2rM}
H^{(2r,M)} = H^{(2r)}_{| \mathcal{G}_M}.
\end{equation}
Similarly, we define $L_2^{(M)}$ (resp. $L_4^{(M)}$, resp. $K_{2m}^{(M)}$) as the restriction of $L_2$ (resp. $L_4$, resp. $K_{2m}$) to $\mathcal{G}_M$. Note that, by construction $K_{2m}^{(M)} \in \mathscr{H}_{2m}^M$ and that it satisfies $\mathfrak{n}_{K_{2m}^{(M)}} = \mathfrak{h}_{K_{2m}^{(M)}}=0$, $\mathfrak{m}_{K_{2m}^{(M)}} = m$ and
\begin{equation}
\label{eq:what_we_did_before}
\| K_{2m}^{(M)}\|_{\lof} \le C^{2m-3} m^{2(m-2)}
\end{equation}
where $C>0$ is a universal constant (i.e. independent of $r$ and $M$).

\begin{definition} A rational Hamiltonian $L \in \mathscr H_{q}^M$ is said \emph{integrable} and we denote it by  $L \in \mathscr H_{q}^{M,\mathrm{Int}}$ if it depends only on the actions, i.e. 
$$
\forall \bj \in \mathcal{N},\forall \bh \in \bN, \quad L_{\bj,\bh} = 0. 
$$

\end{definition}

\begin{theorem}
\label{Rnormalform} Let $r\ge 3$, $M\ge 3$, $\gamma \in (0,1)$, and $H^{(2r,M)}$ be the polynomial defined by \eqref{eq:def_H2rM}. There exists a universal constant $\sfc>1$ (independent of $r,M,\gamma$) such that setting
\begin{equation}
\label{eq:varrho}
\varrho=\gamma^\frac32 \sfc^{-1}  r^{-4}
\end{equation}
there exist two $C^\infty$ symplectic maps $\varphi_{(0)}$ and $\varphi_{(1)}$ making the following diagram to commute
\begin{equation}
\label{diag2}
\xymatrixcolsep{5pc} \xymatrix{ \mathbb{B}_M(0, \varrho )\cap \mathfrak{U}_{\gamma}^{6r,M}  \ar[r]^{ \varphi_{(0)} }
 \ar@/_1pc/[rr]_{ \mathrm{id}_{\mathcal{G}_M} } & \mathbb{B}_M(0, 2\varrho )\cap \mathfrak{U}_{\gamma/2}^{6r,M}   \ar[r]^{ \hspace{1cm}  \varphi_{(1)}}  & 
\mathcal{G}_M}
\end{equation}
such that $H^{(2r,M)}\circ\varphi_{(1)}$ admits the following decomposition on $\mathbb{B}_M(0,\varrho) \cap \mathfrak{U}_{\gamma}^{6r,M}$
\begin{equation}\label{decomp2}
	H^{(2r,M)}\circ\varphi_{(1)}=L_2^{(M)}+L_4^{(M)}+\sum_{q=3}^{r}L_{2q}^{(M)}+ \Upsilon
\end{equation}
where $L_{2q}^{(M)}\in H_{q}^{M,\mathrm{Int}}$ is an integrable Hamiltonian of order $2q$ 
and $\Upsilon:\mathbb{B}_M(0, 2\varrho )\cap \mathfrak{U}_{\gamma/2}^{6r,M} \mapsto\mathbb R$ is a smooth function which is a remainder term of order $2r+2$, i.e.
\begin{gather*}
||\X_{\Upsilon}(z)||_{\sis}\le \sfc^{r}  r^{10r}   \gamma^{- 2r + 3} ||z ||_{\sis}^{2r+1}.
\end{gather*}
Furthermore, the maps $\varphi_{(\iota)}:\bg(0,\varrho+\iota\varrho)\cap  \mathfrak{U}_{\gamma - \frac{\iota}2 \gamma}^{6r,M}\mapsto\mathcal G_M$, $\iota=0,1$, are close to the identity
\begin{equation}
\label{eq:phi-id}
||\varphi_{(\iota)}(z)-z||_{\sis}\le \sfc \gamma^{-2} ||z||_{\sis}^3
\end{equation}
and their differential are not too large: for all $w\in \mathcal{G}_M$, we have
\begin{equation}
\label{eq:Rnf-d}
\|\mathrm{d}\varphi_{(\iota)}(z)(w) \|_{\sis} \le 2^{r -2} \| w \|_{\sis}\,.
\end{equation}
\end{theorem}
\begin{proof} We split the proof in five steps. 

\medskip

\noindent \underline{$\vartriangleright$\textsf{(1) Induction.}}  We are going to prove by induction on $\mr\in \llbracket 1,r-1 \rrbracket$ that there exists a universal constant $\sfc > 1$ such that setting
\[
\varrho=\frac{\gamma^{3/2}}{2^{9}\sfc^2 \mr^3 r},
\]
there exist two $C^\infty$ symplectic maps $\varphi_{(0)}$ and $\varphi_{(1)}$ making the diagram 
\begin{equation}
\label{diag3}
\xymatrixcolsep{5pc} \xymatrix{ \mathbb{B}_M(0, \varrho )\cap \mathfrak{U}_{\nu_{\mr} \gamma}^{6r,M}  \ar[r]^{ \varphi_{(0)} }
 \ar@/_1pc/[rr]_{ \mathrm{id}_{\mathcal{G}_M} } & \mathbb{B}_M(0, 2\varrho )\cap \mathfrak{U}_{\nu_{\mr}\gamma/2}^{6r,M}   \ar[r]^{ \hspace{1cm}  \varphi_{(1)}}  & 
\mathcal{G}_M}
\end{equation}
 to commute, where\footnote{Note that $\nu_{r-1} = 1$.} 
 $$
\nu_{\mr} = 1-  \sum_{i=\mr}^{r-2} 2^{-2i} > 1/2
 $$
 and such that $H^{(2r,M)}\circ\varphi_{(1)}$ admits the decomposition \eqref{decomp2} where $L_{2q}$ (i.e. $L_{2q}^{(M)}\in H_{q}^{M,\mathrm{Int}}$) is integrable for $q\le\mr+1$ and satisfies
\begin{equation}
\label{eq:fun_fun}
 \mathfrak{m}_{L_{2q}^{(M)}} \leq 3q - 6, \quad (1/2)\mathfrak{h}_{L_{2q}^{(M)}}  \leq  3\mr - 3 \quad \mathrm{and} \quad \mathfrak{n}_{L_{2q}^{(M)}} \leq 2q-6,
\end{equation}
and
\begin{equation}
\label{eq:L2qM}
||L_{2q}^{(M)}||_{\lof}\le \sfc^{4q-9}q^{2(q-2)}\min\{q,\mr\}^{4(q-3)},\quad\forall\,3\le q\le r,\\
\end{equation}
and the remainder $\Upsilon:\mathbb{B}_M(0, 2\varrho )\cap \mathfrak{U}_{\nu_\mr\gamma/2}^{6r,M} \mapsto\mathbb R$ is a smooth function satisfying 
$$
\| \X_{\Upsilon}(z) \|_{\sis} \le 4^{\mr-1} \sfc^{9r}  r^{10r}   \gamma^{- 2r + 3} ||z ||_{\sis}^{2r+1} \prod_{i= 1}^{\mr-1}  (1+2^{-2i})^{2r+1}.
$$
Furthermore, the map $\varphi_{(\iota)}:\bg(0,2^{\iota}\varrho)\cap  \mathfrak{U}_{2^{-\iota}\nu_{\mr} \gamma}^{6r,M}$ is close to the identity
\[
||\varphi_{(\iota)}(z)-z||_{\sis}\le2^{30} \gamma^{-2} \sfc^3 \| z\|_{\sis}^3 \sum_{i=1}^{\mr-1} 2^{-2i},\quad\iota=0,1,
\]
and their differential are not too large, that is, for all $w\in \mathcal{G}_M$, we have
$$
\|\mathrm{d}\varphi_{(\iota)}(z)(w) \|_{\sis} \le 2^{\mr -1} \| w \|_{\sis}.
$$

 First we note that the case $\mr=1$ is trivial. Indeed, in view of \eqref{eq:def_H2rM} and \eqref{eq:what_we_did_before}, it is enough to set $L_{2q}^{(M)}=K_{2q}^{(M)}$ and to ensure that $\sfc>0$ is large enough to have
 $$
 \forall q\ge 3, \quad C^{2q-3} q^{2(q-2)} \leq  \sfc^{4q-9}q^{2(q-2)} .
 $$

From now, we assume that the property holds at step $\mr<r-1$ and we are going to prove it at step $\mr+1$. To simplify notations, we denote with a superscript $\sharp$ the maps and indices corresponding to the subsequent step $\mr+1$, such as $L_{2q}^{(M),\sharp}$ and $\varphi_{(1)}^\sharp$.

\medskip

\noindent
\underline{$\vartriangleright$\textsf{(2) Resolution of homological equation.}} Now we are going to remove the non integrable terms of $L_{2(\mr+2)}^{(M)}$ by solving the following cohomological equation:
\begin{equation}\label{homoequ2}
\{L_4^{(M)},\mathcal S\}+L_{2(\mr+2)}^{(M)}=L_{2(\mr+2)}^{(M),\sharp} \in \mathscr{H}_{\mr +2}^{M,\mathrm{Int}}.
\end{equation}
We solve the above equation by setting
\[
L_{2(\mr+2)}^{(M),\sharp} =\sum_{\bj\in \mathrm{Int}} f_{\bj}^{L_{2(\mr+2)}^{(M)}} z_{\bj} 
\quad\text{and}\quad
\mathcal S=\sum_{ \bj\in \bN}\frac{\mi f_{\bj}^{L_{2(\mr+2)}^{(M)}}}{\omega_{ \bj}^M} z_{\bj}.
\]
or more precisely
$$
L_{2(\mr+2),\bj,\bh}^{(M),\sharp} := \mathbbm{1}_{\bj \in \mathrm{Int}} L_{2(\mr+2),\bj,\bh}^{(M)}
$$
and the non zero coefficients of $\mathcal{S}$ are of the form
$$
\mathcal{S}_{\bj,\bh} = L_{2(\mr+2),\bj,\bh}^{(M)} \quad \mathrm{with} \quad \bj \in \mathcal{N}.
$$ 

Obviously, by the induction hypothesis, we have the estimates
\begin{equation}\label{homoestimate2}
\max(||\mathcal S||_{\lof} , ||L_{2(\mr+2)}^{(M),\sharp}||_{\lof})\le||L_{2(\mr+2)}^{(M)}||_{\lof}\le \sfc^{4\mr-1}(\mr+2)^{2\mr}\mr^{4(\mr-1)}.
\end{equation}
Moreover, we have
$$
\mathfrak m_{\mathcal{S}} \leq \mathfrak m_{L_{2(\mr+2)}^{(M)}}, \quad \mathfrak h_{\mathcal{S}} \leq \max(2\mathfrak m_{\mathcal{S}} , \mathfrak h_{L_{2(\mr+2)}^{(M)}} ), \quad  \mathfrak n_{\mathcal{S}} \leq \mathfrak n_{L_{2(\mr+2)}^{(M)}} +1
$$
and so, by induction hypothesis, one has
\begin{equation}
\label{mnh}
\mathfrak m_{\mathcal{S}} \leq 3\mr, \quad \mathfrak h_{\mathcal{S}} \leq  6\mr, \quad  \mathfrak n_{\mathcal{S}} \leq 2\mr -1.
\end{equation}

\noindent
\underline{$\vartriangleright$\textsf{(3) The new variables.}} 
We recall that the Hamiltonian flow $\Phi_{\mathcal S}^t$, given by Lemma \ref{Rflow}, is well-defined for $|t|\le1$ on $\bg_M(0,\varepsilon_2)\cap\mathfrak U_{\gamma/4}^{\mathfrak{h}_{\mathcal{S}},M}$ where 
\stepcounter{equation}
\begin{align*}
\varepsilon_2&=\frac14\left(\frac{\mathfrak m^2_{\mathcal{S}}(1+2\mathfrak h_{\mathcal{S}})^2}{(\gamma/4)^{\mathfrak n_{\mathcal{S}}+2}}||\mathcal S||_\lof\right)^{-\frac1{2\mr}}\ge
\frac14(\gamma/4)^{1+\frac1{2\mr}}\big(3^2\times13^2\times2^{4\mr}\times \sfc^{4\mr-1}\mr^{6\mr}\big)^{-\frac1{2\mr}}\\
&\ge\frac{\gamma^\frac32}{2^6\sfc ^2\mr^3}\ge 2\varrho,\qquad\udl{4\times 3^2\times13^2}.\tag{\theequation}\label{radius2}
\end{align*}
Moreover, since $\Phi_{\mathcal S}^t$ is close to the identity, if $z\in \bg_M(0,2\varrho)\cap\mathfrak U_{\gamma/4}^{\mathfrak{h}_{\mathcal{S}},M}$ we have
\begin{equation}
 \label{eq:good1}
\begin{split}
||\Phi_{\mathcal S}^{t}(z)-z||_\sis&\le  \frac{\mathfrak m_{\mathcal S}(1+\mathfrak h_{\mathcal S})2^{2(\mr+2)}}{(\gamma/4)^{\mathfrak n_{\mathcal S}+1}}||\mathcal S||_{\lof}||z||_\sis^{2\mr+1}\\
&\le  3\mr (1+6\mr )2^{2(2\mr +2)}(\gamma/4)^{-2\mr}\sfc^{4\mr-1}(\mr+2)^{2\mr}\mr^{4(\mr-1)}(2\varrho)^{2(\mr-1)}||z||_\sis^3\\
&\le          2^{12\mr +9}       \gamma^{-2\mr}  \sfc^{4\mr-1}   \mr^{6\mr -2}           \big( \frac{\gamma^{3/2}}{2^{8}\sfc^2 \mr^3 r}  \big)^{2(\mr-1)}||z||_\sis^3\\
&\le      2^{-4\mr +25}     \gamma^{-2\mr}  \sfc^{4\mr-1} \mr^4  \mr^{6(\mr -1)}  \sfc^{-4(\mr-1)} \gamma^{3(\mr-1)} \mr^{-6(\mr-1) }     ||z||_\sis^3 \\
&\le 2^{-2\mr +29}  \gamma^{\mr - 3} \sfc^3  ||z||_\sis^3
\end{split}
\end{equation}
and so a fortiori
\begin{equation}
\label{eq:good2}
||\Phi_{\mathcal S}^{t}(z)-z||_\sis \le  2^{-2\mr +29}  \gamma^{\mr - 3} \sfc^3 (2\varrho)^2 ||z||_\sis \le \gamma 2^{-2\mr +13}   \sfc^{-1} \mr^{-6} r^{-2}  ||z||_\sis .
\end{equation}

At this step, we aim at proving that the new change of variables
\[
\varphi_{(1)}^\sharp:=\varphi_{(1)}\circ\Phi_{\mathcal S}^1,\quad\varphi_{(0)}^\sharp:=\Phi_{\mathcal S}^{-1}\circ\varphi_{(0)}.
\]
are well defined and enjoy the expected properties. 
We set 
$$
\varrho^\sharp =  \frac{\gamma^{3/2}}{2^{9} \sfc^2 (\mr+1)^3 r} = \left(\frac\mr{1+\mr}\right)^3\varrho.
$$

\noindent $\bullet$ First let us prove that $\varphi_{(1)}^\sharp$ is well defined, i.e. that $\Phi_{\mathcal S}^1$ maps $\mathbb{B}_M(0, 2\varrho^\sharp )\cap \mathfrak{U}_{\nu_{\mr+1} \gamma/2}^{6r,M}$ on $\mathbb{B}_M(0, 2\varrho )\cap \mathfrak{U}_{\nu_{\mr}\gamma/2}^{6r,M}$. Indeed, we have
\begin{equation*}
\begin{split}
\| \Phi_{\mathcal S}^1(z)\|_{\sis} &\le \| z \|_{\sis} + ||\Phi_{\mathcal S}^{t}(z)-z||_\sis 
 \le   \| z \|_{\sis} (1 +  \gamma 2^{-2\mr +13}   \sfc^{-1} \mr^{-6} r^{-2} ) \\
& \le \| z \|_{\sis} (   1 + \mr^{-6} ) \qquad\udl{2^{13}} \\
&\le 2\rho \left(\frac\mr{1+\mr}\right)^3 (   1 + \mr^{-6}  ) < 2\rho 
\end{split}
\end{equation*}
and, given $\bj \in \mathcal{N}$ such that $\# \bj \le 6r$ and $\mu_1(\bj) \le M$, by Lemma \ref{lipschitzomega} (and using that $1 + \mr^{-6} \leq 2$)
\begin{equation*}
\begin{split}
|\omega_{\bj}^M(\Phi_{\mathcal S}^1(z)) -\omega_{\bj}^M(z)| &\le (6r) (2 \|z\|_{\sis}) ( \gamma 2^{-2\mr +13}   \sfc^{-1} \mr^{-6} r^{-2}  ||z||_\sis ) \\
&\le 2^{-2\mr-1} \gamma \|z\|_{\sis}^2 \qquad\udl{12 \times 2^{14}}.
\end{split}
\end{equation*}
and so
$$
|\omega_{\bj}^M(\Phi_{\mathcal S}^1(z))|>(  \nu_{\mr+1}- 2^{-2\mr} ) (\gamma/2) \|z\|_{\sis}^2 =  (1-  \sum_{i=\mr+1}^{r-2} 2^{-i}  - 2^{-2\mr})(\gamma/2) \|z\|_{\sis}^2  =  \nu_{\mr}  (\gamma/2) \|z\|_{\sis}^2,
$$
that is $\Phi_{\mathcal S}^1(z) \in \mathfrak{U}_{\nu_{\mr}\gamma/2}^{6r,M}$.

\noindent $\bullet$  We have to prove that $\varphi_{(0)}^\sharp$ is well defined, i.e. that $\varphi_{(0)}$ maps $\mathbb{B}_M(0, \varrho^\sharp )\cap \mathfrak{U}_{\nu_{\mr+1} \gamma}^{6r,M}$ on $\mathbb{B}_M(0, 2\varepsilon_2 )\cap \mathfrak{U}_{\gamma/4}^{6r,M}$. This follows directly of the induction hypothesis because, as we have seen, $2\varepsilon_2 \geq 2 \varrho$ and $\varrho^{\sharp} < \varrho $.

\noindent $\bullet$ The facts that $\varphi_{(0)}^\sharp$ and $\varphi_{(1)}^\sharp$ are symplectic and the estimate on the differential  are obvious by composition.

\noindent $\bullet$ We have to prove that $\varphi_{(1)}^\sharp$ is close to the identity. Indeed, if $z\in \mathbb{B}_M(0, 2\varrho^\sharp )\cap \mathfrak{U}_{\nu_{\mr+1} \gamma/2}^{6r,M}$, by \eqref{eq:good2} and convexity, provided that $\sfc \ge 14 \times 2^{13}$, we have
$$
\| \Phi_{\mathcal S}^1(z) \|_{\sis}^3  \le (1+ \gamma 2^{-2\mr +13}   \sfc^{-1} \mr^{-6} r^{-2})^3  ||z||_\sis^3 \le (1+ 14^{-1} 2^{-2\mr} )^3 ||z||_\sis^3 \le (1+ 2^{-2\mr-1}) ||z||_\sis^3 .
$$
and so 
\begin{equation*}
\begin{split}
\| \varphi_{(1)}^\sharp(z) -z \|_{\sis} &\le \| \varphi_{(1)}( \Phi_{\mathcal S}^1(z)   ) - \Phi_{\mathcal S}^1(z) \|_{\sis} + \| \Phi_{\mathcal S}^1(z) - z \|_{\sis} \\
&\mathop{\le}^{\eqref{eq:good1}}  2^{30} \gamma^{-2} \sfc^3 \sum_{i=1}^{\mr-1} 2^{-2i}   \| \Phi_{\mathcal S}^1(z) \|_{\sis}^3 +2^{-2\mr +29}  \gamma^{\mr - 3} \sfc^3  ||z||_\sis^3 \\
&\le 2^{30} \gamma^{-2} \sfc^3 \| z\|_{\sis}^3 (\sum_{i=1}^{\mr-1} 2^{-2i}  + 2^{-2\mr-1} + 2^{-2\mr-1}) = 2^{30} \gamma^{-2} \sfc^3 \| z\|_{\sis}^3 \sum_{i=1}^{\mr} 2^{-2i}  .
\end{split}
\end{equation*}

\noindent $\bullet$ We have to prove that $\varphi_{(0)}^\sharp$ is close to the identity. Indeed, if $z\in \mathbb{B}_M(0, \varrho^\sharp )\cap \mathfrak{U}_{\nu_{\mr+1} \gamma}^{6r,M}$, we have
\begin{equation}
\label{eq:coucou}
\begin{split}
 \| \varphi_{(0)}(z) \|_{\sis} &\le (1+ \varrho^2  2^{30} \gamma^{-2} \sfc^3 \sum_{i=1}^{\mr-1} 2^{-2i}  ) \|z\|_{\sis} \\
 &\le (1+   2^{12} \gamma \sfc^{-1}  (1/3) r^{-2} ) \|z\|_{\sis}  \ge 2^{1/3}  \|z\|_{\sis}  ,\qquad\udl{2^{20}}
\end{split}
\end{equation}
and so as expected, using \eqref{eq:good1}, we have
\begin{equation}
\label{eq:lala}
\begin{split}
\| \varphi_{(0)}^\sharp(z) -z \|_{\sis} &\le \| \Phi_{\mathcal S}^{-1} ( \varphi_{(0)}(z)) -  \varphi_{(0)}(z) \|_{\sis} +  \| \varphi_{(0)}(z) - z\|_{\sis} \\
&\le 2^{-2\mr +29}  \gamma^{\mr - 3} \sfc^3 || \varphi_{(0)}(z) ||_\sis^3  +  2^{30} \gamma^{-2} \sfc^3 \| z\|_{\sis}^3 \sum_{i=1}^{\mr-1} 2^{-2i} \\
&\le  2^{30} \gamma^{-2} \sfc^3 \| z\|_{\sis}^3 \sum_{i=1}^{\mr} 2^{-2i} .
\end{split}
\end{equation}

\noindent $\bullet$ Finally, we have to prove that  $\varphi_{(0)}^\sharp$ maps $\mathbb{B}_M(0, \varrho^\sharp )\cap \mathfrak{U}_{\nu_{\mr +1} \gamma}^{6r,M}$ on $\mathbb{B}_M(0, 2\varrho^\sharp )\cap \mathfrak{U}_{\nu_{\mr +1}\gamma /2}^{6r,M}$. Note that it directly implies by construction that the diagram \eqref{diag3} commutes.
First, using that $\varphi_{(0)}^\sharp$ is close to the identity (see \eqref{eq:lala})  and proceeding as in \eqref{eq:coucou}, we get that $ \| \varphi_{(0)}^{\sharp}(z) \|_{\sis} \le 2^{1/3}  \|z\|_{\sis}$ and so that it maps $\mathbb{B}_M(0, \varrho^\sharp )\cap \mathfrak{U}_{\nu_{\mr +1} \gamma}^{6r,M}$ on $\mathbb{B}_M(0, 2\varrho^\sharp )$. Now let $z\in \mathbb{B}_M(0, \varrho^\sharp )\cap \mathfrak{U}_{\nu_{\mr +1} \gamma}^{6r,M}$ and $\bj \in \mathcal{N}$ such that $\# \bj \le 6r$ and $\mu_1(\bj) \le M$. By Lemma \ref{lipschitzomega}, and proceeding as in \eqref{eq:coucou}, we have
\begin{equation*}
\begin{split}
|\omega_{\bj}^M(\varphi_{(0)}^\sharp(z)) -\omega_{\bj}^M(z)| &\le (6r) (2 \|z\|_{\sis}) ( 2^{12} \gamma \sfc^{-1}  (1/3) r^{-2} \|z\|_{\sis}  ) \\
&\le   (\gamma /4) \|z\|_{\sis}^2 \qquad\udl{  2^{16}}.
\end{split}
\end{equation*}
and so (using that $ \nu_{\mr+1} > 1/2$)
$$
|\omega_{\bj}^M(\varphi_{(0)}^\sharp(z))|>(  \nu_{\mr+1} - \frac14 ) \gamma \|z\|_{\sis}^2 \ge \nu_{\mr+1}  (\gamma/2) \|z\|_{\sis}^2,
$$
that is $\varphi_{(0)}^\sharp(z)\in \mathfrak{U}_{\nu_{\mr +1}\gamma /2}^{6r,M}$.

\noindent
\underline{$\vartriangleright$\textsf{(4) The new expansion.}} Since $\Phi_{\mathcal S}^t$ is the Hamiltonian flow of $\mathcal S$, if $z \in \mathbb{B}_M(0, 2\varrho^\sharp )\cap \mathfrak{U}_{\nu_{\mr +1}\gamma /2}^{6r,M}$, we have
\[
\frac{\md}{dt}L_{2p}^{(M)}\circ\Phi_{\mathcal S}^t(z)=\{L_{2p}^{(M)},\mathcal S\}\circ\Phi_{\mathcal S}^t(z),\quad 1\le p\le r.
\]
Therefore, by Taylor expansion we obtain that 
\[
L_{2p}^{(M)}\circ\Phi_{\mathcal S}^1(z)=\sum_{k=0}^{k_p^*}\frac1{k!}\ad_{\mathcal S}^k(L_{2p}^{(M)})(z)+\int_0^1\frac{(1-t)^{k_p^*}}{k_p^*!}\ad_{\mathcal S}^{k_p^*+1}(L_{2p}^{(M)})\circ\Phi_{\mathcal S}^t(z)\md t,\quad2\le p\le r.
\]
Notice that $L_2^{(M)} \circ\Phi_{\mathcal S}^1=L_2^{(M)}$ since $\{L_2^{(M)},\mathcal S\}\equiv0$ and that $k_p^*$ denotes the largest integer such that $k\mr+p\le r$, i.e 
\begin{equation}\label{kp}
k_p^*\mr+p\le r\quad\text{and}\quad (k_p^*+1)\mr+p>r.
\end{equation}
Recalling that $\varphi_{(1)}^{\sharp} = \varphi_{(1)} \circ \Phi_{\mathcal S}^1$, we get
\begin{align*}
H^{(2r,M)}\circ\varphi_{(1)}^{\sharp}&=\sum_{p=1}^rL_{2p}^{(M)}\circ\Phi_{\mathcal S}^1+\Upsilon\circ\Phi_{\mathcal S}^1\\
&=L_2^{(M)}+\sum_{\substack{1\le k\le k_p^*\\2\le p\le r}}\frac1{k!}\ad_{\mathcal S}^k(L_{2p}^{(M)})+\sum_{2\le p\le r}\int_0^1\frac{(1-t)^{k_p^*}}{k_p^*!}\ad_{\mathcal S}^{k_p^*+1}(L_{2p}^{(M)})\circ\Phi_{\mathcal S}^t\,\md t+\Upsilon\circ\Phi_{\mathcal S}^1\\
&=:L_2^{(M)}+\sum_{q=2}^rL_{2q}^{(M),\sharp}+\Upsilon^\sharp,
\end{align*}
where
\begin{gather}
L_{2q}^{(M),\sharp}=\sum_{\substack{k\mr+p=q\\ k\ge0,p\ge2}}\frac1{k!}\ad_{\mathcal S}^k(L_{2p}^{(M)}),\quad 2\le q\le r,\label{L2q}\\
\Upsilon^{\sharp}=\Upsilon \circ\Phi_{\mathcal S}^1+\sum_{2\le p\le r}\int_0^1\frac{(1-t)^{k_p^*}}{k_p^*!}\ad_{\mathcal S}^{k_p^*+1}(L_{2p}^{(M)})\circ\Phi_{\mathcal S}^t\,\md t.\label{Ro}
\end{gather}
We will focus on the new remainder $\Upsilon^{\sharp}$ at the next step of the proof. For the moment we only focus on $L_{2q}^{(M),\sharp}$. As previously, we decompose
\begin{equation*}
L_{2q}^{(M),\sharp}=\sum_{\substack{k\mr+2=q\\ k\ge0}}\frac1{k!}\ad_{\mathcal S}^k(L_{4}^{(M)})+\sum_{\substack{k\mr+p=q\\ k\ge0,p\ge3}}\frac1{k!}\ad_{\mathcal S}^k(L_{2p}^{(M)}),
\end{equation*}
which implies that by the cohomological equation \eqref{homoequ2}
\begin{equation}\label{smallq}
\begin{gathered}
L_{2q}^{(M),\sharp}:=L_{2q}^{(M)},\quad\forall\,1\le q \le \mr+1,\\
L_{2(\mr+2)}^{(M),\sharp}=\ad_{\mathcal S}(L_4)+L_{2(\mr+2)}^{(M)}.
\end{gathered}
\end{equation}
Therefore, from now we only focus on $q > \mr+2$. Using that $\mathcal{S}$ solves the cohomological equation \eqref{homoequ2}, we get for  $q > \mr+2$
\begin{equation*}
L_{2q}^{(M),\sharp}=\sum_{\substack{k\mr+p=q\\ k\ge0,p\ge3}}\frac1{k!}\ad_{\mathcal S}^k(L_{2p}^{(M),q}),
\end{equation*}
where $L_{2p}^{(M),q} :=  L_{2p}^{(M)}$ if $p\neq \mr+2$ or $p = \mr+2$ and $\mr $ does not divide $q-\mr -2$, else if there exists an integer $k$ such that $ k\mr+\mr+2=q  $ then we  have set
\begin{equation}
\label{eq:very_nice}
L_{2(\mr+2)}^{(M),q} :=  \frac1{k+1} L_{2(\mr+2)}^{(M)} + \frac{k}{k+1}  L_{2(\mr+2)}^{(M),\sharp}.
\end{equation}
Note that by construction  $L_{p}^{(M),q} \in \mathscr{H}_p^M$, $\|L_{p}^{(M),q}\|_{\lof} \leq \|L_{p}^{(M)}\|_{\lof}$, $\mathfrak{m}_{L_{p}^{(M),q}} = \mathfrak{m}_{L_{p}^{(M)}}$, $\mathfrak{n}_{L_{p}^{(M),q}} = \mathfrak{n}_{L_{p}^{(M)}}$ and $\mathfrak{h}_{L_{p}^{(M),q}} = \mathfrak{h}_{L_{p}^{(M)}}$. 

The Hamiltonian $L_{2q}^{(M),\sharp}$ being defined as a sum of Poisson brackets of rational fractions, thanks to Lemma  \ref{Rpoissonbracket}, we note that $L_{2q}^{(M),\sharp}$ is also a rational fraction of order $2q$, that is  $L_{2q}^{(M),\sharp} \in \mathscr{H}_q^M$. Moreover, this lemma provides quantitative bounds. First, we have using \eqref{mnh} and the induction hypothesis
$$
\mathfrak{m}_{L_{q}^{(M),\sharp}} \leq \max_{k\mr+p=q}  k \mathfrak{m}_{\mathcal S} +  \mathfrak{m}_{L_{p}^{(M),q} } \leq \max_{k\mr+p=q}  3p - 6 + 3 k \mr = 3q-6. 
$$
$$
\mathfrak{n}_{L_{q}^{(M),\sharp}} \leq \max_{k\mr+p=q}  k (\mathfrak{n}_{\mathcal S} +1)+  \mathfrak{n}_{L_{p}^{(M),q} } \leq \max_{k\mr+p=q} 2k\mr + 2p - 6 = 2q-6. 
$$
$$
\mathfrak{h}_{L_{q}^{(M),\sharp}} \leq \max_{k\mr+p=q} \max( \mathfrak{h}_{L_{p}^{(M)}} ,  \mathfrak{h}_{\mathcal{S}} ) \leq 6\mr.
$$
Then we focus on estimating $||L_{2q}^{(M),\sharp}||_\lof$. Applying the triangular inequality, we get
\begin{equation}\label{simpleL2q}
||L_{2q}^{(M),\sharp}||_\lof\le ||L_{2q}^{(M)}||_\lof+\sum_{\substack{k\mr+p=q\\ k\ge1,p\ge3}}\frac1{k!}||\ad_{\mathcal S}^k(L_{2p}^{(M),q})||_\lof.
\end{equation}
Using \eqref{smallq}, to prove \eqref{eq:L2qM}, it remains to verify that for $\mr+2<q\le r$,
\begin{equation}\label{L2qEstimate}
||L_{2q}^{(M),\sharp}||_\lof\le \sfc^{4q-9}q^{2(q-2)}(\mr+1)^{4(q-3)},\quad\forall\,\mr+3\le q\le r.
\end{equation}
Firstly, we have by the induction hypothesis 
\begin{equation}\label{L2qEstimate1}
	\frac{||L_{2q}^{(M)}||_\lof}{\sfc^{4q-9}q^{2(q-2)}(\mr+1)^{4(q-3)}}\le\left(\frac{\mr}{\mr+1}\right)^{4(q-3)}\le\left(\frac{\mr}{\mr+1}\right)^{\mr+1}\le e^{-1}.
\end{equation}
Then, using the quantitative estimate given by Lemma \ref{Rpoissonbracket}, we have
\begin{align*}
&\sum_{\substack{k\mr+p=q\\ k\ge1,p\ge3}}\frac1{k!}||\ad_{\mathcal S}^k(L_{2p}^{(M)})||_\lof\\ \le&\sum_{\substack{k\mr+p=q\\ k\ge1,p\ge3}}\frac1{k!}||L_{2p}^{(M),q}||_\lof\left(4\mathfrak m_{\mathcal{S}}(1+2\max\{\mathfrak h_{\mathcal{S}},\mathfrak h_{L_{2p}^{(M),q}}\})||\mathcal S||_\lof\right)^k\prod_{i=0}^{k-1}(\mathfrak m_{L_{2p}^{(M),q}}+i\mathfrak m_{\mathcal{S}})\\
\le&\sum_{\substack{k\mr+p=q\\ k\ge1,p\ge3}}\frac1{k!}||L_{2p}||_\lof\left(12\mr(1+12\mr)||\mathcal S||_\lof\right)^k\prod_{i=0}^{k-1}3(p+i\mr)\\
\le&\sum_{\substack{k\mr+p=q\\ k\ge1,p\ge3}}\frac1{k!}\sfc^{4p-9}p^{2(p-2)}\mr^{4(p-3)}\left(12\times13\times\mr^2\sfc^{4\mr-1}(\mr+2)^{2\mr}\mr^{4(\mr-1)}\right)^k(3q)^{k}\\
\le&\sum_{\substack{k\mr+p=q\\ k\ge1,p\ge3}}\frac1{k!}(36\times13)^k\sfc^{4p+4k\mr-9-k}q^{2(p-2)}\mr^{4(p-3)}\left(q^{2\mr}\mr^{4\mr-2}\right)^kq^{k}\\
=&\sum_{\substack{k\mr+p=q\\ k\ge1,p\ge3}}\frac1{k!}\left(\frac{36\times13}{\sfc}\right)^k\sfc^{4q-9}q^{2(q-2)}\mr^{4(q-3)}q^{k}\mr^{-2k}
\le\sfc^{4q-9}q^{2(q-2)}(\mr+1)^{4(q-3)} \delta_{q,p,\mr,C}.
\end{align*}
where
$$
\delta_{q,p,\mr,C} := \sum_{k\ge1}\frac1{k!}\left(\frac{36\times13}{\sfc}\right)^kq^{k}\left(\frac{1+\mr}{\mr}\right)^{-4(q-3)}\mr^{-2k}
$$
So, to prove the expected bound \eqref{L2qEstimate}, thanks to \eqref{L2qEstimate1}, we only have to prove that $\delta_{q,p,\mr,C} \leq 1-e^{-1}$. Indeed, we have
\begin{align*}
\delta_{q,p,\mr,C} \le&\sum_{k\ge1}\frac{2^{12}}{k!}\left(\frac{36\times13}{\sfc}\right)^kq^{k}\left(1+\frac1\mr\right)^{-4q}\mr^{-2k}\\
\le&\sum_{k\ge1}2^{12}\left(\frac{36\times13}{\sfc}\right)^k\frac1{k!}\left(\frac ke\right)^k\left(\log\big[(1+\mr^{-1})^4\big]\right)^{-k}\mr^{-2k}\quad\underline{~\text{by Lemma \ref{mnelog}}~}\\
\le&\sum_{k\ge1}2^{12}\left(\frac{36\times13}{4\sfc}\right)^k(1+\mr)^k\mr^{-2k}
\le\sum_{k\ge1}2^{12}\left(\frac{36\times13}{2\sfc}\right)^k\\
\le&\frac{2^{12}\times13\times36}{\sfc}
\le\frac12,\quad\underline{~\text{provided }\sfc\ge2^{13}\times13\times36.~}
\end{align*}

\noindent
\underline{$\vartriangleright$\textsf{(5) The new remainder.}} Finally, we just have to estimate the new remainder term $\Upsilon^{\sharp}$ given by \eqref{Ro}. At this step we fix $z\in \mathbb{B}_M(0, 2\varrho^\sharp )\cap \mathfrak{U}_{\nu_{\mr +1}\gamma /2}^{6r,M} $. First, note that we can simplify a little bit the expression of  $\Upsilon^{\sharp}$ by removing the term of the sum associated with $p=2$ using the Hamiltonians $L_{2p}^{(M),q}$ (defined by \eqref{eq:very_nice}) :
$$
\Upsilon^{\sharp}=\Upsilon \circ\Phi_{\mathcal S}^1+\sum_{3\le p\le r}\int_0^1 (1-t)^{k_p^*} (k_p^*+1)  Q_{p}\circ \circ\Phi_{\mathcal S}^t\,\md t.
$$
where $q_p^* :=  (k_p^*+1)\mr+p \in \llbracket r+1,2r\rrbracket$ and $Q_{p} = ((k_p^*+1)!)^{-1}\ad_{\mathcal S}^{k_p^*+1}(L_{2p}^{(M),q_p^*})$. Note that applying Lemma \ref{Rpoissonbracket}, we have that
$Q_{p}  \in \mathscr{H}_{q_p^*}^M$
and that proceeding as we did at the previous step (i.e. the estimates are the same) we have that
$$
\| Q_{p} \|_{\lof} \le \sfc^{4q_p^*-9}(q_p^*)^{2(q_p^*-2)}(\mr+1)^{4(q_p^*-3)} \le \sfc^{8r - 9} 2^{4r - 4} r^{10r - 16}.
$$
and
$$
\mathfrak{m}_{Q_{p} } \leq 3q_p^*-6, \quad \mathfrak{n}_{Q_{p} } \leq 2q_p^*-6, \quad \mathfrak{h}_{Q_{p} } \leq 6\mr. 
$$
Then,  applying the triangular inequality and using that $\|\mathrm{d}\Phi_{\mathcal S}^1(z)(w) \|_{\sis} \le 2 \| w \|_{\sis}$ we get
\begin{equation*}
\begin{split}
\| \X_{\Upsilon^{\sharp}}(z)\|_{\sis} &\le \|\X_{\Upsilon \circ\Phi_{\mathcal S}^1}(z) \|_{\sis} +\sum_{3\le p\le r}  \sup_{0\leq t \leq 1} \| \X_{Q_{p} \circ\Phi_{\mathcal S}^t}(z) \|_{\sis} \\
&\le 2 \|\X_{\Upsilon} \circ\Phi_{\mathcal S}^1(z) \|_{\sis} +2\sum_{3\le p\le r}  \sup_{0\leq t \leq 1} \| \X_{Q_{p}} \circ\Phi_{\mathcal S}^t(z) \|_{\sis}.
\end{split}
\end{equation*}
Using the rough estimate $\| \Phi_{\mathcal S}^t(z)  \|_{\sis} \le 2 \|z\|_{\sis}$, Lemma \ref{Rvectorfield} and  $\|z\|_{\sis}\le 2 \varrho^\sharp <\gamma $ (and that $\nu_{\mr +1}>1/2$), we get
\begin{equation*}
\begin{split}
\| \X_{\Upsilon^{\sharp}}(z)\|_{\sis} &\le 2 \|\X_{\Upsilon} \circ\Phi_{\mathcal S}^1(z) \|_{\sis}  +2\sum_{3\le p\le r}  \sup_{0\leq t \leq 1}  2\frac{\mathfrak m_{Q_p}(1+\mathfrak h_{Q_p})}{(\nu_{\mr +1}\gamma /2)^{\mathfrak n_{Q_p}+1}}||Q_p||_{\lof} || \Phi_{\mathcal S}^t(z) ||_{\sis}^{2q_p^*-1}\\
&\le      2 \|\X_{\Upsilon} \circ\Phi_{\mathcal S}^1(z) \|_{\sis}+   \sfc^{8r - 9} 2^{16r + 2} r^{10r - 14}  \sum_{3\le p\le r} \gamma^{- 2q_p^*+5} ||z ||_{\sis}^{2q_p^*-1} \\
&\le    2 \|\X_{\Upsilon} \circ\Phi_{\mathcal S}^1(z) \|_{\sis}+   \sfc^{9r}  r^{10r-14}   \gamma^{- 2r + 3} ||z ||_{\sis}^{2r+1} ,\quad\underline{~\text{provided }\sfc\ge 2^{16}}.
\end{split}
\end{equation*}
Finally, using the induction hypothesis and the sharper estimate $\| \Phi_{\mathcal S}^1(z)  \|_{\sis} \le (1+2^{-2\mr}) \|z\|_{\sis}$ (proved in \eqref{eq:good2}), we get
\begin{equation*}
\begin{split}
\| \X_{\Upsilon^{\sharp}}(z)\|_{\sis}  &\le    ( \sfc^{9r}  r^{10r-14}   \gamma^{- 2r + 3}) \big(  4^{\mr-1} 2 ||\Phi_{\mathcal S}^1(z) ||_{\sis}^{2r+1} \prod_{i= 1}^{\mr-1}  (1+2^{-2i})^{2r+1} +  ||z ||_{\sis}^{2r+1}\big) \\
 &\le    ( \sfc^{9r}  r^{10r-14}   \gamma^{- 2r + 3} ||z ||_{\sis}^{2r+1}) \big( 4^{\mr-1} 2 \prod_{i= 1}^{\mr}  (1+2^{-2i})^{2r+1} +  1\big) \\
 &\le 4^{\mr} \sfc^{9r}  r^{10r}   \gamma^{- 2r + 3} ||z ||_{\sis}^{2r+1} \prod_{i= 1}^{\mr}  (1+2^{-2i})^{2r+1}.
\end{split}
\end{equation*}
This concludes the proof of the quantitative rational normal form Theorem~\ref{Rnormalform}.
\end{proof}

\section{Proof of the main results}\label{sec:dyn}
In this section we prove the main results: Theorem \ref{thm:main}, Proposition \ref{prop:geom}, Proposition \ref{prop:meas} and Proposition \ref{prop:proba}. As you might expect, the proof of Theorem \ref{thm:main} will require the most effort. It is based on a bootstrap argument on $T_\epsilon$ the maximal existence time of the local solution initiated from an initial data of size $\epsilon$. As usual in a Birkhoff normal form procedure, we obtain that $T_\eps=O(\eps^r)$. Then our precise estimates obtained in the previous section allow us to optimize $r$ in term of $\eps$. Actually for simplicity of presentation, we prefer to give the optimized values of the various parameters as a function of $\eps$, and then check that they are compatible with our bootstrap.\\
So provided that $\varepsilon \leq e^{-1}$, we define the optimal normal form order
\[
r_\varepsilon :=  \lfloor\frac{\min(1,\sigma)\theta(1-\theta)}{500}\frac{\log \varepsilon^{-1}}{\log\log \varepsilon^{-1}} \rfloor
\]
and the optimal truncation parameters
\[
M_\varepsilon = (\log \varepsilon^{-1})^{1+\frac{4}{\theta}}, \quad N_\varepsilon=(\log \varepsilon^{-1})^{\frac{2}{\theta}}, \quad L_\varepsilon:=6r_\eps N_\varepsilon^{2}\,.
\]
Then we define the optimal set of good initial data
\begin{equation}
\label{eq:big_set}
\Theta_\varepsilon := \Pi_{L_\varepsilon}^{-1}\mathfrak{V}_{10\delta_\varepsilon}^{6r_\varepsilon,L_\varepsilon}
\end{equation}
where 
$$
\delta_{\epsilon} = \epsilon^{2}\gamma_\varepsilon\,,\quad \gamma_{\epsilon}=\epsilon^{1/2}
$$
and
\[
\mathfrak{V}_\delta^{r,M}:=\{ z \in \mathcal{G}_M \ |\ \underset{\bj\in \mathcal{N}^{r,M}}{\min}  \ |\omega_{\bj}^M(z)|> \delta\}\ .
\]
We notice that this new non resonant set $\mathfrak{V}_\delta^{r,M}$ is related to the already defined set $\mathfrak{U}_{\gamma}^{r,M}$ (see \eqref{eq:U})  as follow: if $\delta = \eta^{2}\gamma$, then if $z\in\mathbb{B}(\eta)$, we have  
\[
\Pi_{M}z\in\mathfrak{V}_\delta^{r,M}\implies \Pi_{M}z\in\mathfrak{U}_{\gamma}^{r,M}\,.
\]

\subsection{Geometric aspects : proof of Proposition \ref{prop:geom}} First, we note that by its definition \eqref{eq:U} the set $\mathfrak{V}_\gamma^{r,M}$ is open and invariant by translation of the angles. 
To prove that it is a right cylinder of direction  $(\mathrm{Id}_\mathcal{G} - \Pi_{M_\varepsilon}) \mathcal{G} $, it is enough to note that by definition $\Theta_\varepsilon$ is a right cylinder of direction  $(\mathrm{Id}_\mathcal{G} - \Pi_{L_\varepsilon}) \mathcal{G} $ provided $\varepsilon$ is small enough so that
$M_\varepsilon = (\log \varepsilon^{-1})^{1+\frac{4}{\theta}} \geq L_\varepsilon$ and thus $(\mathrm{Id}_\mathcal{G} - \Pi_{M_\varepsilon}) \mathcal{G} \subset (\mathrm{Id}_\mathcal{G} - \Pi_{L_\varepsilon}) \mathcal{G} $.

\subsection{Dynamics : proof of Theorem \ref{thm:main}} We are going to prove that provided $\varepsilon$ is smaller than a constant depending only $(\sigma,\theta)$
then for all $z^{(0)}\in \Theta_{\epsilon}$ {of size $\| z^{(0)} \|_{\sigma} \leq \varepsilon$}
the local solution to~\eqref{sp} initiated from $z^{(0)}$ exists in $C([-T_{\epsilon},T_{\epsilon}],\mathcal{G})$, where  $T_{\epsilon}:= \varepsilon^{-\frac{r_\eps}3}$, and satisfies, for all $|t|\leq T_{\epsilon}$,
 \begin{equation}
\label{eq:dyn}
\|z(t)\|_{\sigma}\leq2 \| z(0) \|_{\sis}\,,\quad 2\sum_{a\in\mathbb{Z}}e^{\sigma|a|^{\theta}}|I_{a}(z(t))-I_{a}(z(0))|^{\frac{1}{2}}\leq\| z(0) \|_{\sis}^{\frac{3}{2}}\,.
\end{equation}
We only prove the result for the forward evolution $t\geq0$. So let $z^{(0)}\in\Theta_{\epsilon}$ and let $z$ be the corresponding local solution\footnote{Note that the existence and uniqueness of such a solution is a consequence of a local Cauchy theory for \eqref{sp} in $\mathcal{G}$ that we do not specify here. However, it is just a basic corollary of the multilinear estimates given by Lemma \ref{vectorfield}.}, defined on a maximal lifespan interval $[0,T_{\ast}]$ with $T_{\ast}>0$.

\medskip

\noindent $\bullet$\ {\bf Step 0: Preliminaries and statement of the bootstrap.} From now on, to avoid overloading the notations, we'll omit the index $\eps$ for the various parameters introduced above. Thus, $r=r_\eps$, $N=N_\eps$, $M=M_\eps$ and so on.\\
 First, for later use, we note that by assumption the parameters satisfy 
\begin{equation}\label{eq:parameters2}
e^{-\sigma(1-\theta)N^{\theta}}\leq \epsilon^{\frac{500}{\theta}r\log\log\epsilon^{-1}}\,,\quad M^{r}\leq(\log\epsilon^{-1})^{r}\epsilon^{-\frac{1}{125}\min(1,\sigma)(1-\theta)} \,,
\end{equation}
and in particular, when $\epsilon$ is small enough (depending on $\sigma, \theta$)
\begin{equation}
\label{eq:parameters3}
r^{r}\leq\epsilon^{-\frac{1}{250}}\,,\quad e^{-\sigma(1-\theta)N^{\theta}}\leq \epsilon^{500r}\,,\quad M^{r}\leq\epsilon^{-\frac{1}{108}}\,,\quad e^{6\sigma r}\leq \epsilon^{-\frac{1}{500}}\, .
\end{equation}
\begin{remark}\label{rem:resp}
We stress out that this previous estimates are crucial: they show that factors of order $r^r$ (that appears in the estimates of Theorem~\ref{Bnormalform} and Theorem~\ref{Rnormalform}) and even $M^r$ (that appears in the control of  the measure of the set of "good initial data" see Proposition~\ref{prop:measbis}) are actually not so big since they are controlled by a small power of $\eps^{-1}$. We also note that this fact will not be true with $r_\eps=c\log \eps^{-1}$, i.e. we really need to choose an optimal order smaller. It turns out that $r_\eps=c\frac{\log \eps^{-1}}{\log\log \eps^{-1}}$ is small enough and in particular we don't need to take $r_\eps=c(\log \eps^{-1})^\beta$ with $0<\beta<1$ as it is done in \cite{FG13,BMP20,LX23}).
\end{remark}

 Then, in order to apply the normal form Theorem~\ref{Bnormalform} (resp. Theorem~\ref{Rnormalform})  we shall make sure that $10\epsilon\leq\rho(r_{\epsilon})$ (resp. $10\epsilon\leq\varrho(r_{\epsilon})$), where $\rho$ (resp. $\varrho$) is defined in~\eqref{eq:rho} (resp.~\eqref{eq:varrho}). The latter constraint is stronger. Observe that 
$$
\varrho(r_{\epsilon}):=\gamma^{\frac{3}{2}}\sfc^{-1}r_{\epsilon}^{-4} \mathop{\gg}_{\varepsilon \to 0} \varepsilon^{\frac1{50}} \varepsilon^{\frac{3}{4}}  \mathop{\gg}_{\varepsilon \to 0} \varepsilon \,, 
$$
in which case, when $\epsilon$ is small enough we have 
\[
\mathbb{B}(0,10\epsilon)\subset \mathbb{B}(0,\varrho)\subset \mathbb{B}(0,\rho)\,. 
\]
Therefore, as we will see, under our choice of parameters the solution, in resonant and in rational coordinates, will always be in the set where Theorem~\ref{Bnormalform} and Theorem~\ref{Rnormalform} can be applied, with good estimates on the change of coordinates.

Let us now prove that $T_{\ast}\geq T_{\epsilon}$ (defined in \eqref{Teps}), together with the estimates~\eqref{eq:dyn}. We argue by contraction: set
\begin{equation}
\label{eq:boot-ass}
T := \sup\ \Big\{t\in[0,T_{\ast}]\ |\ \forall\,\tau\in[0,t],\,2\sum_{a\in\mathbb Z}e^{\sigma|a|^{\theta}}|I_{a}(z(\tau))-I_{a}(z(0))|^{\frac{1}{2}}\leq \|z(0)\|_{\sigma}^{\frac{3}{2}}\Big\}\,,
\end{equation}
 and assume that $T<T_{\epsilon}$. By a continuity argument, we have
\begin{equation}
\label{eq:contradiction}
2\sum_{a\in\mathbb Z}e^{\sigma|a|^{\theta}}|I_{a}(z(T))-I_{a}(z(0))|^{\frac{1}{2}}=\|z(0)\|_{\sigma}^{\frac{3}{2}}\,.
\end{equation}
Let us contradict~\eqref{eq:contradiction}. 

\medskip

\noindent $\bullet$\ {\bf Step 1: Consequences of the bootstrap assumption.}  It follows from the bootstrap assumption~\eqref{eq:boot-ass} that for all $t\in[0,T]$,
\begin{align}
\label{eq:boot-z}
\|z(t)\|_{\sigma}\leq \|z(0)\|_{\sigma}+\|z(0)\|_{\sigma}^{\frac{3}{2}}&\leq 2\|z(0)\|_{\sigma}\,,\\
\label{eq:boot-act}
2\sum_{a\in\mathbb Z}e^{2\sigma|a|^{\theta}}|I_{a}(z(t))-I_{a}(z(0))|&\leq\|z(0)\|_{\sigma}^{3}\,.
\end{align}
To obtain~\eqref{eq:boot-z} we use that for all $x,y\geq0$,
\begin{equation}
\label{eq:sqrt}
|x^{\frac{1}{2}}-y^{\frac{1}{2}}|\leq |x-y|^{\frac{1}{2}}\,,
\end{equation}
from which we deduce~\eqref{eq:boot-z}:
\begin{multline*}
\|z(t)\|_{\sigma}=2\sum_{a\in\mathbb Z}e^{\sigma|a|^{\theta}}|I_{a}(z(t))|^{\frac{1}{2}}\leq \|z(0)\|_{\sigma} +2\sum_{a\in\mathbb Z}e^{\sigma|a|^{\theta}}|I_{a}(z(t))^{\frac{1}{2}}-I_{a}(z(0))^{\frac{1}{2}}| \\
\leq\|z(0)\|_{\sigma} + 2\sum_{a\in\mathbb Z}e^{\sigma|a|^{\theta}}|I_{a}(z(t))-I_{a}(z(0))|^{\frac{1}{2}}\leq \|z(0)\|_{\sigma}+\|z(0)\|_{\sigma}^{\frac{3}{2}}\,.
\end{multline*}
The estimate~\eqref{eq:boot-act} follows from the embedding $\ell^{2}(\mathbb Z)\subset\ell^{1}(\mathbb Z)$:
\[
2\sum_{a\in\mathbb Z}e^{2\sigma|a|^{\theta}}|I_{a}(z(t))-I_{a}(z(0))|\leq 2\Big(\sum_{a\in\mathbb Z}e^{\sigma|a|^{\theta}}|I_{a}(z(t))-I_{a}(z(0))|^{\frac{1}{2}}\Big)^{2}\leq \|z(0)\|^{3}\,.
\]
In particular we have from \eqref{eq:boot-z} that $z(t)\in\mathbb{B}(0,2\epsilon)$.  
Moreover, we deduce from \eqref{eq:boot-act} and from the stability Lemma \ref{lipschitzomega} that $\Pi_{L}z(t)\in \mathfrak{V}_{3\delta}^{6r,L}$: given $\bj\in\mathcal{N}^{6r,L}$ we have 
\begin{align*}
|\omega_{\bj}(\Pi_{L}z(t))| &\geq  |\omega_{\bj}(\Pi_{L}z(0))| - |\omega_{\bj}(\Pi_{L}z(t))-\omega{\bj}(\Pi_{L}z(0))| \\
	&\geq 4\delta - 6r \sum_{|a|\leq L}|I_{a}(z(t))-I_{a}(z(0))| \\
	&\geq 4\delta - 6r\epsilon^{3}\geq 3\delta \ \text{ (use \eqref{eq:parameters3})}\,.
\end{align*}
Then, we apply the resonant normal form Theorem~\ref{Bnormalform} to $H$ the original Hamiltonian on $\mathbb{B}(10\epsilon)$, and denote $\phi_{(0)},\phi_{(1)}$ the change of coordinates (as in the statement of Theorem~\ref{Bnormalform}): 
\[
H\circ\phi_{(1)}=\underbrace{L_{2}+L_{4}+\sum_{m=3}^{r}K_{2m}}_{:=H_{\mathrm{res}}}+R\,,
\]
where each $K_{2m}$ is resonant and $R$ is the remainder. It follows from~\eqref{eq:boot-z} and~\eqref{eq:Bnf-close} that the local solution written in resonant coordinates satisfies, for $\eps$ small enough,
\begin{equation}
\label{eq:boot-u2}
u(t):=\phi_{(0)}(z(t))\in\mathbb{B}(0,3\epsilon)\,,
\end{equation} 
and that 
\begin{equation}
\label{eq:boot-u}
\|u(t)\|_{\sigma}\leq 3\|z(0)\|_{\sigma}\,.
\end{equation}
It is the solution to the Cauchy problem\footnote{This identity is classical and relies on the fact that the change of variable $\phi_{(0)}$ is symplectic. Nevertheless, to be proven rigorously, the solution should be approximated  by a smoother one in order to have the extra property that $z \in C^1([0,T]; \mathcal{G})$. This process is quite classical (and heavy) so we omit it (we refer for example to \cite{BG22} for a detailed proof). The key point is that if $\|\partial_x^2 z(0) \|_{\sigma}$ is finite then while $\|z(t) \|_{\sigma}$ is finite, we can prove (using tame estimates and the Gr\"onwall inequality) that $\|\partial_x^2 z(t) \|_{\sigma}$ is also finite. This is nothing but the classical property of the preservation of regularity for dispersive equations in the context of Gevrey spaces (the proof is the same).}
\[
i\partial_{t}u = \nabla H_{\mathrm{res}}(u)+\nabla R(u)\,,\quad u(0)=\phi_{(0)}(z(0))\,.
\]
Moreover, the stability Lemma~\ref{lipschitzomega} (more precisely Remark \ref{rem:lip}) combined with \eqref{eq:Bnf-close} and the fact that $\Pi_{L}z(t)\in \mathfrak{V}_{3\delta}^{6r,L}$ imply, again for $\eps$ small enough, that 
\[
\Pi_{L}u(t)\in\mathfrak{V}_{2\delta}^{6r,L}\,.
\]
We now split $u$ into its \emph{high} and its \emph{low} modes parts:
\[
u:= u_{\leq L}+u_{>L}\,,\quad u_{\leq L}:=\Pi_{L}u\,.
\]  
We control these two parts by performing a double bootstrap argument. The high-frequency (or high-mode) part is handled by the estimates obtained in Section~\ref{sec:high}. As for the low frequency (or low mode) part, we do a rational normal form as in Theorem \ref{Rnormalform} on $\mathbb{B}(0,10\epsilon)\cap\mathfrak{U}_{\gamma}^{6r,M}$, with $\gamma=\epsilon^{\frac{1}{2}}$, for the Hamiltonian
\[
H_{\mathrm{res}}^{(L)}:= L_{2}^{(L)}+L_{4}^{(L)}+\sum_{m=3}^{r}K_{2m}^{(L)}\,.\]
 Thus we obtain $\varphi_{(0)}$ and $\varphi_{(1)}$  such that 
\[
H_{\mathrm{res}}^{(L)}\circ\varphi_{(1)}= L_{2}^{(L)}+L_{4}^{(L)}+\sum_{m=3}^{r}L_{2m}^{(L)}+\Upsilon\,,
\]
where $L_{2m}^{(L)}$ is integrable and homogeneous of degree $2m$ and $\Upsilon$ is a remainder of order $2r+1$.  Set
\[
v(t):=\varphi_{(0)}(\Pi_Lu)\in \mathbb{B}(0,20\epsilon)\cap\mathfrak{U}_{\frac{\gamma}{2}}^{6r,L}\,.
\]
It follows from~\eqref{eq:phi-id} and from the a priori bound~\eqref{eq:boot-u} that we actually have for all $t\in[0,T]$,
\begin{equation}
\label{eq:boot-v0}
\|v(t)\|_{\sigma}\leq 4\|z(0)\|_{\sigma}\,.
\end{equation}
Moreover, by freezing the high frequencies of $u$ we see that $v$ solves the finite dimensional non-autonomous ODE
\begin{align}
\nonumber
i\partial_{t}v &= \nabla (H_{\mathrm{res}}^{(L)}\circ\varphi_{(1)})(v) + \mathrm{d}\varphi_{(0)}(u_{\leq L})\Big(\Pi_{L}\X_{H_{\mathrm{res}}^{(>L)}}(u)\Big) + \mathrm{d}\varphi_{(0)}(u_{\leq L})\Big(\Pi_{L}\X_{R}(u)\Big)\\
\label{eq:v}
&=: \nabla (H_{\mathrm{res}}^{(L)}\circ\varphi_{(1)})(v)+f(t)\,.
\end{align}
\\
In particular, we have from~\eqref{eq:Rnf-d} that, since $u\in \mathbb{B}(0,3\epsilon)$ under the bootstrap assumption,  
\begin{align*}
\|f(t)\|_{\sigma} &\leq \|\mathrm{d}\varphi_{(0)}(u_{\leq L})\|_{\sigma\to\sigma}\Big(\|\Pi_{L}\X_{H_{\mathrm{res}}^{(>L)}}(u)\|_{\sigma}+\|\Pi_{L}\X_{R}(u)\|_{\sigma}\Big)\\
&\leq 2^{r-2}\Big(\|\Pi_{L}\X_{H_{\mathrm{res}}^{(>L)}}(u)\|_{\sigma}+\|\Pi_{L}\X_{R}(u)\|_{\sigma}\Big)\\
&\leq 2^{r-2}\Big(2r^{2}\sfc^{2r-3}r^{2(r-2)}e^{-\sigma N^{\theta}}\|u(t)\|_{\sigma}^{5}+\sfc^{4r-1}r^{2r}\|u(t)\|_{\sigma}^{2r+1}\Big)\,,
\end{align*}
where the last inequality follows from Lemma~\ref{lem:mis} and \eqref{eq:b-bnf}, \eqref{XR1}. In particular, we deduce from \eqref{eq:boot-u} that under our choices of parameters \eqref{eq:parameters3} we have for all $t\in(0,T)$,
\begin{equation}
\label{eq:f}
\|f(t)\|_{\sigma}\leq \epsilon^{2r-2}\|z(0)\|_{\sigma}^{2}\,.
\end{equation}
We are now ready to perform the double bootstrap argument. 

\medskip

\noindent
$\bullet$ {\bf Step 2: Bootstrap estimates.} Let us deduce some a priori bounds from the bootstrap assumption~\eqref{eq:boot-ass}: if $T<T_{\ast}$ then for all $t\in(0,T]$ and $\epsilon\in(0,\epsilon_{\ast})$,
\begin{align}
\label{eq:boot-uhi}
2\sum_{|a|>L}e^{\sigma|a|^{\theta}}|I_{a}(u(t))-I_{a}(u(0))|^{\frac{1}{2}}\leq\epsilon^{\frac r6}\|z(0)\|_{\sigma}^{\frac{3}{2}}\,,\\
\label{eq:boot-v}
2\sum_{|a|\leq L}e^{\sigma|a|^{\theta}}|I_{a}(v(t))-I_{a}(v(0))|^{\frac{1}{2}}\leq\epsilon^{\frac{r}{20}}\|z(0)\|_{\sigma}^{\frac{3}{2}}\,.
\end{align}
These bounds provide a control on the long-time evolution of the high (resp. low) frequency components of the solution in the resonant (resp. rational) coordinates. 

\medskip

\noindent
$-$ {\it Bootstrap for the high-frequencies:} We show~\eqref{eq:boot-uhi}. Recall that under the bootstrap assumption, $u\in B(0,10\epsilon)$ and that
\[
H\circ\phi_{(1)}(u) = L_{2}(u)+L_{4}(u)+\sum_{m=3}^{r}K_{2m}(u)+R(u)\,,
\]
as in~\eqref{decomp1}, where each $K_{2m}\in\mathcal{H}_{m}^{\mathcal{R}}$ is a resonant homogeneous polynomial of degree $2m$, bounded by~\eqref{eq:b-bnf}. For $|a|>L$, we have
\[
\frac{d}{dt}I_{a}(u(t)) = \sum_{m=3}^{r}\{I_{a},K_{2m}\}(u(t)) + \{I_{a},R\}(u(t))\,.
\]

It follows from~\eqref{XR1} and from Cauchy-Schwarz inequality that
\begin{align*}
\sum_{|a|>M}e^{\sigma|a|^{\theta}}|\{I_{a},R\}(u(t))|^{\frac{1}{2}}&\leq 2\sum_{|a|>L}e^{\sigma|a|^{\theta}}|u_{a}(s)|^{\frac{1}{2}}|(\X_{R}(u))_{a}|^{\frac{1}{2}} \\
&\leq 2\|u_{>L}(s)\|_{\sigma}^{\frac{1}{2}}\|\X_{R}(u(s))\|_{\sigma}^{\frac{1}{2}} \leq 2\sfc^{2r}r^{r}\|u(t)\|_{\sigma}^{r+1}\,.
\end{align*}
Moreover, we have from Proposition~\ref{prop:highmodes} and form the bounds~\eqref{eq:b-bnf} that 
\[
\sum_{|a|>L}\sum_{m=3}^{r}e^{\sigma|a|^{\theta}}|\{I_{a},K_{2m}\}(u(t))|^{\frac{1}{2}} \leq 2r \sfc^{r-2}r^{r-2}e^{-\frac{1}{2}\sigma(1-\theta)N^{\theta}}\|u(t)\|_{\sigma}^{3}\,.
\]
We deduce from the bootstrap assumption and from \eqref{eq:parameters3} that for all $t\in[0,T]$, 
\begin{align*}
\sum_{|a|>L}e^{\sigma|a|^{\theta}}|\frac{d}{dt}I_{a}(u(t))|^{\frac{1}{2}}&\leq 2\sfc^{2r}r^{r}\|u(t)\|_{\sigma}^{r-\frac{1}{2}}\|u(t)\|_{\sigma}^{\frac{3}{2}}+2r\sfc^{r-2}r^{r-2}e^{-\frac{1}{2}\sigma(1-\theta)N^{\theta}}\|u(t)\|_{\sigma}^{3}\\
&\leq (2\sfc)^{2r}\epsilon^{-\frac{1}{50}+r-\frac{1}{2}}\|u(t)\|_{\sigma}^{\frac{3}{2}}+2r\sfc^{r-2}\epsilon^{-\frac{1}{50}+100r}2^{3}\|u(t)\|_{\sigma}^{3}\\
&\leq \epsilon^{\frac{r}{2}}\|z(0)\|_{\sigma}^{\frac{3}{2}}\,,
\end{align*}
assuming $\epsilon_{\ast}$ is sufficiently small with respect to the universal constant $\sfc$. Therefore, since we assumed that $T\leq T_{\epsilon}$, we obtain
\[
\sum_{|a|>L}e^{\sigma|a|^{\theta}}|I_{a}(u(t))-I_{a}(u(0))|^{\frac{1}{2}}\leq T\underset{s\in[0,T]}{\sup}\ \sum_{|a|>L}e^{\sigma|a|^{\theta}}|\frac{d}{dt}I_{a}(u(s))|^{\frac{1}{2}}\leq T_{\epsilon}\epsilon^{\frac{r}{2}}\|z(0)\|_{\sigma}^{\frac{3}{2}}\leq \epsilon^{\frac{r}{6}}\|z(0)\|_{\sigma}^{\frac{3}{2}}\,.
\]
This proves \eqref{eq:boot-uhi}.

\medskip

\noindent$-$ {\it Bootstrap for the low frequencies:} It remains to prove~\eqref{eq:boot-v}. Our estimate rely on the fact that $H_{\mathrm{res}}^{(L)}\circ\varphi_{(1)}$ is essentially integrable (up to the remainder $\Upsilon$). In light of~\eqref{eq:v}, for fixed $|a|\leq L$ we have 
\[
\frac{d}{dt}I_{a}(v(t))= \{ I_a, H_{\mathrm{res}}^{(L)}\circ\varphi_{(1)} \}(v(t)) + 2\im\Big(\overline{v}_{a}f_{a}(t)\Big)=2\im\Big(\overline{v_{a}}\partial_{\overline{v}_{a}}\Upsilon\Big)+2\im\Big(\overline{v}_{a}f_{a}(t)\Big)\,,
\]
so that, for $t\in(0,T)$,
\begin{equation}\label{ouf}
\sum_{a}e^{\sigma|a|^{\theta}}|\frac{d}{dt}I_{a}(v(s))|^{\frac{1}{2}}\leq 2\|v\|_{\sigma}^{\frac{1}{2}}\big(\|\X_{\Upsilon}(v(t))\|_{\sigma}^{\frac{1}{2}} + \|f(t)\|_{\sigma}^{\frac{1}{2}}\big)\,.
\end{equation}
We deduce from the bootstrap assumption that $v\in\mathfrak{U}_{\frac\gamma2}^{6r,L}\cap\mathbb{B}(0,4\epsilon)$ and $\mathbb{B}(0,4\epsilon)\subset \mathbb{B}(0,2\varrho)$. According to Theorem~\ref{Rnormalform} and to \eqref{eq:boot-v0} we obtain
\[
\|\X_{\Upsilon}(v(t))\|_{\sigma}\leq \sfc^{r}r^{10r}\gamma^{-2r+3}(4\|z(0)\|)^{2r+1}\leq\epsilon^{\frac{9r}{10}-2}\|z(0)\|_{\sigma}^{2}\,.
\] 

 Inserting this last estimate and \eqref{eq:f} in \eqref{ouf} we deduce that  for $t\leq T_\eps=\eps^{-r/3}$,
\begin{multline*}
\sum_{|a|\leq L}e^{\sigma|a|^{\theta}}|I_{a}(v(t))-I_{a}(v(0))|^{\frac{1}{2}}
	\leq t \underset{s\in[0,]}{\sup}\ \sum_{|a|\leq L}e^{\sigma|a|^{\theta}}|\frac{d}{dt}I_{a}(v(s))|^{\frac{1}{2}}\\
	\leq T_{\epsilon}(\epsilon^{\frac{9r}{20}}+\epsilon^{r})\epsilon^{-2}\|z(0)\|_{\sigma}^{\frac{3}{2}}
	\leq (\epsilon^{\frac{7r}{60}}+\epsilon^{\frac{2r}{3}})\epsilon^{-2}\|z(0)\|_{\sigma}^{\frac{3}{2}}\leq \frac{1}{2}\epsilon^{\frac{r}{10}-2}\|z(0)\|_{\sigma}^{\frac{3}{2}}\,.
\end{multline*}
This concludes the proof of~\eqref{eq:boot-v}. 
%
%
%

\medskip

\noindent
$\bullet$ {\bf Step 3: Raising the contradiction.} 
It follows from the triangle inequality and the from fact that for all $x,y\geq0$,
\[
(x+y)^{\frac{1}{2}}\leq x^{\frac{1}{2}}+y^{\frac{1}{2}}\,,
\]
that we have
\begin{multline*}
\sum_{a\in\mathbb Z}e^{\sigma|a|^{\theta}}|I_{a}(z(t))-I_{a}(z(0))|^{\frac{1}{2}}\leq \sum_{|a|\leq L}e^{\sigma|a|^{\theta}}\Big[|I_{a}(z(t))-I_{a}(u(t))|^{\frac{1}{2}}
\\
+|I_{a}(u(t))-I_{a}(v(t))|^{\frac{1}{2}}
+|I_{a}(v(t))-I_{a}(v(0))|^{\frac{1}{2}}
+|I_{a}(v(0))-I_{a}(u(0))|^{\frac{1}{2}}+|I_{a}(u(0))-I_{a}(z(0))|^{\frac{1}{2}}\Big] \\
+\sum_{|a|>L}e^{\sigma|a|^{\theta}}\Big[|I_{a}(z(t))-I_{a}(u(t))|^{\frac{1}{2}}+|I_{a}(u(t))-I_{a}(u(0))|^{\frac{1}{2}}+|I_{a}(u(0))-I_{a}(z(0))|^{\frac{1}{2}}\Big]\,.
\end{multline*}
The bootstrap estimate~\eqref{eq:boot-uhi} (resp.~\eqref{eq:boot-v}) provides a control  of $|I_{a}(u(t))-I_{a}(u(0))|$ (resp. $|I_{a}(v(t))-I_{a}(v(0))|$) when $|a|\leq L$ (resp. $|a|>L$),  for $0\leq t\leq T$. Therefore
\begin{multline*}
\sum_{a\in\mathbb Z}e^{\sigma|a|^{\theta}}|I_{a}(z(t))-I_{a}(z(0))|^{\frac{1}{2}}\leq  \sum_{|a|\leq L}e^{\sigma|a|^{\theta}}\Big[|I_{a}(u(t))-I_{a}(v(t))|^{\frac{1}{2}}+|I_{a}(u(0))-I_{a}(v(0))|^{\frac{1}{2}}\Big]\\
	(\epsilon^{\frac{r}{6}}+\epsilon^{\frac{r}{20}})\|z(0)\|_{\sigma}^{\frac{3}{2}}+ \sum_{a\in\mathbb Z}e^{\sigma|a|^{\theta}}\Big[|I_{a}(z(t))-I_{a}(u(t))|^{\frac{1}{2}}+|I_{a}(z(0))-I_{a}(u(0))|^{\frac{1}{2}}\Big]\,.
\end{multline*}
To control the remaining terms, observe that for functions $w, w'\in\mathcal{G}$,
\[
\sum_{a}e^{\sigma|a|^{\theta}}|I_{a}(w)-I_{a}(w')|^{\frac{1}{2}}\leq (\|w\|_{\sigma}+\|w'\|_{\sigma})^{\frac{1}{2}}\|w-w'\|_{\sigma}^{\frac{1}{2}}\,.
\]
Then, we use the estimates~\eqref{eq:Bnf-close} and~\eqref{eq:phi-id} on the different changes of coordinates to deduce that for all $t\in(0,T)$, 
\begin{align*}
 \sum_{|a|\leq L}e^{\sigma|a|^{\theta}}|I_{a}(u(t))-I_{a}(v(t))|^{\frac{1}{2}}
 	&\leq\underset{t}{\sup} (\|u(t)\|_{\sigma}+\|v(t)\|_{\sigma})^{\frac{1}{2}}\|u(t)-v(t)\|^{\frac{1}{2}} \leq \sfc^{\frac12} (4\|z(0)\|_{\sigma})^{2}\gamma^{-1}\,,\\
\sum_{a\in\mathbb{Z}}e^{\sigma|a|^{\theta}}|I_{a}(u(t))-I_{a}(z(t))|^{\frac{1}{2}}
 	&\leq\underset{t}{\sup} (\|u(t)\|_{\sigma}+\|z(t)\|_{\sigma})^{\frac{1}{2}}\|u(t)-z(t)\|^{\frac{1}{2}} \leq \sfc^{\frac12} (3\|z(0)\|_{\sigma})^{2}\,.
\end{align*}
Therefore, again for $\eps$ small enough (depending only on $\sigma,\theta$), we conclude 
\[
2\sum_{a\in\mathbb Z}e^{\sigma|a|^{\theta}}|I_{a}(z(t))-I_{a}(z(0))|^{\frac{1}{2}}\leq \Big(2\epsilon^{\frac{r}{6}}+2\epsilon^{\frac{r}{20}}+\sfc^{\frac12}4^{2}(\epsilon\gamma^{-1})^{\frac{1}{2}}+\sfc^{\frac12}3^{2}\epsilon^{\frac{1}{2}}\Big)\|z(0)\|_{\sigma}^{\frac{3}{2}}< \|z(0)\|_{\sigma}^{\frac{3}{2}}\,.
\]
This contradicts~\eqref{eq:contradiction}, and thus $T_*\geq T_\eps$. We conclude that \eqref{eq:dyn} holds true.

\subsection{Measure estimates: proof of Proposition \ref{prop:meas}}
In this section we prove the measure estimate of the non-resonant set of initial data, that is
\begin{equation}
\label{eq:meas2}
\mathrm{meas}(\Theta_\varepsilon \cap \mathbb{B}_{M_\varepsilon}(0,\varepsilon))\geq(1-\varepsilon^{\frac{1}{6}})\mathrm{meas}(\mathbb{B}_{M_\varepsilon}(0,\varepsilon))
\end{equation}
where $M_\varepsilon = (\log \varepsilon^{-1})^{1+\frac{4}{\theta}}$. By definition of $\Theta_\varepsilon$ (see \eqref{eq:big_set}),   it will be a direct corollary of the following proposition.  
\begin{proposition}[Measure estimate for non-resonant set] \label{prop:measbis} For $\epsilon>0$, $\kappa\in(0,1)$, $r\geq 3$, $M\geq L \geq 3r$, if
\begin{equation}
\label{eq:coucoula}
\delta \leq \kappa \varepsilon^2 (64M^2)^{-1}   (9r)^{-2r}L^{-3r} e^{-3 \sigma r } \,
\end{equation}
 then, we have
\[
\mathrm{meas}(\mathfrak{V}_{\delta}^{r,L} \cap\mathbb{B}_{M}(0,\varepsilon))\geq (1-\kappa)\, \mathrm{meas}(\mathbb{B}_{M}(0,\varepsilon))\,.
\]
\end{proposition}
In order to apply this proposition with $r= 6r_\varepsilon$, $\delta = 2\delta_\varepsilon=2  \varepsilon^{5/2}$, $\kappa = \varepsilon^{1/6}$, $M=M_\varepsilon$, $L=L_\varepsilon$, we just have to check \eqref{eq:coucoula} which, in view of \eqref{eq:parameters3}, is clear provided that $\varepsilon$ is small enough.

Now, just before starting to prove Proposition \ref{prop:measbis}, let us just mention the following property of the modulated frequencies (whose proof is given in Lemma 5.11 in~\cite{KillBill}) :
\begin{lemma}\label{lem:mod-freq} For all $m\geq2$ and $\bj=(\delta_{\alpha},a_{\alpha})_{\alpha=1}^{2m} \in \mathcal{J}_m \setminus \mathrm{Int}$ there exists $a_{\ast}\in(-3m,3m)\setminus\{a_{1},\cdots,a_{2m}\}$ such that 
\[
\Big|\sum_{\alpha=1}^{2m}\frac{\delta_{\alpha}}{(a_{\ast}-a_{\alpha})^{2}}\Big|\geq\frac{1}{(6m)^{4m}\prod_{\alpha=1}^{2m}\langle a_{\alpha}\rangle^{2}}\,.
\]
\end{lemma}

\begin{proof}[\underline{Proof of Proposition \ref{prop:measbis}}]  First, we note that by homogeneity, it is enough to deal with the case $\varepsilon = 1$.
We have to prove that 
\begin{equation}
\label{eq:gravelax}
\mathrm{meas}\Big(\Big\{z\in\mathbb{B}_{M}(0,1)\ |\  \underset{\bj\in\mathcal{N}^{r,L}}{\max}\ |\omega_{\bj}^L(z)|\leq \delta \Big\}\Big)\leq \kappa\ \mathrm{meas}(\mathbb{B}_{M}(0,1))\,.
\end{equation}
Fix $\bj=(\delta_{\alpha},a_{\alpha})_{\alpha=1}^{2m} \in\mathcal{N}^{r,L}$ and set
\[
E_{\bj} = \Big\{z\in\mathbb{B}_{M}(0,1)\,,\ |\  |\omega_{\bj}^L(z)|\leq \delta\Big\}\,.
\] 
Take $a_{\ast}(\bj)$ given by Lemma~\ref{lem:mod-freq}. Note that since $L\geq 3r$, we have that $a_{\ast}(\bj) \in [-L,L]$. We are going to impose some restrictions on $I_{a_{\ast}}(z)$ without moving the variables $(z_{j})_{j\neq a_{\ast}}$. By a slight abuse of notations  we write
\[
\omega_{\bj}^L(z) = \omega_{\bj}^L(I_{a_{\ast}})\,.
\]
Recall that 
\[
\omega_{\bj}^L(I_{a_{\ast}}) = \sum_{\alpha=1}^{2m}\delta_{\alpha}\sum_{a\neq a_{\alpha}}\frac{I_{a}}{2(a-a_{\alpha})^{2}} = \sum_{\alpha=1}^{2m}\frac{\delta_{\alpha}}{2(a_{\alpha}-a_{\ast})^{2}}I_{a_{\ast}} + \sum_{\alpha=1}^{2m}\delta_{\alpha}\sum_{a\neq a_{\alpha}, a_{\ast}}\frac{I_{a}}{2(a-a_{\alpha})^{2}}\,.
\]
Then, we observe that for all $I_{a_{\ast}}, I_{a_{\ast}}'$ we have 
\begin{equation*}
\begin{split}
|\omega_{\bj}^L(I_{a_{\ast}})-\omega_{\bj}^L(I_{a_{\ast}}')| = \Big|\sum_{\alpha=1}^{2m}\frac{\delta_{\alpha}}{2(a_{\alpha}-a_{\ast})^{2}}\Big||I_{a_{\ast}}-I_{a_{\ast}}'| 
\geq \frac{1}{(6m)^{4m}\prod_{\alpha=1}^{2m}\langle a_{\alpha}\rangle^{2}} |I_{a_{\ast}}-I_{a_{\ast}}'| \,.
\end{split}
\end{equation*}
Since $\bj\in\mathcal{N}^{r,L}$ (see \eqref{eq:NrM}) we deduce that
\begin{equation}
\label{eq:lip-omega}
|\omega_{\bj}^L(I_{a_{\ast}})-\omega_{\bj}^L(I_{a_{\ast}}')|\geq (3r)^{-2r}L^{-2r}|I_{a_{\ast}}-I_{a_{\ast}}'|\,.
\end{equation}
Therefore, if $z,z'$ are such that $z_{a}=z_{a}'$ for all $a\neq a_{\ast}$ then 
\begin{equation}
\label{eq:diam}
z,z'\in E_{\bj}\quad \implies\quad |I_{a_{\ast}}-I_{a_{\ast}}'| \leq 2\delta (3r)^{2r}L^{2r}\,.
\end{equation}
As a consequence, setting 
$$
\mathcal{G}_{M,a_*} = \{ z \in \mathcal{G}_M \ | \ z_{a_*}=0\} \quad  \mathrm{and} \quad \mathbb{B}_{M,a_*}(0,y) = \mathbb{B}_{M}(0,y) \cap \mathcal{G}_{M,a_*},
$$
we have
\begin{equation}
\label{eq:diam2}
\underset{z' \in \mathcal{G}_{M,a_*} }{\sup}\ \int_{\mathbb{C}}\mathbf{1}_{E_{\bj}}(z' ; z_{a_{\ast}}) dz_{a_{\ast}}\leq 2\pi \delta(3r)^{2r}L^{2r}\,.
\end{equation}
Indeed, for fixed $z'$, writing $z_{a_{\ast}}$ in polar coordinates, $z_{a_{\ast}}=ye^{\mi\phi}$, yields 
\begin{equation*}
\int_{\mathbb{C}}\mathbf{1}_{E_{\bj}}(z' ; z_{a_{\ast}}) dz_{a_{\ast}} = \int_{(0,2\pi)}\int_{(0,1)}\mathbf{1}_{E_{\bj}}(z' ; ye^{\mi\phi} )ydy 
= \pi\int_{(0,1)}\mathbf{1}_{E_{\bj}}(z' ; ye^{\mi\phi}) dy^{2}\,,
\end{equation*}
and~\eqref{eq:diam2} follows from~\eqref{eq:diam}. Now, we apply Fubini's theorem to get
\begin{align*}
\mathrm{meas}(E_{\bj}) &= \int_{\mathbb{B}_{M}(0,1)}\mathbf{1}_{E_{\bj}}(z)dz \\
&\leq \int_{z'\in \mathbb{B}_{M,a_*}(0,1) } \int_{\mathbb{C}}\mathbf{1}_{E_{\bj}}(z' ; z_{a_{\ast}}) dz_{a_{\ast}} d z'  \\
&\leq\mathrm{meas}(\mathbb{B}_{M,a_*}(0,1) )2\pi \delta(3r)^{3r}L^{2r}\,,
\end{align*}
where the last estimate follows from~\eqref{eq:diam2}. Now we note that, by homogeneity,
\begin{multline*}
\mathrm{meas}(\mathbb{B}_{M}(0,1)) = \int_{ \substack{ z' \in  \mathbb{B}_{M,a_*}(0,1), z_{a_*} \in \mathbb{C}  \\  \| z'\|_{\sigma} + 2e^{\sigma |a_*|^\theta} |z_{a_*}| \leq 1} } 1 dz' \,dz_{a_*} = \frac{\pi e^{-2 \sigma |a_*|^\theta}}2  \int_{0}^1  \int_{  z' \in  \mathbb{B}_{M,a_*}(0,1-y)} 1 dz' \, ydy \\
= \mathrm{meas}(\mathbb{B}_{M,a_*}(0,1) ) \frac{\pi e^{-2 \sigma |a_*|^\theta}}2  \int_{0}^1  y(1-y)^{4M} dy =  \mathrm{meas}(\mathbb{B}_{M,a_*}(0,1) )  \frac{ \pi e^{-2 \sigma |a_*|^\theta} }{8M (4M+1)}.
\end{multline*}
Therefore, since $|a_{\ast}|\leq 3r/2$
\[
\mathrm{meas}(E_{\bj})\leq 64M^2  \delta(3r)^{2r}L^{2r} e^{3 \sigma r } \mathrm{meas}(\mathbb{B}_{M}(0,1))
\]
and using that $\#\mathcal{N}^{r,L} \leq 2^{r} (2L+1)^r \leq (9L)^r  $, we conclude that 
\begin{equation*}
\begin{split}
\mathrm{meas}\Big(\Big\{z\in\mathbb{B}_{M}(0,1)\ |\  \underset{\bj\in\mathcal{N}^{r,L}}{\max}\ |\omega_{\bj}^L(z)|&\leq \delta\Big\}\Big)\leq \sum_{\bj\in\mathcal{N}^{r,L}}\mathrm{meas}(E_{\bj}) \\
  &\leq 64M^2   \delta(9r)^{2r}L^{3r} e^{3 \sigma r } \mathrm{meas}(\mathbb{B}_{M}(0,1)) .
 \end{split}
\end{equation*}
The proof of Proposition~\ref{prop:measbis} then follows from assumption \eqref{eq:coucoula}.
\end{proof}

\subsection{Probability estimates: proof of Proposition \ref{prop:proba}} Let $Y$ be the random function in $\mathcal{G}$, whose Fourier coefficients $Y_a$ are independent and uniformly distributed in $(0, \langle a \rangle^{-2} e^{-\sigma |a|^\theta})$, and let $Z^{(0)} = Y / \| Y\|_{\sigma}$ be the projection of $Y$ on the unit sphere of $\mathcal{G}$. We aim at proving that, provided that $\varepsilon_0$ is small enough, we have
$$
\mathbb{P}(\forall 0<\varepsilon \leq \varepsilon_0, \ \varepsilon Z^{(0)} \in \Theta_\varepsilon) \geq 1 - \varepsilon_0^{1/12}.
$$

\noindent \underline{$\bullet$ Step 1 : probabilistic part.} First, we consider $\bj \in \mathcal{N}$, $\gamma >0$ and as previously $a_{\ast}(\bj)$ given by Lemma~\ref{lem:mod-freq}. Then we note that 
$$
\omega_{\bj}^\infty (Y) = c_{a_*} |Y_{a_*}|^2 + R_{a_*} \quad \mathrm{where} \quad |c_{a_*}| \geq (\mu_1(\bj))^{-2\# \bj} (3\# \bj)^{-2 \# \bj}
$$
and $R_{a_*}$ is independent of $Y_{a_*}$. As a consequence, applying Lemma 4.17 of \cite{KtA}, we have
\begin{multline*}
\mathbb{P}( |\omega_{\bj}^\infty (Y)|<\gamma ) = \mathbb{E} \int_{0}^1  \mathbbm{1}_{ |c_{a_*} \langle a_*\rangle^{-4} e^{-2\sigma |a_*|^{\theta}} y^2 + R_{a_*}| <\gamma} \, \mathrm{d}y \leq 4 \sqrt{\gamma} c_{a_*}^{-1/2} \langle a_*\rangle^{2} e^{\sigma |a_*|^{\theta}} \\
\leq 4 \sqrt{\gamma} (\mu_1(\bj))^{\# \bj} (3\# \bj)^{ \# bj+2} e^{3\sigma \# \bj /2 }.
\end{multline*}
It follows that\footnote{These estimates are very standard, for a more detailed proof of them, we refer for example the reader to \cite{KtA}.}
$$
\mathbb{P}( \exists \bj \in \mathcal{N}, \quad |\omega_{\bj}^\infty (Y)|<10\gamma \eta_{\bj}  ) \leq   \sqrt{\gamma}  \sum_{\bj \in \mathcal{N}} 30^{-\#\bj}  (\mu_1(\bj))^{-2\# \bj} \leq \sqrt{\gamma}
$$
where
$$
\eta_{\bj} := \frac1{160} (\mu_1(\bj))^{-6\# \bj} (90\# \bj)^{-2 \# \bj -4} e^{-3\sigma  \# \bj } .
$$
 Then we note that almost surely we have
$$
\| Y\|_{\sigma} \leq 2 \sum_{a \in \mathbb{Z}} \langle a \rangle^{-2} \leq 2(1+\pi) < 10.
$$
Since $\omega_{\bj}^\infty$ is homogeneous, it follows that for all $\varepsilon_0>0$, we have
$$
\mathbb{P}( \forall \bj \in \mathcal{N}, \ |\omega_{\bj}^\infty (Z^{(0)})| \geq \varepsilon_0^{1/6} \eta_{\bj}  )  \geq 1 - \varepsilon_0^{1/12}.
$$
\noindent \underline{$\bullet$ Step 2 : deterministic part.} To conclude this proof it is enough to prove that, provided that $\varepsilon_0$ is small enough, if $z^{(0)} \in \mathcal{G}$ is such that for all $\bj \in \mathcal{N}$,  $|\omega_{\bj}^\infty (z^{(0)})| \geq \varepsilon_0^{1/6} \eta_{\bj} $ and if $\varepsilon \leq \varepsilon_0$ then $\varepsilon z^{(0)} \in \Theta_\varepsilon$.

More concretely, we fix such $\varepsilon$ and $z^{(0)}$ and $\bj \in \mathcal{N}^{6r_\varepsilon, L_\varepsilon}$ and we aim at proving that 
\begin{equation}
\label{eq:what_we_want}
|\omega_{\bj}^{L_\varepsilon}(\varepsilon \Pi_{L_\varepsilon} z^{(0)})|> 2 \delta_{\varepsilon}.
\end{equation}
Applying the triangular inequality and noticing that $\omega_{\bj}^{L_\varepsilon} = \omega_{\bj}^{\infty} \circ \Pi_{L_\varepsilon} $ , we get
$$
|\omega_{\bj}^{L_\varepsilon}(\varepsilon \Pi_{L_\varepsilon} z^{(0)})| \geq |\omega_{\bj}^{\infty}(\varepsilon  z^{(0)})| - |\omega_{\bj}^{\infty}(\varepsilon  z^{(0)}) -\omega_{\bj}^{\infty}(\varepsilon \Pi_{L_\varepsilon} z^{(0)}) | .
$$
On the one hand, since $\bj \in \mathcal{N}^{6r_\varepsilon, L_\varepsilon}$, we have (provided that $\varepsilon_0$ is small enough)
$$
|\omega_{\bj}^{\infty}(\varepsilon  z^{(0)})| \geq \varepsilon_0^{1/6} \varepsilon^2 \eta_{\bj} \geq \frac1{160} \varepsilon^{2+1/6} L_\varepsilon^{-36 r_\varepsilon} (540 r_\varepsilon)^{-12 r_\varepsilon -4} e^{-18\sigma  r_\varepsilon  } \geq 4 \delta_\varepsilon.
$$
where the last estimate is ensured by our choice of the parameters \eqref{eq:parameters3}. On the other hand, applying Lemma \ref{eq:lip}, we have (provided that $\varepsilon_0$ is small enough)
\begin{equation*}
 |\omega_{\bj}^{\infty}(\varepsilon  z^{(0)}) -\omega_{\bj}^{\infty}(\varepsilon \Pi_{L_\varepsilon} z^{(0)}) |  \leq \varepsilon^2 \# \bj \sum_{|a|>L_\varepsilon} |z^{(0)}_a|^2 
 \lesssim \varepsilon^2 6r_\varepsilon e^{-2\sigma |L_\varepsilon|^\theta}  \leq \delta_\varepsilon.
\end{equation*}
Putting together the three last estimates, we get, as expected, \eqref{eq:what_we_want}.
\appendix

\section{Proof of Lemma~\ref{lem:expdecay}}
\label{sec:appendix}
This short appendix is devoted to the proof of Lemma~\ref{lem:expdecay}. We first show preliminary technical Lemmas.

\begin{lemma}\label{elog}
Let $f(x)=\frac{a^x}{\log a}-ex$ on $x\in\mathbb R$ of $a>1$, then one has $f(x)\ge0$.
\end{lemma}
\begin{proof}
We first compute the zero point of the derivative
\[
f'(x)=a^x-e=0\Rightarrow x=\frac1{\log a}.
\]
Obviously $f(x)=0$ at the point, which leads to $f(x)\ge0$.
\end{proof}

\begin{lemma}\label{mnelog}
Let $f(x)=\left(\frac ne\right)^n(\log x)^{-n}-m^nx^{-m}$ on $x\in(1,\infty)$ of $m,n\in\mathbb N^*$, then we have $f(x)\ge0$.
\end{lemma}
\begin{proof}
Thanks to Lemma \ref{elog} one gets
\[
\frac{x^\frac mn}{\log x}\ge e\frac mn.
\]
It follows that for $x>1$
\[
\frac ne(\log x)^{-1}\ge mx^{-\frac mn},
\]
which implies $f(x)\ge0$.
\end{proof}

\begin{lemma}\label{sigmax}
Let $f(x)=1+\theta x-(1+x)^\theta$ on $x\in[0,\infty)$ of $\theta\in(0,1)$, then we have $f(x)\ge0$.
\end{lemma}
\begin{proof}
Thanks to $f(0)=0$ and $f'(x)\ge0$, we directly obtain the result.
\end{proof}
We are now ready to prove Lemma~\ref{lem:expdecay}.
\begin{proof}[Proof of Lemma~\ref{lem:expdecay}]
First, we note that if $\bj\in\mathcal M_m$ satisfies $\mu_3(\bj)>N$, then $\mu_2(\bj_2,\cdots,\bj_{2m}) >N$. As a consequence, we have
\begin{equation}\label{inf0.1}
\inf\limits_{\substack{\bj\in\mathcal M_m\\\mu_3(\bj)>N}} \sum_{\beta=2}^{2m}|j_\beta|^\theta-\Big(\sum_{\beta=2}^{2m}|j_\beta|\Big)^\theta
 \geq \inf\limits_{ \substack{a_2 \geq \cdots \geq a_{2m} \geq 0 \\ a_3> N}  } \sum_{\beta=2}^{2m}a_\beta^\theta-\big(\sum_{\beta=2}^{2m}a_\beta\big)^\theta 
\end{equation}
Then, observe that 
\begin{align*}
\sum_{\beta=2}^{2m}a_\beta^\theta-\big(\sum_{\beta=2}^{2m}a_\beta\big)^\theta\ge a_2^\theta+a_3^\theta-(a_2+a_3)^\theta+ \sum_{\beta=4}^{2m}a_\beta^\theta-\big(\sum_{\beta=4}^{2m}a_\beta\big)^\theta\ge a_2^\theta+a_3^\theta-(a_2+a_3)^\theta.
\end{align*}
It follows that
\[
\eqref{inf0.1}\ge\inf_{\substack{a,b\in\mathbb Z\\ a\ge b\ge N}}\{a^\theta+b^\theta-(a+b)^{\theta}\}.
\]
Thanks to Lemma \ref{sigmax}, letting $q=\frac ba\in(0,1]$, we have
\[
(a+b)^\theta=a^\theta(1+q)^\theta\le a^\theta(1+\theta q)=a^\theta+\theta\frac{b}{a^{1-\theta}}=a^\theta+\theta b^\theta q^{1-\theta}\le a^\theta+\theta b^\theta.
\]
It leads to that
\[
\eqref{inf0.1}\ge\inf_{b\ge N}\{b^\theta-\theta b^\theta\}\ge(1-\theta)N^\theta.
\]

\end{proof}

\end{document}